\documentclass[10pt,reqno,twoside,draft]{amsart}

\usepackage{lineno}
\usepackage{mathrsfs}
\usepackage{enumerate}

\makeatletter%
\hsize\textwidth \vsize\textheight \@colht\textheight
\@colroom\textheight \columnwidth\textwidth \topskip0\p@ \headsep0pt
\makeatother

\theoremstyle{plain}
\newtheorem{theorem}{Theorem}[section]
\newtheorem{lemma}[theorem]{Lemma}
\newtheorem{prop}[theorem]{Proposition}

\theoremstyle{definition}
\newtheorem{rem}[theorem]{Remark}

\newtheorem{ipos}[theorem]{Hypotheses}
\newtheorem{definition}[theorem]{Definition}

\newcommand{\no}{\nonumber}
\newcommand{\hs}[1]{\hskip -#1pt}

\def\po{{p_0}}
\def\timeT{{T_0}}
\def\konst{{\mathbf{a}}}

\newcommand{\Tr}{\mathop{\mathrm{Tr}}}
\newcommand{\II}{\mathscr{I}}
\newcommand{\qq}{\qquad}
\newcommand{\q}{\quad}
\newcommand\ov{\overline}
\newcommand\va{\varphi}
\newcommand\pn{\par\noindent}
\newcommand\med{\medskip}
\newcommand\ds{\displaystyle}
\newcommand{\ve}{\varepsilon}

\newcommand{\Om}{\Omega}

\newcommand\R{{\mathbb R}}
\newcommand\N{{\mathbb N}}

\begin{document}
%\linenumbers
\numberwithin{equation}{section}
\title[On a class of hypoelliptic operators]{On a class of  hypoelliptic operators with unbounded coefficients in $\R^N$}

\author{Balint Farkas}
\address{Technische Universit\"at Darmstadt\newline Fachbereich Mathematik\newline Schlo\ss{}gartenstra\ss{}e 7, D-64289 Darmstadt, Germany}
\email{farkas@mathematik.tu-darmstadt.de}
\urladdr{www.mathematik.tu-darmstadt.de/\~{}farkas/}
\author{Luca Lorenzi}
\address{Dipartimento di Matematica\newline Universit\`a degli Studi di Parma\newline Viale G.P. Usberti 53/A, I-43100 Parma, Italy} \email{luca.lorenzi@unipr.it}
\urladdr{www.unipr.it/\~{}lorluc99/index.html}
%\address{~\newpage\noindent
%{\rm
%Stampato in proprio, a cura degli autori, presso il Dipartimento di Matematica dell'Universit\`a di Parma in V.le G.P. Usberti 53/A, 43100 PARMA.
%Adempiuti gli obblighi ai sensi della Legge n. 106 del 15.04.2004 ``Norme relative al deposito legale dei documenti di interesse culturale
%destinati all'uso pubblico'' (G.U. n. 98 del 27.04.04) e del Regolamento di attuazione emanato con il D.P.R. n. 252 del 3.05.2006 (G.U.
%n. 191 del 18.08.06) entrato in vigore il 2.09.2006 (precedente normativa abrogata: Legge n. 374 del 2.2.1939, modificata in D.L. n. 660 del 31.08.1945).
%}}
%\address{~\vskip 18.5truecm\noindent
%{\rm Esemplare fuori commercio per il deposito legale agli effetti della legge 15.04.2004, n. 106}}
\thanks{Work supported by the M.I.U.R. research projects
Prin 2004 and 2006 ``Kolmogorov equations''.}

\thanks{Work partially supported by the
research  project ``Kolmogorov equations'' of the Ministero
dell'Istruzione, dell'Universit\`a e della Ricerca (M.I.U.R.) and by
the European Community's Human Potential Programme under contract
HPRN-CT-2002-00281 ``Evolution Equations".}
\thanks{The second author wishes to thank the Department of Mathematics at the Technische Universit\"at of
Darmstadt for the warm hospitality during his visit.}
\date{February 28, 2008}
\begin{abstract}
We consider a class of non-trivial perturbations ${\mathscr A}$ of
the degenerate Ornstein-Uhlenbeck operator in $\R^N$. In fact we
perturb both the diffusion  and the drift part of the operator (say
$Q$ and $B$) allowing the diffusion part to be unbounded in $\R^N$.
Assuming that the kernel of the matrix $Q(x)$ is invariant with
respect to $x\in\R^N$ and the Kalman rank condition is satisfied at
any $x\in\R^N$ by the same $m<N$, and developing a revised version
of Bernstein's method we prove that we can associate a semigroup
$\{T(t)\}$ of bounded operators (in the space of bounded and
continuous functions) with the operator ${\mathscr A}$. Moreover, we
provide several uniform estimates for the spatial derivatives of the
semigroup $\{T(t)\}$ both in isotropic and anisotropic spaces of
(H\"older-) continuous functions. Finally, we prove Schauder
estimates for some elliptic and parabolic problems associated with
the operator ${\mathscr A}$.
\end{abstract}
\maketitle

    %%%%%%%%%%%%%%%%%%%%%%%%%%%%%%%%%%%%%%%%%%%%%%%%%%%%%%%%%%%%%%%%%%%%%%%
    %%                                                                   %%
    %%             INTRODUCTION                                          %%
    %%                                                                   %%
    %%%%%%%%%%%%%%%%%%%%%%%%%%%%%%%%%%%%%%%%%%%%%%%%%%%%%%%%%%%%%%%%%%%%%%%

\tableofcontents
\section{Introduction}
In the last decades the interest towards elliptic and parabolic
operators with unbounded coefficients in unbounded domains has grown
considerably due to their applications to stochastic analysis and
mathematical finance.

The literature on uniformly elliptic operators with unbounded
coefficients in $\R^N$ is nowadays rather complete (we refer the
interested reader, e.g., to \cite{bertoldi-lorenzi}). The picture
changes drastically when one considers degenerate elliptic operators
with unbounded coefficients. The prototype of such operators is the
degenerate Ornstein-Uhlenbeck operator defined on smooth functions
by
\begin{equation}
{\mathscr
A}\va(x)=\sum_{i,j=1}^Nq_{ij}D_{ij}\va(x)+\sum_{i,j=1}^Nb_{ij}x_jD_i\va(x),\qq\;\,x\in\R^N,
\label{O-U-deg}
\end{equation}
\pn where $Q=(q_{ij})$ and $B=(b_{ij})$ are suitable square matrices
such that $Q$ is singular and the condition ${\rm det}\,Q_t>0$ is nevertheless
satisfied for any $t>0$. Here,
\begin{equation*}
Q_t=\int_0^te^{sB}Qe^{sB^*}ds,\qq\;\,t>0.
\end{equation*}
\pn The condition ${\rm det}\,Q_t>0$ is equivalent to the well-known
\emph{Kalman rank condition} which requires that
\begin{equation}
{\rm
rank}[Q^{\frac{1}{2}},BQ^{\frac{1}{2}},\ldots,B^{m}Q^{\frac{1}{2}}]=N,
\label{kalman}
\end{equation}
\pn for some $m<N$. In particular, ${\mathscr A}$ is hypoelliptic in
H\"ormander's sense.

A suitable change of the orthonormal basis of $\R^N$ (see Remark
\ref{rem:2}) allows to rewrite the operator ${\mathscr A}$ on smooth
functions $\varphi$ as
\begin{equation}
{\mathscr A}\va(x)=\sum_{i,j=1}^\po\hat
q_{ij}D_{ij}\va(x)+\sum_{i,j=1}^N\hat
b_{ij}x_jD_i\va(x),\qq\;\,x\in\R^N, \label{O-U-deg-1}
\end{equation}
\pn for some positive definite and not singular $\po\times \po$ matrix
$\hat Q=(\hat q_{ij})$ and some $\po\in\{1,\ldots,N-1\}$.

In \cite{L} Lunardi  proves that one can associate a semigroup of
bounded operators $\{T(t)\}$ in $C_b(\R^N)$ (the space of all
bounded and continuous functions) with the operator ${\mathscr A}$
in a natural way, i.e., for any $f\in C_b(\R^N)$, $T(t)f$ is the
value at $t>0$ of the (unique) classical solution to the homogeneous
Cauchy problem
\begin{equation}
\left\{
\begin{array}{lll}
D_tu(t,x)={\mathscr A}u(t,x), & t\in ]0,+\infty[, &x\in\R^N,\\[2mm]
u(0,x)=f(x), && x\in\R^N,
\end{array}
\right. \label{pbhom-intro}
\end{equation}
\pn where by classical solution we mean a function $u$ which  (i) is
once continuously differentiable with respect to the time variables
and twice continuously differentiable with respect to the spatial
variable in $]0,+\infty[\times\R^N$, (ii) is continuous in
$[0,+\infty[\times\R^N$ and bounded in $[0,\timeT]\times\R^N$ for
any $\timeT>0$ and (iii) solves \eqref{pbhom-intro}.

One of the main peculiarities of the Ornstein-Uhlenbeck operator is
that an explicit representation formula for the associated semigroup
is available. This fact allows the author of \cite{L} to prove
uniform estimates for the spatial derivatives of the function
$T(t)f$ when $t$ approaches $0$ and $f$ belongs to various spaces of
(H\"older-) continuous functions. In fact, the behavior of the
spatial derivatives of $T(t)f$ depends on the variable along which
one differentiates. As a byproduct, this shows that the right
(H\"older-) spaces where to study the semigroup $\{T(t)\}$ are not
the usual ones but rather anisotropic spaces modelled on the
degeneracy of the operator ${\mathscr A}$. Denoting, roughly
speaking, by ${\mathscr C}^{\theta}(\R^N)$ these anisotropic spaces,
Lunardi shows that
\begin{equation}
\|T(t)f\|_{\mathscr C^{\theta}(\R^N)}\le
Ct^{-\frac{\theta-\alpha}{2}}\|f\|_{\mathscr C^{\alpha}(\R^N)},
\qq\;\,t\in ]0,1], \label{stima-anis}
\end{equation}
\pn for any $0<\alpha\le\theta$ and some positive constant $C$,
independent of $t$, i.e., what one can expect in the non-degenerate
case when ${\mathscr C}^{\alpha}$ and ${\mathscr C}^{\theta}$ are the usual H\"older spaces,
even for unbounded coefficients; see e.g.,
\cite{bertoldi-lorenzi-0,lunardi-studia}. Estimate
\eqref{stima-anis} represents the key stone to apply an abstract
interpolation argument from \cite{Lu-sem} to prove optimal Schauder
estimates for the solution both to the elliptic equation
\begin{equation}
\lambda u(x)-{\mathscr A}u(x)=h(x),\qq\;\,x\in\R^N,\qq\;\,\lambda>0,
\label{ell-intro}
\end{equation}
\pn and to the non-homogeneous Cauchy problem
\begin{equation}
\left\{
\begin{array}{lll}
D_tu(t,x)={\mathscr A}u(t,x)+g(t,x), & t\in ]0,\timeT[, &x\in\R^N,\\[2mm]
u(0,x)=f(x), && x\in\R^N,
\end{array}
\right. \label{pbnonhom-intro}
\end{equation}
\pn when $f,g,h$ are suitable continuous functions such that
$g(t,\cdot)$, $f$, $h$ have some additional degrees of smoothness.

Recently, the second author, in \cite{Lo-0,Lo}, has extended these
results to some non-trivial perturbations of the Ornstein-Uhlenbeck
operator in \eqref{O-U-deg}. More precisely, in \cite{Lo-0,Lo} the
operator \eqref{O-U-deg-1} has been studied under the assumption
that
\begin{equation*}
\po\ge N/2,\qq\;\,\hat B=\left (
\begin{array}{cc}
\hat B_1 & \hat B_2\\[2mm]
\hat B_3 & \hat B_4
\end{array}
\right ),
\end{equation*}
\pn the $(N-\po)\times \po$ matrix $\hat B_3$ has full rank, and
assuming that the matrix $\hat Q$ depends on $x\in\R^N$ and its
entries are possibly unbounded functions at infinity. These
assumptions imply that the Kalman rank condition \eqref{kalman} is
satisfied at any $x\in\R^N$, with $m=1$.

To prove the crucial estimates \eqref{stima-anis} a different
technique than that in \cite{L} has been applied since in this new situation no explicit
representation formulas for the associated semigroup is available.
More precisely, such estimates have been obtained by developing a
variant of the classical Bernstein method in \cite{B0}.

Recently, the results in \cite{L} have been generalized, both with
analytic and probabilistic methods, in \cite{priola,priola-1,nic} to
non-trivial perturbations of the operator ${\mathscr A}$ in
\eqref{O-U-deg-1} in which an additional unbounded drift term is
added. More specifically, Saintier in \cite{nic} considers the case
when the differential operator is of type $\hat {\mathscr
A}={\mathscr A}+\sum_{j=1}^\po F_jD_j$, with ${\mathscr A}$ being given
by \eqref{O-U-deg-1}, with an even $N$ and $\po =N/2$ and $Q=B=I$.
Here, $F$ is any smooth function with bounded derivatives up to the
third-order. This operator arises e.g., in the study of the motion
of a particle $y$ of mass one subject to a force field depending on
$y$ and its first-order derivative, perturbed by a noise. We refer
the interested reader to \cite{freidlin} for further details.
Applying the same techniques as those in \cite{Lo-0,Lo}, Saintier proves
optimal Schauder estimates for both the solutions to
\eqref{ell-intro} and \eqref{pbnonhom-intro}. Note that in this
situation, the operator ${\mathscr A}$ satisfies the Kalman rank
condition with $m=1$. The same problem is investigated with
a stochastic approach in \cite{priola}.

Very recently the results in \cite{priola,nic} have been generalized
in \cite{priola-1} with both analytic and stochastic methods to the
case when $\hat {\mathscr A}={\mathscr A}+\sum_{j=1}^\po F_jD_j$ with
some $\po<N$, ${\mathscr A}$ still being given by \eqref{O-U-deg-1}.

In this paper we extend a part of the results in
\cite{Lo-0,Lo,priola,priola-1,nic} considering a class of elliptic
operators that, up to a change of the coordinates, may be written in
the following form
\begin{equation}
{\mathscr
A}\va(x)=\sum_{i,j=1}^{p_0}q_{ij}(x)D_{ij}\va(x)+\sum_{i,j=1}^Nb_{ij}x_jD_i\va(x)+\sum_{j=1}^{p_0}F_j(x)D_j\va(x),
\qq\;\,x\in\R^N, \label{oper-1}
\end{equation}
for some $p_0<N$, where the matrices $Q_0(x)=(\hat q_{ij}(x))$,
defined by $\hat q_{ij}\equiv q_{ij}$ if $i,j\le p_0$ and $\hat
q_{ij}\equiv 0$ otherwise, and $B$ satisfy the Kalman rank condition
\eqref{kalman} for some $m$ independent of $x$. We assume that $F:\R^N\to\R^N$ is a smooth function
with derivatives whose growth at infinity is comparable with the
growth of the minimum eigenvalue of the matrix $Q^{\frac{1}{2}}(x)$.
In the particular case when $F\equiv 0$, our results apply to any
elliptic operator of the type
\begin{equation}
{\mathscr
A}\va(x)=\sum_{i,j=1}^Nq_{ij}(x)D_{ij}\va(x)+\sum_{i,j=1}^Nb_{ij}x_jD_i\va(x),\qq\;\,x\in\R^N,
\label{oper-2}
\end{equation}
when the Kalman rank condition is satisfied, by any fixed
$x\in\R^N$, for some $m< N$, independent of $x$.
\par
The paper is organized as follows. First, in Section
\ref{sect:prelim} we introduce the function spaces we deal with, as well as some notation. Moreover, we introduce the
Hypotheses that will be assumed in the whole of the paper and we
recall some preliminary results mainly from \cite{Lo-0}. Next, in
Section \ref{main-sec}, the main part of this paper, we prove
uniform estimates of the spatial derivatives for the semigroups associated with the family of
non-degenerate elliptic operators ${\mathscr A}_{\ve}:={\mathscr A}+\ve \Delta_\star$
with $\Delta_\star$ being the Laplacian containing the missing second order derivatives, i.e.,
$\Delta_\star=D_{(p_0+1)(p_0+1)}+\cdots + D_{NN}^2$.
More precisely, we show that the constants appearing in the estimates can be chosen to be independent of
$\ve\in ]0,1[$. Then, in Section \ref{sec-sem}, using these
estimates, we prove that we can associate a semigroup $\{T(t)\}$ of
bounded operators in $C_b(\R^N)$ with the operators ${\mathscr A}$
in \eqref{oper-1} and \eqref{oper-2} and that the uniform estimates
of the preceding section may be extended to $\{T(t)\}$. We also state
some remarkable continuity properties of the semigroup $\{T(t)\}$.
Further, we show that we can associate a ``weak'' generator with the
semigroup $\{T(t)\}$, a generalization of the classical concept of
infinitesimal generator of a strongly continuous semigroup, and we
give a characterization of its domain. In Section \ref{sec-optimal},
we prove Schauder estimates for the distributional solutions to the
elliptic equation \eqref{ell-intro} and the non-homogeneous Cauchy
problem \eqref{pbnonhom-intro}. Finally, in Appendix \ref{sec:tec} we
 prove some technical lemmas that are used in the proof of
the uniform estimates.

%%%%%%%%%%%%%%%%%%%%%%%%%%%%%%%%%%%%%%%%%%%%%%%%%%%%%%%%%%%%%%%%%%%%%%%
%%                                                                   %%
%%             PRELIMINARIES                                         %%
%%                                                                   %%
%%%%%%%%%%%%%%%%%%%%%%%%%%%%%%%%%%%%%%%%%%%%%%%%%%%%%%%%%%%%%%%%%%%%%%%
\section{Main assumptions and preliminaries}\label{sect:prelim}

In this section we introduce the main assumptions on the operators
we consider. We also fix the notation and the define the function
spaces we use in this paper.

\subsection{Hypotheses}

The assumptions on the coefficients of the operator ${\mathscr A}$
in \eqref{oper-1} and \eqref{oper-2}, we always put throughout
this paper are the following. We begin by considering the case when
${\mathscr A}$ is given by \eqref{oper-1}.
\begin{ipos}
\label{ipos-1} \pn
\begin{enumerate}[\rm(i)]
\item
$Q(x)=(q_{ij}(x))$ is a $p_0\times p_0$ symmetric matrix, with
entries which belong to $C^{\kappa}(\R^N)$ for some $\kappa\in\N$,
$\kappa\ge 3$, such that
\begin{equation}
\sum_{i,j=1}^{p_0}q_{ij}(x)\xi_i\xi_j\ge\nu(x)|\xi|^2,\qq\;\,x\in\R^N,\;\,\xi\in\R^r,
\label{cond-ellipt-1}
\end{equation}
\pn for some positive function $\nu$ such that
$\inf_{\R^N}\nu(x)=\nu_0>0$. Further,
\begin{equation}
|D^{\alpha}q_{ij}(x)|\le
C_{\|\alpha\|}|x|^{(1-|\alpha|)^+}\sqrt{\nu(x)},\qq\;\,x\in\R^N,\q\;\,i,j=1,\ldots,p_0,
\q\;\,\|\alpha\|\le\kappa, \label{cond:perlimit-1}
\end{equation}
for some positive constant $C_{\|\alpha\|}$.
\item
There exist integers $p_1,\ldots,p_r$ with $p_0\ge p_1>\ldots\ge
p_r$ such that the matrix $B$ can be split into blocks as follows:
\begin{equation}
B=\left (
\begin{array}{ccccc}
\star & \star & \ldots & \ldots & \star\\[1mm]
B_1 &  \star  & \ldots & \ldots & \star\\[1mm]
0 & B_2 & \star & \ldots & \star\\[1mm]
\vdots & \vdots & \ddots & \ddots & \vdots\\[1mm]
0 & 0 & 0 & B_r & \star
\end{array}
\right), \label{matrix-B}
\end{equation}
where $B_h$ is a $p_{h}\times p_{h-1}$ matrix with full rank, i.e.,
${\rm rank}(B_h)=p_h$ $(h=1,\ldots,r)$.
\item $F\in C^{\kappa}(\R^N,\R^{p_0})$ and
\begin{equation*}
|D^{\alpha}F(x)|\le
C\sqrt{\nu(x)},\qq\;\,x\in\R^N,\;\,\|\alpha\|\le\kappa.
\end{equation*}
\end{enumerate}
\end{ipos}

\begin{rem}
\label{rem:2.9} \pn
\begin{enumerate}[\rm(i)]
\item
Since the coefficients $q_{ij}$ ($i,j=1,\ldots, p_0$) need to
satisfy both \eqref{cond-ellipt-1} and \eqref{cond:perlimit-1}, the
$q_{ij}$'s ($i,j=1,\ldots,p_0$) and $\nu$  may grow {\it at most} as
$|x|^2$ as $|x|\to +\infty$;
\item
Hypotheses \ref{ipos-1} guarantee that the operator ${\mathscr A}$
is hypoelliptic in the sense of H\"ormander, at any $x\in\R^N$.
\end{enumerate}
\end{rem}
\pn The hypotheses on the coefficients of the operator ${\mathscr
A}$ in \eqref{oper-2} are the following.
\begin{ipos}\label{ipos} \pn
\begin{enumerate}[\rm(i)]
\item $Q=(q_{ij})$ is a $N\times N$ symmetric matrix with $q_{ij}\in
C^{\kappa}(\R^N)$ $(i,j=1,\ldots,N)$ for some $\kappa\in\N$,
$\kappa\ge 3$, and there exists a function $\nu:\R^N\to\, ]0,+\infty[$ such that
$\nu_0:=\inf_{x\in\R^N}\nu(x)>0$ and
\begin{equation*}
\sum_{i,j=1}^Nq_{ij}(x)\xi_i\xi_j\ge \nu(x)|\xi|^2,\qq\;\,\xi\in
(\ker(Q(0)))^{\perp},\q\;\,x\in\R^N.
\end{equation*}
\item
For any $\alpha\in\N_0^N$ with length at
most $\kappa$, there exists a positive constant $C=C_{\|\alpha\|}$
such that
\begin{equation}
|D^{\alpha}q_{ij}(x)|\le
C|x|^{(1-|\alpha|)^+}\sqrt{\nu(x)},\qq\;\,x\in\R^N,\q\;\,i,j=1,\ldots,N,
\q\;\,\|\alpha\|\le\kappa.\label{cond:perlimit}
\end{equation}
\item
The kernel of the matrix $Q(x)$ is independent of $x\in\R^N$ and it
is a proper subspace of $\R^N$. Moreover, $\ker(Q(0))$ does not
contain non-trivial subspaces which are invariant for $B^*$.
\end{enumerate}
\end{ipos}

\begin{rem}
\label{rem:2.5} Note that Hypothesis \ref{ipos}(iii) can be rewritten
in one of the following equivalent forms:
\begin{enumerate}[(a)]
\item
the matrix $Q_t(x)=\int_0^te^{sB}Q(x)e^{sB^*}ds$ is positive
definite for any $t>0$ and any $x\in\R^N$;
\item
there exists $r<N$ such that the rank of the block matrix
\begin{equation*}[Q(x),BQ(x),B^2Q(x),\ldots,B^rQ(x)]
\end{equation*}
 is $N$ for any
$x\in\R^N$.
\end{enumerate}
To prove this claim, it suffices to adapt to our situation the proof
of \cite[Proposition A.1]{LP}. For the reader's convenience we give
a detailed proof in the appendix (see Lemma \ref{lemma:linear-algebra}).
\end{rem}
\begin{rem}
\label{rem:2} If the coefficients of the operator ${\mathscr A}$ in
\eqref{oper-2} satisfy Hypotheses \ref{ipos-1}, then one can find a
suitable change of variables which transforms ${\mathscr A}$ in an
operator of the type \eqref{oper-1} (with $F\equiv 0$). To check
this fact, let us denote by $\{V_k: k\in\N\}$ the sequence of nested
vector spaces defined by
\begin{equation*}
V_k=\left (\ker(Q(0))\cap \ker(Q(0)B^*)\cap\ldots\cap
\ker(Q(0)(B^*)^k)\right )^{\perp},
\end{equation*}
for any $k\in\N$. In view of Lemma \ref{lemma:linear-algebra} and
Hypothesis \ref{ipos}(iii), there exists a positive integer $p_0<N$
such that $V_{p_0}=\R^N$ and $V_k$ is properly contained in
$V_{k+1}$ if $k<p_0$.

Let now $W_0=V_0$ and $W^k$ be the orthogonal of $V_{k-1}$ in $V_k$,
for any $k=1,\ldots,p_0$. Let $p_k={\rm dim}(W_k)$ for any $k\le
p_0$. Of course, $\R^N=\bigoplus_{k=0}^rW_k$. Fix an orthonormal
basis $\{e_1',\ldots,e_N'\}$ of $\R^N$ consisting of vectors of the
spaces $W_k$ ($k=0,\ldots,r$). Adapting the proof of
\cite[Proposition 2.1]{LP} to our situation, we can show that in the
basis $\{e_1',\ldots,e_N'\}$ the operator ${\mathscr A}$ may be
written as in \eqref{oper-1} with the coefficients satisfying
Hypotheses \ref{ipos-1}.
\end{rem}

In view of Remark \ref{rem:2}, without loss of generality,
throughout the paper, we can limit ourselves to dealing with the
case when ${\mathscr A}$ is given by \eqref{oper-1} and its
coefficients satisfy Hypotheses \ref{ipos-1}.

\subsection{General notation}
\label{notation}
\subsubsection*{Functions}
For any real-valued function $u$ defined on a domain of
$\R\times\R^N$, we indiscriminately write $u(t,\cdot)$ and $u(t)$ when
we want to stress the dependence of $u$ on the time variable $t$.
Moreover, for any smooth real-valued function $v$ defined on a
domain of $\R^N$, we denote by $Dv$ its gradient and by $|Dv(x)|$
the Euclidean norm of $Dv(x)$ at $x$. Similarly, by $D^kv$
($k\in\N$) we denote the vector consisting of all the $k^{\text{th}}$ order
derivatives of $v$ with no repetitions. This means that we identify
$k^\text{th}$ order derivatives of type $\frac{\partial^k v}{\partial
x_{i_1}\ldots\partial x_{i_k}}$ and $\frac{\partial^k v}{\partial
x_{j_1}\ldots\partial x_{j_k}}$ when $(j_1,\ldots,j_k)$ is a
permutation of $(i_1,\ldots,i_k)$. We agree that the vector $D^kv$
contains only derivatives $\frac{\partial^k u}{\partial
x_{i_1}\ldots\partial x_{i_k}}$ with $i_1\le i_2\le\ldots i_k$. We
denote by $|D^kv(x)|$ the Euclidean norm of the vector $D^kv(x)$.

\subsubsection*{Asymptotics}
Given any real-valued function $u$ defined in some neighborhood of
$+\infty$ and $m\in\N$, we use the usual notation $u=o(s^m)$ when
$\lim_{s\to +\infty}s^{-m}u(s)=0$.  If
$\{u_{\konst}\}_{\konst\in{\mathscr F}}$ is a family of functions
which are defined in a right-neighborhood of $0$ (independent of
$\konst$), we write $u_{\konst}=o(t^m)$ (for some $m\in\N$) when
$\lim_{t\to 0^+}t^{-m}u_{\konst}(t)=0$ for any of such parameters.

\subsubsection*{Matrices}
We denote the  $k\times k$ identity matrix by $I_k$ and the
transposed of a matrix $A$ by  $A^*$. For any matrix $A$ we denote
by $\|A\|$ its Euclidean norm. If $A$ is symmetric,
 $\lambda_{\min}(A)$ is the
minimum eigenvalue of $A$. Finally, we use the notation ``$\star$''
to denote matrices when we are not interested in their entries.

\subsubsection*{Miscellanea}
We agree that $\N_0=\N\cup\{0\}$. Given a multi-index
$\alpha=(\alpha_1,\ldots,\alpha_m)\in\N_0^m$, we denote by
$\|\alpha\|:=\sum_{i=1}^m\alpha_i$ its length. Moreover, by $a^+$ we
denote the maximum between $a\in\R$ and $0$. For any $R>0$, we
denote by $B(R)$ the open ball in $\R^N$ centered at $x=0$ and with
radius $R$. $\ov{B(R)}$ is its closure.

\subsection{Ordering the derivatives of smooth functions}
\label{sub-notation} Here, we introduce a splitting of the vector of all the
derivatives of a function $u:\R^N\to\R$ of a given order into
sub-blocks. This splitting will be extensively used in Section
\ref{main-sec}.

Given $k,q\in\N$, we introduce a (total) ordering ``$\preceq_q$'' in
the set ${\mathscr I}_{k,q}$ of all the multi-indices in
$\N_0^{q+1}$ with length $k$. We say that
$(m_0,\ldots,m_q)\preceq_q(m_0',\ldots,m_q')$ if there exists
$h=0,\ldots,q$ such that $m_j=m_j'$ for any $j=0,\ldots,h-1$ and
$m_{h}>m_{h}'$. We thus may order the elements of ${\mathscr
I}_{k,q}$ in a sequence $i_{1}^{(k,q)}\preceq_q\cdots\preceq_q
i_{c_{k,q}}^{(k,q)}$. Here, $c_{k,q}:=\left ({q+k}\atop{q}\right )$.

Now, to order the entries of the vector $D^ku$ ($k\in\N$) we proceed
as follows. Let $\{p_0,\ldots,p_r\}$ be a given  set of
non-increasing integers such that $p_0+\cdots+p_r=N$, throughout the paper these will be fixed as in Hypotheses \ref{ipos-1} (ii). We set
$p_{-1}:=0$ and introduce the sets $\II_j=\{i\in\N: r_j<i\le
r_{j+1}\}$, ($j=0,\ldots,r$), where $r_l=\sum_{k=0}^lp_{k-1}$ for
any $l=0,\ldots,r+1$. Moreover, we split $\R^N$ into the direct sum
$\R^N=\bigotimes_{j=0}^r\R^{p_j}$. Hence, any multi-index
$\alpha\in\N^N_0$ can be split as
$\alpha=(\alpha_0,\ldots,\alpha_r)$ with $\alpha_j\in\N^{p_j}_0$
($j=0,\ldots,r$) and we can write $|\alpha|:=(\|\alpha_0\|,\dots,\|\alpha_r\|)$.
 We can now split the vector $D^ku$ as follows:
\begin{enumerate}[\rm(i)]
\item
we split $D^ku$ into blocks according to the rule:
$D^ku=(D^k_1u,\ldots,D^k_{c_{k,r}}u)$, where $D^k_ju$ ($j=1,\ldots,c_{k,r}$)
contains all the derivatives $D^{\alpha}\va$ of order $k$ such that
$|\alpha|=i_j^{(k,r)}$, where

\item
we order the entries of the vectors $D^k_{j}u$
($j=1,\ldots,c_{k,r}$) according to the following rule: if
$D^{\alpha}u$ and $D^{\beta}u$ belong to the block $D^k_{j}u$, we
say that $D^{\alpha}u$ precedes $D^{\beta}u$ if
$\beta\preceq_{N-1}\alpha$.
\end{enumerate}

\subsection{H\"older spaces}
Here, we introduce most of the isotropic function spaces we deal
with in this paper.

\begin{definition}\label{def:2.1} For any $k\ge 0$, $C_b^k(\R^N)$ denotes the subset
of $C^k(\R^N)$ of functions which are bounded together with their
derivatives up to the $[k]^{\text{th}}$ order. We endow it with the norm
$$
\|u\|_{C^k_b(\R^N)}=\sum_{|\alpha|\le [k]}
\|D^{\alpha}f\|_{\infty}+\sum_{|\alpha|=[k]}[D^{\alpha}f]_{C^{k-[k]}_b(\R^N)},
$$
where $\|D^{\alpha}f\|_{\infty}$ denotes the sup-norm of
$D^{\alpha}f$ and $[D^{\alpha}f]_{C^{k-[k]}_b(\R^N)}$ is the
$(k-[k])$-H\"older seminorm of $f$. We say that $u\in
C^{\infty}_b(\R^N)$ if it belongs to $C^k_b(\R^N)$ for any $k\ge 0$.
Finally, given an open set $\Omega$ $($eventually, $\Omega=\R^N)$,
by $C^{\infty}_c(\Om)$ we denote the set of all infinitely many
times differentiable functions with compact support.
\end{definition}

We now define the anisotropic spaces ${\mathscr C}^{\theta}(\R^N)$
($\theta\in\R_+$). Let $p_0,\ldots,p_r$ be as in Hypothesis
\ref{ipos-1}(ii). To simplify the notation, we split any $x\in\R^N$
as $x=(x_0,\ldots,x_r)$ with $x_j\in\R^{p_j}$ ($j=0,\ldots,r$).

\begin{definition}\label{spazi:holder} For any $\theta>0$,
${\mathscr C}^{\theta}(\R^N)$ consists of all bounded functions
$f:\R^N\to\R$ such that
$f(x_0,\ldots,x_{j-1},\cdot,x_{j+1},\ldots,x_r)$ belongs to the
H\"older space $C_b^{\theta/(2j+1)}(\R^{p_j})$ for any $\hat
x_j:=(x_0,\ldots,x_{j-1},x_{j+1},\ldots,x_r)$ in $\R^{N-p_j}$, and
\begin{equation}
\|f\|_{j,\theta}:=\sup_{\hat
x_j\in\R^{N-p_j}}\|f(x_1,\ldots,x_{j-1},\cdot,x_{j+1},\ldots,x_r)\|_{
C^{\theta/(2j+1)}_b(\R^{p_j})}<+\infty. \label{fjtheta}
\end{equation}
\pn We norm it by $\|f\|_{{\mathscr
C}^{\theta}(\R^N)}=\sum_{j=0}^r\|f\|_{j,\theta}$ for any $f\in
{\mathscr C}^{\theta}(\R^N)$. When $\theta$ is such that
$\theta/(2j+1)\in\N$ for some $j=0\,\dots,r$, we assume that all the
existing derivatives of $f\in {\mathscr C}^{\theta}(\R^N)$ are
continuous in $\R^N$.
\end{definition}

%%%%%%%%%%%%%%%%%%%%%%%%%%%%%%%%%%%%%%%%%%%%%%%%%%%%%%%%%%%%%%%%%%%%%%%
%%                                                                   %%
%%             UNIF. GRAD. EST.                                      %%
%%                                                                   %%
%%%%%%%%%%%%%%%%%%%%%%%%%%%%%%%%%%%%%%%%%%%%%%%%%%%%%%%%%%%%%%%%%%%%%%%

\section{Uniform estimates for the approximating semigroups}
\label{main-sec}

\noindent To investigate the elliptic and parabolic problems associated with
${\mathscr A}$ we approximate this operator by the uniformly elliptic
operator ${\mathscr A}_{\ve}$ defined on smooth function $\va$ by
\begin{equation*}
{\mathscr A}_{\ve}\va(x):={\mathscr
A}\va(x)+\ve\sum_{i=p_0+1}^ND_{ii}\va(x),\qq\;\,x\in\R^N,
\end{equation*}
\pn
for any $\ve>0$. It is known that one can associate a
semigroup of bounded linear operators $\{T_{\ve}(t)\}$ on
$C_b(\R^N)$ with each operator ${\mathscr A}_{\ve}$. For any $f\in
C_b(\R^N)$ and any $t>0$, $T_{\ve}(t)f$ is the value at $t$ of the
unique classical solution to the Cauchy problem
\begin{equation}
\left\{
\begin{array}{lll}
D_tu(t,x)={\mathscr A}_{\ve}u(t,x),\q &t\in ]0,+\infty[, &x\in\R^N,\\[2mm]
u(0,x)=f(x), &&x\in\R^N.
\end{array}
\right. \label{pb-approx}
\end{equation}
\pn The uniqueness of the classical solution to problem
\eqref{pb-approx} follows from a corresponding maximum principle
(see, e.g., Proposition \ref{prop:maxprinc:0}(ii)). The existence of
a solution to problem \eqref{pb-approx} can be proved approximating
such a problem with Dirichlet Cauchy problems in balls centered at
$0$ and radius $n$ and using classical Schauder estimates and a
compactness argument to show that the sequence of solutions
$\{u_n\}$ to such  Dirichlet Cauchy problems converges, as $n\to +\infty$, to a
function $u_\ve$ which turns out to solve problem \eqref{pb-approx}.
We refer the reader, for example, to \cite[Chapter
1]{bertoldi-lorenzi} and \cite[Section 4]{MPW} for more details.

By letting $\ve$ go to $0$ and applying a compactness
argument we will show the existence of a semigroup ``generated by''
${\mathscr A}$. For this purpose we need  estimates for the spatial derivatives of $\{T_\ve(t)\}$ uniformly for $\ve\in ]0,1]$. This section is devoted
to the proof of such estimates.

We start with a maximum principle for (degenerate) elliptic and
parabolic equation, which leads  to uniqueness of the distributional
solutions to the problems \eqref{ell-intro} and
\eqref{pbnonhom-intro}, but which will be also crucial in the proof
of the estimates for the spatial derivatives in Theorem \ref{thm:3.1} and Theorem
\ref{thm:3.2}. We postpone the, more or less standard, proof to
Appendix \ref{sec:tec}.

\begin{prop}\label{prop:maxprinc:0}
Let ${\mathscr L}$ be any, degenerate or non-degenerate, elliptic operator defined on smooth functions
$\psi$ by
\begin{equation*}
{\mathscr
L}\psi(x)=\sum_{i,j=1}^mq_{ij}(x)D_{ij}\psi(x)+\sum_{i,j=1}^Nb_{ij}x_jD_i\psi(x)
+\sum_{j=1}^mF_j(x)D_j\psi(x),\qq\;\,x\in\R^N,
\end{equation*}
\pn with the coefficients $q_{ij}$ and $F_j$ $(i,j=1,\ldots,m)$
being $($possibly$)$ unbounded functions in $\R^N$ which may grow,
respectively, at most quadratically and linearly at infinity. Then
the following assertions hold true.
\begin{enumerate}[\rm(i)]
\item
Let $u\in C_b(\R^N)$ be a distributional solution to the equation
$\lambda u-{\mathscr L}u=f$, corresponding to some $f\in C_b(\R^N)$
and $\lambda>0$, Further, suppose that $D_iu$ and $D_{ij}u$ exist in
the classical sense for any $i,j=1,\ldots,m$. Then,
\begin{equation*}
\lambda\|u\|_{C_b(\R^N)}\le \|f\|_{C_b(\R^N)}. %\label{stima-cont}
\end{equation*}
\item
Let $u:[0,\timeT]\times\R^N\to\R$ $(\timeT>0)$ be a distributional
solution of the Cauchy problem
\begin{equation*}
\left\{
\begin{array}{lll}
D_tu(t,x)={\mathscr L}u(t,x)+g(t,x), & t\in ]0,\timeT[, &x\in\R^N,\\[2mm]
u(0,x)=f(x), && x\in\R^N,
\end{array}
\right.
\end{equation*}
\pn corresponding to some $f\in C_b(\R^N)$ and $g\in
C(]0,\timeT]\times\R^N)$. Further, assume that $D_iu$, $D_{ij}u$
$(i,j=1,\ldots,m)$ exist in the classical sense. If $g\le 0$ in
$]0,\timeT]\times\R^N$, then $\sup_{[0,\timeT]\times
\R^N}u\le\sup_{\R^N}f$. Similarly, if $g\ge 0$ in
$]0,\timeT]\times\R^N$, then
$\inf_{[0,\timeT]\times\R^N}u\ge\inf_{\R^N}f$. In particular, if
$g\equiv 0$, then
\begin{equation}
\|u(t,\cdot)\|_{\infty}\le\|f\|_{\infty},\qq\;\,t\in [0,\timeT].
\label{stimaapriori:1}
\end{equation}
\end{enumerate}
\end{prop}

The following theorem will be the most crucial ingredient for the
construction of the semigroup associated with ${\mathscr A}$.
\begin{theorem}
\label{thm:3.1} For any $\ve>0$, any $h\in\N$ and any $f\in
C^h_b(\R^N)$, the function $T_{\ve}(t)f$ belongs to
$C^{\kappa}_b(\R^N)$ for any $t>0$. Moreover, for any $\timeT>0$ and
any $h,l\in\N$ with $h\le l$, the function $(t,x)\mapsto
t^{(l-h)^+/2}(D^lT_{\ve}(t)f)(x)$ is bounded and continuous in
$[0,+\infty[\times\R^N$, and when $l>h$ it vanishes at $t=0$.
\end{theorem}

\begin{proof}
We restrict ourselves to showing the assertion in the case when
$h=0$, the other cases being similar and even easier. We split the
proof into two steps. In the first one, we prove that there exists a
positive constant $C$, independent of $f$, such that
\begin{equation}
\|D^lT_{\ve}(t)f\|_{\infty}\le Ct^{-\frac{l}{2}}\|f\|_{\infty},
\label{unif-estim-ve}
\end{equation}
\pn for any $f\in C_b(\R^N)$, $t>0$. Next, in Step 2, we prove that
the function $(t,x)\mapsto t^{l/2}(D^lT_{\ve}(t)f)(x)$ is continuous
up to $t=0$. \med \pn {\em Step 1.} Without loss of generality, we
can limit ourselves to proving \eqref{unif-estim-ve} in the
particular case when $f\in C^{\infty}_c(\R^N)$. Indeed, in the
general case it suffices to approximate $f\in C_b(\R^N)$ with a
sequence of smooth functions $f\in C^{\infty}_c(\R^N)$, bounded in
$C_b(\R^N)$ and converging to $f$ locally uniformly in $\R^N$. It is
well known that $T_{\ve}(\cdot)f_n$ converges to $T_{\ve}(\cdot)f$
uniformly in $[0,\timeT]\times \ov{B(M)}$, as $n\to +\infty$, for
any $M,\timeT>0$ (see e.g., \cite[Proposition
2.2.9]{bertoldi-lorenzi} or \cite[Proposition 4.6]{MPW}). Moreover,
the classical interior estimates in \cite[Chapter 4, Theorem
5.1]{LSU} imply that
\begin{equation*}
\|D^lT_{\ve}(\cdot)f_n-D^lT_{\ve}(\cdot)f\|_{C([\timeT/2,\timeT]\times
B(M))}\le \hat
C\|T_\ve(\cdot)f_n-T_\ve(\cdot)f\|_{L^{\infty}([0,2\timeT]\times
\ov{B(2M)})},
\end{equation*}
\pn for any $M,\timeT>0$ and some positive constant $\hat C$,
depending on $M,\timeT$. Hence, $D^lT_{\ve}(t)f_n$ converges to
$D^lT_{\ve}(t)f_n$ locally uniformly in $]0,+\infty[\times\R^N$ and
this allows us to extend \eqref{unif-estim-ve} to any $f\in
C_b(\R^N)$.

Now for the proof of \eqref{unif-estim-ve} for $f\in
C_c^\infty(\R^N)$, let $\va\in C^{\infty}_c(\R)$ be a non-increasing
function such that $\va(t)=1$ for any $t\in\, ]-1/2,1/2[$, $\va(t)=0$
for any $t\in\R\setminus ]-1,1[$. For $R>1$ define the functions
$\eta_R:\R^N\to\R$ by $\eta_R(x):=\va(|x|/R)$ and $v_R$ by
\begin{equation*}
v_R(t,x):=\sum_{m=0}^l{\konst}^mt^m\eta^{2m}_R(x)|D^mu_R(t,x)|^2,\qq\;\,
t\in [0,\timeT],\;\,x\in B(R),
\end{equation*}
\pn where $u_R$ denotes the classical solution to the Dirichlet
Cauchy problem in the ball $B(R)$ with initial value $f$, and
${\konst}$ is a positive parameter to be fixed later on ($\konst$
will be \emph{small}). To simplify the notation, we drop out the
index $R$, when there is no danger of confusion.

The classical Schauder estimates of \cite[Chapter 4, Theorem
5.1]{LSU} imply that $v$ is continuous in
$[0,\timeT]\times\ov{B(R)}$. Moreover, a straightforward computation
shows that $v$ solves the Cauchy problem
\begin{equation}
\left\{
\begin{array}{lll}
D_tv(t,x)={\mathscr A}v(t,x)+g(t,x),\q &t\in [0,\timeT], &x\in B(R),\\[1.5mm]
v(t,x)=0, &t\in [0,\timeT], &x\in\partial B(R),\\[1.5mm]
v(0,x)=(f(x))^2, &&x\in B(R),
\end{array}
\right. \no
\end{equation}
\pn where, for any $t\in [0,\timeT]$ and any $x\in B(R)$, the
function $g$ is given by $g(t,x)=\sum_{j=1}^4g_j(t,x)$ with
\begin{align*}
g_1(t,\cdot)&=
-2\sum_{i,j=1}^N\sum_{m=0}^l\konst^mt^m\eta^{2m}q_{ij}\langle D^mD_iu(t),D^mD_ju(t)\rangle,\\[1.5mm]
g_2(t,\cdot)&= -\langle
QD\eta,D\eta\rangle\sum_{m=1}^l2m(2m-1)\konst^mt^m\eta^{2m-2}|D^mu(t)|^2\\
&\quad+\sum_{m=1}^lm\konst^mt^{m-1}\eta^{2m}|D^mu(t)|^2,\\[1.5mm]
g_3(t,\cdot)&= -2{\mathscr
A}\eta\sum_{m=1}^lm\konst^mt^m\eta^{2m-1}|D^mu(t)|^2\\
&\quad
-8\sum_{i,j=1}^N\sum_{m=1}^lm\konst^mt^m\eta^{2m-1}q_{ij}D_i\eta\langle D^mu(t),D^mD_ju(t)\rangle,\\[1.5mm]
g_4(t,\cdot)&= 2\sum_{m=1}^l\konst^mt^m\eta^{2m}\langle
[D^m,{\mathscr A}]u(t),D^mu(t)\rangle.
\end{align*}
Here, $[D^m,{\mathscr A}]$ denotes the commutator between the
operators $D^m$ and ${\mathscr A}$. Using the ellipticity assumption
on $q_{ij}$ we get
\begin{align}
g_1(t)&\,\le
-2\nu\sum_{m=0}^l\konst^mt^m\eta^{2m}|D^{m+1}_{\star}u(t)|^2
-2\ve\sum_{m=0}^l\konst^mt^m\eta^{2m}|D^{m+1}_{\star\star}u(t)|^2\label{stimasug1}\\\no
&\,=-2\nu\sum_{m=1}^{l+1}\konst^{m-1}t^{m-1}\eta^{2m-2}|D^{m}_{\star}u(t)|^2
-2\ve\sum_{m=1}^{l+1}\konst^{m-1}t^{m-1}\eta^{2m-2}|D^{m}_{\star\star}u(t)|^2,
\end{align}
\pn where $D^m_{\star}u$ (respectively $D^m_{\star\star}u$) denotes
the vector whose entries are the $m^{\text{th}}$ order derivatives
$\frac{\partial^m u}{\partial x_{i_1}\ldots\partial x_{i_m}}$ with
$i_j\le p_0$ for some $j=1,\ldots, m$ (respectively $i_j>p_0$ for
all $j=1,\ldots,m$).

\med\par We turn to estimating the function $g_3$. From Hypotheses
\ref{ipos-1} it follows easily that
\begin{equation}
\big |{\mathscr A}\eta(x)\big |\le C_1,\qq\;\, \big
|(Q(x)D\eta(x))_i\big |\le C_1\left\{
\begin{array}{ll}
\sqrt{\nu(x)},\q & \mbox{if}~i\le p_0,\\[2mm]
\ve, & \mbox{if}~i>p_0,
\end{array}
\right. \label{stimaperg4:1}
\end{equation}
\pn for any $x\in\R^N$ and some positive constant $C_1$. Taking this
into account and using Young's inequality we conclude
\begin{align}
g_3(t)&\le\, 2C_1\sum_{m=1}^lm\konst^mt^m\eta^{2m-1}|D^mu(t)|^2
\no\\
&\quad+C_2\sum_{m=1}^lm\konst^mt^m\eta^{2m-1}\sqrt{\nu}\,|D^mu(t)|\cdot|D^{m+1}_{\star}u(t)|\no\\
&\quad+ C_2\ve\sum_{m=1}^lm\konst^mt^m\eta^{2m-1}|D^mu(t)|\cdot |D^{m+1}u(t)|\no\\
&\le\, 2C_1\sum_{m=1}^lm\konst^mt^m\eta^{2m-1}|D^mu(t)|^2\label{stimasug4}\\
&\quad+C_2\sum_{m=1}^lm\left
(\konst^{m-\frac{1}{2}}t^{m-\frac{1}{2}}\eta^{2m-2}|D^mu(t)|^2
+\konst^{m+\frac{1}{2}}t^{m+\frac{1}{2}}\eta^{2m}\nu |D^{m+1}_{\star}u(t)|^2\right )\no\\
&\quad\, + C_2\ve\sum_{m=1}^lm\left
(\konst^{m-\frac{1}{2}}t^{m-\frac{1}{2}}\eta^{2m-2}|D^mu(t)|^2
+\konst^{m+\frac{1}{2}}t^{m+\frac{1}{2}}\eta^{2m}
|D^{m+1}u(t)|^2\right ), \no
\end{align}
\pn for any $t\in [0,\timeT]$ and some positive constant $C_2$,
independent of $\ve$ and $t$. The term $g_4$ can be estimated
similarly, taking now \eqref{cond:perlimit} into account. We obtain
\begin{align*}
g_4(t)\le &\;2\|B\|_{\infty}\sum_{m=1}^l\konst^mt^m\eta^{2m}|D^mu(t)|^2\\
&+2\sum_{m=1}^l\konst^mt^m\eta^{2m}\sum_{n=1}^m\Bigl(\|D^nQ\|_{\infty}|D^{m+2-n}_{\star}u(t)|\cdot|D^mu(t)|
\\&\quad\quad\quad\quad\quad\quad\quad\quad\quad\quad\quad+\|D^nF\|_{\infty}|D^{m+1-n}_{\star}u(t)|\cdot|D^mu(t)|\Bigr),
\end{align*}
\pn for any $t\in [0,\timeT]$, where $\|D^hQ\|_{\infty}$
(respectively $\|D^hF\|_{\infty}$) ($h=1,\ldots,l$) denotes the
maximum of the sup-norm of the functions $D^hq_{ij}$ (respectively
$D^hF_j$) ($i,j=1,\ldots,N$). Hence, taking Hypotheses
\ref{ipos-1}(i) and \ref{ipos-1}(ii) into account, we can write
\begin{align}
g_4(t) \le\,& C_3\nu\sum_{m=1}^{l+1}
\konst^{m-\frac{1}{2}}t^{m-\frac{1}{2}}\eta^{2m}|D^{m}_{\star}u(t)|^2
+C_3\sum_{m=1}^{l}\konst^{m}t^{m}\eta^{2m}|D^mu(t)|^2\no\\
&+C_3\sum_{m=1}^{l}\konst^{m-\frac{1}{2}}t^{m-\frac{1}{2}}\eta^{2m}|D^mu(t)|^2,
\label{stimasug5}
\end{align}
\pn for any $t\in [0,\timeT]$ and some positive constant $C_3$,
independent of $t$. Summing up, from \eqref{stimasug1},
\eqref{stimasug4} and \eqref{stimasug5} we easily deduce that
\begin{align*}
g(t)\le\;&\sum_{m=1}^{l+1}
M_m^{\star}({\konst},\timeT){\konst}^{m-1}t^{m-1}\eta^{2m-2}|D^m_{\star}u(t)|^2\\
&+\sum_{m=1}^{l+1}
M_m^{\star\star}({\konst},\timeT){\konst}^{m-1}t^{m-1}\eta^{2m-2}|D^m_{\star\star}u(t)|^2,
\end{align*}
\pn for any $t\in [0,\timeT]$, where
\begin{align*}
M_{m}^{\star}({\konst},\timeT):=\,&
(-2+C_2(m-1)\sqrt{\konst}\sqrt \timeT+C_3\sqrt{\konst}\sqrt \timeT)\nu\no\\
&+(C_2m+C_2\ve m+C_3+C_2\ve(m-1))\sqrt{\konst}\sqrt \timeT\no\\
&+(m+2C_1m\timeT+C_3\timeT)\konst,\\[1mm]
M_{m}^{\star\star}({\konst},\timeT):=\,& -2\ve+(C_2m+C_2\ve
m+C_2\ve(m-1)+C_3)\sqrt{\konst}\sqrt \timeT\no\\
&+(m+2C_1m\timeT+C_3\timeT)\konst,
\end{align*}
\pn for any $m=1,\ldots,l+1$. Since in $M_m^{\star}(\konst,\timeT)$
and $M_m^{\star\star}(\konst,\timeT)$ apart from the first negative
term everything vanishes as ${\konst}\searrow 0$ for any
$m=1,\ldots,l+1$, it  follows that for sufficiently small
${\konst}>0$ (independent of $R$!) the inequality  $g(t,x)\le 0$
holds for any $t\in [0,\timeT]$ and any $x\in B(R)$. The classical
maximum principle yields then
\begin{equation*}
|v_R(t,x)|\le\|f\|^2_{\infty} \quad\mbox{and so}\quad
t^m\eta_R^{2m}(x)|D^m u_R(t,x)|^2 \le C_m \|f\|^2_{\infty},\qq\;\,
 %\label{serve-per-cont}
\end{equation*}
for any $(t,x)\in [0,\timeT]\times\ov{B(R)}$.

\pn Now, \eqref{unif-estim-ve} follows by letting $R\to +\infty$.

\med \pn {\em Step 2.} We now conclude the proof by showing that the
function $w_l:[0,+\infty[\times\R^N\hskip-0.2em\to\R$ defined by
$w_l(t,x):=t^{-l/2}D^lu(t,x)$ is continuous on
$[0,+\infty[\times\R^N$. \med \pn If $f\in C^{\infty}_c(\R^N)$ this
claim is easily checked. Indeed, in this case, if  $u_R$ denotes the
solution of the Dirichlet Cauchy problem on $B(R)$ with initial value
$f$, it is well-known (see \cite[Chapter 4, Theorem 5.1]{LSU}) that, for
any $\timeT>0$ and any $m,M\in\N$, with $m<M$ and ${\rm
supp}(f)\subset B(m)$, there exists a positive constant
$C_l=C_l(m,M,\timeT)$ such that
\begin{align*}
\|u_R\|_{C^{l+{\theta}/2,2l+{\theta}}([0,\timeT]\times B(m))}\le\,&
C_l\left (\|f\|_{C^{2l+{\theta}}_c(\R^N)}
+\|u_R\|_{C([0,2\timeT]\times B(M))}\right )\\
\le\,& 2C_l\|f\|_{C^{2l+{\theta}}_c(\R^N)},
\end{align*}
\pn for any $R>0$. Hence, by a compactness argument, we can easily
show that $u_R$ converges to $T_{\ve}(\cdot)f$ in $C^{l,2l}_{\rm
loc}([0,+\infty[\times\R^N)$. Since the function $|D^lu_R|$ is
continuous in $[0,+\infty[\times\R^N$ so is the functions $w_l$,
too. \med\pn Let us now consider the general case when $f\in
C_b(\R^N)$. Then, there exists a sequence $\{f_n\}\in
C^{\infty}_c(\R^N)$ which is bounded in $C_b(\R^N)$ and converges to
$f$ locally uniformly in $\R^N$. Let us fix $k,m\in\N$. By
\cite[Proposition 1.1.3(iii)]{Lu1}, we know that
\begin{equation}
\|\psi\|_{C^{l}(\ov{B(M)})}\le
P_l\|\psi\|_{C(\ov{B(M)})}^{\frac{1}{l+1}}\|\psi\|_{C^{l+1}(\ov{B(M)})}^{\frac{l}{l+1}},
\label{intermediate}
\end{equation}
\pn for some positive constant $P_l=P_l(M)$ and any function
$\psi\in C^{l+1}(\ov{B(M)})$. Apply \eqref{intermediate} for
$\psi=t^{l/2}T(\cdot)f_n-t^{l/2}T(\cdot)f$ and use the already
proved inequality \eqref{unif-estim-ve} to conclude
\begin{align}
&\sup_{t\in ]0,\timeT]}\|t^{\frac{l}{2}}D^lT_\ve(t)f_n-t^{\frac{l}{2}}D^lT_\ve(t)f\|_{C(\ov{B(M)})}\no\\
\le\,& P_l
\|T_\ve(\cdot)f_n-T_\ve(\cdot)f\|_{C([0,\timeT]\times\ov{B(M)})}^{\frac{1}{l+1}}
\sup_{t\in ]0,\timeT]}\|t^{\frac{l+1}{2}}T_\ve(t)f_n-t^{\frac{l+1}{2}}T_\ve(t)f\|_{C^{l+1}(\ov{B(M)})}^{\frac{l}{l+1}}\no\\
\le\,&
P_l'\|T_\ve(\cdot)f_n-T_\ve(\cdot)f\|_{C([0,\timeT]\times\ov{B(M)})}^{\frac{1}{l+1}},
\label{conv-deriv}
\end{align}
\pn for some positive constant $P_l'$. The right-hand side of
\eqref{conv-deriv} vanishes as $n\to +\infty$. By the arbitrariness
of $\timeT$ and $M$, it follows immediately that the function $w_l$
is continuous in $[0,+\infty[\times\R^N$. In particular, it vanishes
at $t=0$ since the function $(t,x)\mapsto t^{l/2}(D^lT(t)f_n)(x)$
does for any $n\in\N$. This completes the proof.
\end{proof}

%%%%%%%%%%%%%%%%%%%%%%%%%%%%%%%%%%%%%%%%%%%%%%%%%%%%%%%%%%%%%%%%%%%%%%%
%%                                                                   %%
%%            MAIN GRAD.EST.UNIF                                     %%
%%                                                                   %%
%%%%%%%%%%%%%%%%%%%%%%%%%%%%%%%%%%%%%%%%%%%%%%%%%%%%%%%%%%%%%%%%%%%%%%%

We are now in a position to prove the main result of this section.
Our ultimate aim is to show that the semigroups $\{T_\ve(t)\}$
converge to a semigroup $\{T(t)\}$ which is associated with the
operator ${\mathscr A}$, and we also wish to establish estimates for the spatial derivatives of $\{T(t)\}$. Contrary to the uniformly elliptic
situation of $\{T_\ve(t)\}$ the behavior near $t=0$ of the partial
derivatives of $D^\alpha T(t)f$ is expected to depend not only on
the length $\|\alpha\|$ of the multi-index $\alpha$, but also on the
directions along which we differentiate. Thus the well-know behavior
$t^{-\|\alpha\|/2}$ is replaced by some function growing faster near
$0$. The exact behavior is well-known, e.g., for the
Ornstein-Uhlenbeck semigroup (see \cite{L}) and the optimal exponent
is actually given by the following function $q$. We define
$q:\N_0^{r+1}\to\R$ as
\begin{equation*}
q(\alpha)=\sum_{k=0}^{r}\frac{2k+1}{2}\alpha_k=\frac{1}{2}\|\alpha\|+\sum_{k=1}^{r}k\alpha_k,\quad
\alpha\in \N_0^{r+1}.
\end{equation*}
With this function the, still to be constructed, semigroup
$\{T(t)\}$ will obey the estimate $$\|D^\alpha
T(t)f\|_\infty \leq C t^{-q(|\alpha|)} \|f\|_\infty$$ for any $f\in
C_b(\R^N)$ and $\alpha\in \N_0^N$ (recall the notation $|\alpha|=(\alpha_0,\alpha_1,\dots,\alpha_r)$
from Subsection \ref{sub-notation}). Whereas, if we have a control
over certain derivatives of $f$, say $f\in C_b^h(\R^N)$ we expect a
better behavior. Indeed, this will be the case. For the precise
statement we will need the following function $q_h:\N_0^{r+1}\to\R$
($h\in\N_0$). We let
\begin{equation*}
q_h(\beta)=\frac{1}{2}
\|\beta\|-\frac{1}{2}h+\sum_{k=0}^{j(\beta)-1}k\beta_k+(j(\beta)-1)\bigg
(\sum_{k=j(\beta)}^r\beta_k-h\bigg ),% \label{exp-m}
\end{equation*}
for any multi-index $\beta\in \N_0^{r+1}$, where $j(\beta)\in\N_0$
is the smallest integer such that $\sum_{j=j(\beta)}^r\alpha_j\le
h$, and we agree that $\sum_{k=0}^{-1}k\alpha_k:=0$ and
$q_h(\beta)=0$, if $h\ge\|\beta\|$. This function describes the
expected behavior near $t=0$ in the estimates of the derivatives, and it
models the following: if we have a function in $C_b^h(\R^N)$, then
we can drop out any $h$ partial derivatives from a multi-index
$\alpha$, since these should not contribute to the power of $t$. We
do this in a way that derivatives which would give the largest
contribution  in the derivative-estimate are dropped out. Then, we can
evaluate our $q$ on this new multi-index and get the right behavior
near $t=0$.

\med\pn

The in $\ve\in ]0,1]$ uniform estimates for spatial derivatives of $\{T_\ve(t)\}$
 are given by the following result.
\begin{theorem}
\label{thm:3.2} For any compact interval $J\subseteq ]0,+\infty[$, multi-index
$\alpha=(\alpha_0,\ldots,\alpha_r)$, with $\alpha_j\in\N_0^{p_j}$
$(j=0,\ldots,r)$ and $\|\alpha\|\leq \kappa$, any $h\in\N_0$, with $h\le \|\alpha\|$, there exists a
positive constant $\tilde C$, depending on $\alpha$, but being
independent of $\ve\in ]0,1]$, such that
\begin{equation}
\|D^{\alpha}T_{\ve}(t)f\|_{\infty}\le \tilde Ct^{-q_h(|\alpha|)}
\|f\|_{C^h_b(\R^N)},\qq\;\,t\in J,\;\,\ve\in ]0,1].
\label{stimaapriori:ave}
\end{equation}
\end{theorem}

For the proof we need some preparation and auxiliary results. Define
the function $\ell:\N\setminus\{1\}\to\N$ as follows. Let
$i_m^{(k,r)}=(0,\alpha_1,\ldots,\alpha_r)$, where $k=\|\alpha\|$,
and let $j$ be the smallest integer such that $\alpha_{ j}>0$. Then,
$\ell(m)$ is the index such that
$i_{\ell(m)}^{(k,r)}=(0,\ldots,0,1,\alpha_{ j}-1,\alpha_{
j+1},\ldots,\alpha_r)$. As it is immediately seen,
$q(i^{(k,r)}_{\ell(m)})=q(i^{(k,r)}_m)-1$. Moreover, we have
$\ell(m)<m$. These will be used in the sequel without further
mentioning. We also need some properties of the function $q_h$
presented in the next lemma (the proof is in Appendix \ref{sec:tec}).
\begin{lemma}\label{lem:propqh}For $\alpha\in \N_0^{r+1}$, the
following hold.
\begin{enumerate}[{\rm (i)}]
\item
$\|\alpha\|\leq h$ if and only if $q_h(\alpha)=0$.
\item
We have $q_h(\alpha)\geq (\|\alpha\|-h)^+/2$.
\item
If $\|\alpha\|\geq h$  and
$\beta=(\alpha_0+1,\alpha_1,\dots,\alpha_r)$, then
$q_h(\beta)=q_h(\alpha)+1/2$.
\item
If $\beta=\alpha-e^{(r+1)}_j+e^{(r+1)}_{j'}$ for some $0\leq
j,j'\leq r$ such that $\alpha_j>0$ and $j'\leq j+1$, then
$q_h(\alpha)\geq q_h(\beta)-1$.
\item
Suppose that $\alpha_0,\dots,\alpha_{j_0-1}=0$ and $\alpha_{j_0}>0$
for some $j_0>0$. Set
$\beta=\alpha-e^{(r+1)}_{j_0}+e^{(r+1)}_{j_0-1}$. If  $h<l$, then
$q_h(\alpha)>1$ and $q_h(\beta)=q_h(\alpha)-1$.
\item
Suppose that $\alpha_0,\dots,\alpha_{j_0-1}=0$ and $\alpha_{j_0}>0$
for some $j_0>0$. Set
$\hat\alpha=\alpha-e^{(r+1)}_{j_0}+e^{(r+1)}_{j_0-1}$ and
$\beta=\hat\alpha-e^{(r+1)}_j+e^{(r+1)}_{j'}$ for some $0\leq
j,j'\leq r$ such that $\alpha_j>0$ and $j'\leq j+1$. Then
$q_h(\alpha)+q_h(\hat\alpha)\geq 2q_h(\beta)-1$.
\item
Let $\widetilde\alpha$ and $\alpha$ be two multi-indices such that
$\widetilde\alpha_j\le\alpha_j$ for all $j=0,\ldots,r$ and
$\widetilde\alpha_{j_0}<\alpha_{j_0}$ for some $j_0$. Further, let
$\beta=2 e_1^{(r+1)}+\widetilde \alpha$. Then, we have
$q_h(\alpha)\geq q_h(\beta)-1$.
\end{enumerate}
\end{lemma}
\pn Also the next linear algebraic lemma will be used in the proof
Theorem \ref{thm:3.2}. For a proof we refer to
\cite[Lemma 2.6]{Lo-0}.
\begin{lemma}\label{lem:1}
Suppose that $Q=(q_{ij})$ and $A$ are non-negative definite $N\times
N$ square matrices. Further, assume that, for some $m\in\N$, the
$m\times m$-submatrix $Q_0=(q_{ij})$, obtained erasing the last
$N-m$ rows and columns, is positive definite and $q_{ij}=0$ if
$\max\{i,j\}>m$. Then,
\begin{equation*}
\Tr(QA)\ge\lambda_{\min}(Q_0)\Tr (A_1),
\end{equation*}
\pn where $A_1$ is the submatrix obtained from $A$ by erasing the
last $N-m$ rows and columns.
%\end{enumerate}
\end{lemma}

\begin{proof}[Proof of Theorem $\ref{thm:3.2}$]
Throughout the proof, we simply write $c_k$ and $i^{(k)}_m$ instead
of $c_{k,r}$ and $i^{(k,r)}_m$.

Let $\ve\in ]0,1]$, $h,k\in\N$ with $h\le k\leq\kappa$, $f\in C_b^h(\R^N)$.
Further, we introduce the function
$v_{\ve}:[0,+\infty[\times\R^N\to\R$ defined by
\begin{equation*}
v_{\ve}(t,x)=\sum_{l=0}^k\langle {\mathscr
H}^{(l)}(t)D^lu_{\ve}(t,x),D^lu_{\ve}(t,x)\rangle,\qq\;\, (t,x)\in
[0,+\infty[\times\R^N, %\label{v-ve}
\end{equation*}
\pn where $u_{\ve}=T_{\ve}(\cdot)f$ and ${\mathscr H}^{(l)}(t)$
($l=0,\ldots,k$) are suitable symmetric matrices. Namely, ${\mathscr
H}^{(0)}=1$ and the matrices ${\mathscr H}^{(l)}(t)$
($l=1,\ldots,k$) are split into $c_l$ blocks $H_{m,p}^{(l)}(t)$
according to the splitting of the vector $D^lu_{\ve}$ introduced in
Subsection \ref{sub-notation}. We set
$s_p^{(l)}=\#\{\alpha\in\N_0^N:|\alpha|=i_p^{(l)}\}$. Now the
matrices $H_{m,p}^{(l)}(t)$ have the form
\begin{equation*}
H_{m,p}^{(l)}(t)=t^{q_h(i^{(l)}_m)+q_h(i^{(l)}_p)} \left\{
\begin{array}{ll}
\konst^{\eta_{p,p}^{(l)}}I_{s_p^{(l)}}, &{\rm if}~m=p,\\[2mm]
\konst^{\eta_{m,\ell (m)}^{(l)}}H^{(l)}_{m,\ell(m)}, &{\rm if}~m>c_{l-1}~{\rm and}~p=\ell(m),\\[2mm]
\konst^{\eta_{p,\ell (p)}^{(l)}}(H^{(l)}_{p,\ell(p)})^*, &{\rm if}~p>c_{l-1}~{\rm and}~m=\ell(p),\\[2mm]
0, &{\rm otherwise},
\end{array}
\right. %\label{matrix-Hj}
\end{equation*}
\pn for any $t>0$ and some constant $s_m^{(l)}\times
s_{\ell(m)}^{(l)}$-matrices $H_{m,\ell(m)}^{(l)}$ to be determined
later on just as well as the positive parameters $\konst>1$,
$\eta_{m,m}^{(l)}$ and $\eta_{m,\ell(m)}^{(l)}$. We put the
following requirements on these parameters:
\begin{equation}
\left\{
\begin{array}{cl}
\text{(a)} &\eta^{(l)}_{\ell(m),\ell(m)}+\eta^{(l)}_{m,m}>2\eta^{(l)}_{m,\ell(m)},\\[2mm]
\text{(b)}
&\konst^{\eta^{(l)}_{\ell(m),\ell(m)}+\eta^{(l)}_{m,m}-2\eta_{m,\ell(m)}^{(l)}}>2\|H_{m,\ell(m)}\|,
\label{cond-parameter}
\end{array}
\right.
\end{equation}
\pn for any $l=1,\ldots,k$ and any $m>c_{l-1}$. Conditions
\eqref{cond-parameter} guarantee that the matrix ${\mathscr
H}^{(l)}$ ($l=1,\ldots,k$) is positive definite for any $t>0$.
Moreover, we will also need to assume that $\eta^{(0)}_{1,1}=1$ and
\begin{equation}
\left\{\begin{array}{cll} \text{(a)} &
2\eta^{(l)}_{m,m}<\eta^{(l)}_{m,\ell(m)}, & c_{l-1}<m\le c_l,
\\[2mm]
\text{(b)} & \eta^{(l)}_{m,\ell(m)}<\eta^{(l)}_{p,\ell(p)}, &
c_{l-1}<p<m\le c_l,
\\[2mm]
\text{(c)}
&2\ds\max_{\mbox{\scriptsize$m=1,\ldots,c_l$}}\eta^{(l)}_{m,m}<\ds\min_{m=1,\ldots,c_{l-1}}\eta^{(l-1)}_{m,m}
=:\eta^{(l)},\q & %l=1,\ldots,k,
\\[2mm]
\text{(d)} & \eta^{(l)}_{m,\ell(m)}<\eta^{(l)}_{p,p}, &p\le
c_{l-1}<m,
\\[2mm]
\text{(e)} &2\eta_{m,\ell(m)}^{(l)}<
\eta^{(l)}_{\ell(m),\ell(\ell(m))},& \mbox{$m,\ell(m)>c_{l-1}$,}
 \label{fletcher-1}
\end{array}
\right.
\end{equation}
for any $l=1,\ldots,k$.
 For the moment, as it will be crucial
in the following, we assume that the constants $\eta^{(l)}_{m,m}$,
$\eta^{(l)}_{p,\ell(p)}$ ($l=1,\ldots,k$, $m=1,\ldots,c_l$,
$p=c_{l-1}+1,\ldots,c_l$), satisfying the conditions
\eqref{cond-parameter}(a) and \eqref{fletcher-1}, can be actually
determined. We will return to this point at the end and show that
this is actually the case.

\med

From Theorem \ref{thm:3.1} it follows that the function $v_\ve$ is
continuous on $[0,+\infty[\times \R^N$.  A straightforward
computation shows that it satisfies the Cauchy problem
\begin{equation*}
\left\{
\begin{array}{lll}
D_tv_{\ve}(t,x)={\mathscr A}_{\ve}v_{\ve}(t,x)+g_{\ve}(t,x),\q &t\in ]0,+\infty[, &x\in\R^N,\\[2mm]
v_{\ve}(0,x)=\ds\sum_{l=0}^h\langle {\mathscr
H}^{(l)}(0)D^lf(x),D^lf(x)\rangle,\q &&x\in\R^N,
\end{array}\right.
%\label{beatles}
\end{equation*}
\pn where the function $g_\ve$ is given by
\begin{align}
g_{\ve}=\,&-2\sum_{i,j=1}^N\sum_{l=0}^kq_{ij}^{\ve}\langle {\mathscr
H}^{(l)}D^lD_iu_{\ve},D^lD_ju_{\ve}\rangle
+2\sum_{l=1}^k\langle {\mathscr H}^{(l)}[D^l,\langle B\cdot,D\rangle ]u_{\ve},D^lu_{\ve}\rangle\no\\
&\, +\sum_{l=h+1}^k \langle \dot{{\mathscr
H}^{(l)}}D^lu_{\ve},D^lu_{\ve}\rangle
+2\sum_{l=1}^k\langle {\mathscr H}^{(l)}[D^l,\Tr (Q_{\ve}D^2)]u_{\ve},D^lu_{\ve}\rangle\no\\
&\,+2\sum_{l=1}^k\langle {\mathscr H}^{(l)}[D^l,\langle F,
D^1_1\rangle ]u_{\ve},D^lu_{\ve}\rangle :=\sum_{j=1}^5g_{j,\ve},
\label{funct-g}
\end{align}
the matrix $\dot{{\mathscr H}^{(l)}}$ is obtained by entrywise differentiating
the matrix ${\mathscr H}^{(l)}$ with respect to time,
and we have $D^0D_iu_{\ve}=D_iu_{\ve}$. Note also that the
commutators here are understood coordinatewise. When $h=k$ we agree
that the first sum in the second line of \eqref{funct-g} disappears.

\med We are going to prove that we can fix $\timeT$ small enough,
but \emph{independent of $\ve$}, such that $g_{\ve}\le 0$ in
$[0,\timeT]\times\R^N$. Proposition \ref{prop:maxprinc:0}(ii) then
will yield $v_{\ve}\le \sum_{l=0}^h\langle {\mathscr
H}^{(l)}(0)D^lf,D^lf\rangle$ in $[0,\timeT]\times\R^N$. In
particular, this implies that
\begin{align*}
\langle {\mathscr
H}^{(j)}(t)D^ju_{\ve}(t,x),D^ju_{\ve}(t,x)\rangle\le \hat
C\|f\|_{C^h_b(\R^N)}^2, \quad(t,x)\in
[0,\timeT]\times\R^N,\;\,j=1,\ldots,k,
\end{align*}
\pn for some positive constant $\hat C$. Since the matrices
${\mathscr H}^{(j)}(t)$ are positive definite for any $j$ and any
$t$ if we assume \eqref{cond-parameter}, we obtain that
\eqref{stimaapriori:ave} holds in the time interval $[0,\timeT]$.
The semigroup property allows then  to extend this estimate to any
compact time interval $J\subset ]0,+\infty[$.

\med\pn  We now  turn to the estimation of $g_\ve$.
\subsubsection*{Estimating the function $g_{1,\ve}$}
Lemma \ref{lem:1} and the ellipticity condition
\eqref{cond-ellipt-1} imply that
\begin{equation}
g_{1,\ve}(t)\le -2 \nu\sum_{j=1}^{p_0}\sum_{l=0}^k\langle {\mathscr
H}^{(l)}(t)D^lD_ju_{\ve}(t),D^lD_ju_{\ve}(t)\rangle,\qq\;\,t\in
]0,+\infty[. \label{g10}
\end{equation}
\pn This is a term of negative type and it will help us to control
(most of) the remaining terms in \eqref{funct-g}. More precisely,
the right-hand side of \eqref{g10} contains all the derivatives
$D^{\alpha}u_{\ve}$ of order less than or equal to $k+1$ such that,
if we split $\alpha=(\alpha_0,\ldots \alpha_r)$ (as explained in
Subsection \ref{sub-notation}), then $\|\alpha_0\|\neq 0$. So, we
miss all the derivatives of $u$ of the type $D^{\alpha}u$ with
$\|\alpha\|\le k+1$ and $\|\alpha_0\|=0$. We will recover these
latter derivatives from (a part of) the term $g_{2,\ve}$.

Using the very definition of the matrices ${\mathscr H}^{(l)}$
($l=1,\ldots,k$) we obtain
\begin{align*}
\sum_{j=1}^{p_0}\sum_{l=1}^k&\langle {\mathscr H}^{(l)}(t)D^lD_ju_{\ve}(t),D^lD_ju_{\ve}(t)\rangle\\
\ge\,&\sum_{j=1}^{p_0}\sum_{l=1}^k\sum_{m=1}^{c_l} \konst^{\eta_{m,m}^{(l)}}t^{2q_h(i_m^{(l)})}|D_m^lD_ju_{\ve}(t)|^2\no\\
&\:-2\sum_{j=1}^{p_0}\sum_{l=1}^k\sum_{m=c_{l-1}+1}^{c_l}\hskip-1em\konst^{\eta^{(l)}_{m,\ell(m)}}t^{q_h(i_m^{(l)})+q_h(i_{\ell(m)}^{(l)})}\|H_{\ell(m),m}^{(l)}
\||D_{\ell(m)}^lD_ju_{\ve}(t)||D_{m}^lD_ju_{\ve}(t)|,
\end{align*}
\pn for any $t>0$. Thanks to \eqref{cond-parameter}(a), we can fix
$\gamma_m^{(l)}$ and $\gamma_{\ell(m)}^{(l)}$ such that
\begin{equation*}
\gamma_m^{(l)}<\eta_{m,m}^{(l)},\qq\q
\gamma_{\ell(m)}^{(l)}<\eta_{\ell(m),\ell(m)}^{(l)},\qq\q
2\eta_{m,\ell(m)}^{(l)}=\gamma_{\ell(m)}^{(l)}+\gamma_m^{(l)}.
\end{equation*}
\pn By Young's inequality (we will use the same trick several times
in the sequel) and Lemma \ref{lem:propqh}(iii) we now infer that
\begin{align*}
&2\konst^{\eta^{(l)}_{m,\ell(m)}}t^{q_h(i_m^{(l)})+q_h(i_{\ell(m)}^{(l)})} \|H_{\ell(m),m}^{(l)}\|\cdot |D_{\ell(m)}^lD_ju_{\ve}(t)|\cdot|D_{m}^lD_ju_{\ve}(t)|\no\\
&\quad\le \konst^{\gamma_{\ell(m)}^{(l)}}t^{2q_h(i_{\ell(m)}^{(l)})}\|H_{\ell(m),m}^{(l)}\|\cdot |D_{\ell(m)}^lD_ju_{\ve}(t)|^2\\
&\quad\quad +\konst^{\gamma_{m}^{(l)}}t^{2q_h(i_{m}^{(l)})}\|H_{\ell(m),m}^{(l)}\|\cdot|D_{m}^lD_ju_{\ve}(t)|^2\no\\
&\quad=
o(\konst^{\eta_{\ell(m),\ell(m)}^{(l)}})t^{2q_h(i_{\ell(m)}^{(l)})}|D_{\ell(m)}^lD_ju_{\ve}(t)|^2
+o(\konst^{\eta_{m,m}^{(l)}})t^{2q_h(i_{m}^{(l)})}|D_{m}^lD_ju_{\ve}(t)|^2,
\end{align*}
for any $t>0$. Since $j\leq p_0$,  we obtain
\begin{align}
g_{1,\ve}(t)&\le\,-2\nu|D^1_1u_{\ve}(t)|^2
-2\nu\sum_{j=1}^{p_0}\sum_{l=1}^k\sum_{m=1}^{c_{l}}
\{\konst^{\eta_{m,m}^{(l)}}+o(\konst^{\eta_{m,m}^{(l)}})\}
t^{2q_h(i_{m}^{(l)})}|D_m^lD_ju_{\ve}(t)|^2\no\\
&\le\, -2\nu|D^1_1u_{\ve}(t)|^2
-2\nu\sum_{l=2}^{k+1}\sum_{m=1}^{c_{l-1}}
\{\konst^{\eta^{(l)}}+o(\konst^{\eta^{(l)}})\}
t^{(2q_h(i_{m}^{(l)})-1)^+}|D_m^lu_{\ve}(t)|^2, \label{g-1-0}
\end{align}
for any $t>0$, where (see \eqref{fletcher-1}(c)) we have
\begin{equation*}
\eta^{(l)}=\min_{m=1,\ldots,c_{l-1}}\eta^{(l-1)}_{m,m},\qq\;\,
l=2,\ldots,k+1. %\label{eta-l}
\end{equation*}

\subsubsection*{Estimating the term $g_{2,\ve}$}
Observe that
\begin{align*}
&g_{2,\ve}(t)=\,
2\sum_{l=1}^k\sum_{m=1}^{c_l}\konst^{\eta^{(l)}_{m,m}}t^{2q_h(i_m^{(l)})}\langle
[D^l_m,\langle B\cdot,D\rangle ]u_{\ve}(t),
D^l_mu_{\ve}(t)\rangle\\
&\,
+2\sum_{l=1}^k\sum_{m=c_{l-1}+1}^{c_l}\konst^{\eta^{(l)}_{m,\ell(m)}}t^{q_h(i_m^{(l)})+q_h(i_{\ell(m)}^{(l)})}
\langle H_{m,\ell(m)}^{(l)} [D^l_{\ell(m)},\langle B\cdot,D\rangle
]u_{\ve}(t), D^l_{m}u_{\ve}(t)\rangle\\
&\,
+2\sum_{l=1}^k\sum_{m=c_{l-1}+1}^{c_l}\konst^{\eta^{(l)}_{m,\ell(m)}}t^{q_h(i_m^{(l)})+q_h(i_{\ell(m)}^{(l)})}
\langle H_{\ell(m),m}^{(l)} [D^l_{m},\langle B\cdot,D\rangle
]u_{\ve}(t), D^l_{\ell(m)}u_{\ve}(t)\rangle,
\end{align*}
\pn for any $t>0$. By virtue of Lemma \ref{lemma-crucial} and a
straightforward computation, we can write
\begin{align}
&g_{2,\ve}(t)=\,
2\sum_{l=1}^k\sum_{m=c_{l-1}+1}^{c_l}\konst^{\eta^{(l)}_{m,\ell(m)}}t^{q_h(i_m^{(l)})+q_h(i_{\ell(m)}^{(l)})}
\langle H_{m,\ell(m)}^{(l)}{\mathscr J}_m^{(l)}D^l_mu_{\ve}(t),D^l_mu_{\ve}(t)\rangle\no\\
&+2\sum_{l=1}^k\sum_{m=c_{l-1}+1}^{c_l} \sum_{p\in
A_m^{(l)}\setminus\{m\}}\konst^{\eta^{(l)}_{m,\ell(m)}}t^{q_h(i_m^{(l)})+q_h(i_{\ell(m)}^{(l)})}
\langle {\mathscr M}_{m,p}^{(l)}D^l_pu_{\ve}(t),D^l_mu_{\ve}(t)\rangle\no\\
&+2\sum_{l=1}^k\sum_{m=1}^{c_l}\sum_{s\in B^{(l)}_m}
\konst^{\eta^{(l)}_{m,m}}t^{2q_h(i_m^{(l)})} \langle {\mathscr
N}^{(l)}_{m,s}D^l_su_{\ve}(t),D^l_mu_{\ve}(t)\rangle\no\\
&+2\sum_{l=1}^k\sum_{m=c_{l-1}+1}^{c_l}\sum_{s\in
B^{(l)}_m}\konst^{\eta^{(l)}_{m,\ell(m)}}t^{q_h(i_m^{(l)})+q_h(i_{\ell(m)}^{(l)})}
\langle H^{(l)}_{\ell(m),m}{\mathscr
N}^{(l)}_{m,s}D^l_su_{\ve}(t),D^l_{\ell(m)}u_{\ve}(t)\rangle,
\label{dom-sun}
\end{align}
\pn for any $t>0$, where $A_m^{(l)}$ is given by \eqref{Aml}, and
\begin{align*}
B_m^{(l)}=\Bigl\{s:
i_s^{(l)}=i_m^{(l)}-e_{j}^{(r+1)}+e_{h}^{(r+1)}&\;\,\mbox{for some }
j=0,\ldots, r \\&\mbox{\, such that } \alpha_j>0 \mbox{,
}h\le\min\{j+1,r\}\Bigr\},
\end{align*}
if $i^{(l)}_m=(\alpha_0,\ldots,\alpha_r)$, moreover the entries of
the matrices ${\mathscr M}_{m,p}^{(l)}$ ($m=c_{l-1}+1,\ldots,c_l$,
$p\in A^{(l)}_m$) and ${\mathscr N}_{m,s}^{(l)}$ ($m=
1,\ldots,c_l,s\in B_m^{(l)}$) depend (linearly) only on the entries
of the matrix $B$.

 Since by Lemma \ref{lemma-crucial} the matrix${\mathscr J}_m^{(l)}$ has maximum rank
(which equals the number of its co\-lumns) for any $m$ and any $l$,
we can fix the matrix $H^{(l)}_{m,\ell(m)}$ such that the matrix
$-H_{m,\ell(m)}^{(l)}{\mathscr J}_m^{(l)}-
(H_{m,\ell(m)}^{(l)}{\mathscr J}_m^{(l)})^*$ is positive definite.
We set
\begin{equation*}
\iota^{(k)}
=\min_{{l=1,\ldots,k}\atop{m=c_{l-1}+1,\ldots,c_l}}\lambda_{\min}\bigl(-H_{m,\ell(m)}^{(l)}{\mathscr J}_m^{(l)}
-(H_{m,\ell(m)}^{(l)}{\mathscr J}_m^{(l)})^*\bigr).
\end{equation*}
Hence, observing that, by properties (ii) and (v) in Lemma
\ref{lem:propqh} we have
$q_h(i_{\ell(m)}^{(l)})+q_h(i_m^{(l)})=(2q_h(i_m^{(l)})-1)^+$, we
can estimate
\begin{align}
&2\sum_{l=1}^k\sum_{m=c_{l-1}+1}^{c_l}\konst^{\eta^{(l)}_{m,\ell(m)}}t^{q_h(i_m^{(l)})+q_h(i_{\ell(m)}^{(l)})}
\langle H_{m,\ell(m)}^{(l)}{\mathscr J}_m^{(l)}D^l_mu_{\ve}(t),D^l_mu_{\ve}(t)\rangle\no\\
&\qquad
\le-\iota^{(k)}\sum_{l=1}^k\sum_{m=c_{l-1}+1}^{c_l}\konst^{\eta^{(l)}_{m,\ell(m)}}t^{(2q_h(i_m^{(l)})-1)^+}
|D^l_mu_{\ve}(t)|^2, \label{fletcher-7}
\end{align}
for any $t>0$. Now, we estimate the second and the third terms in
\eqref{dom-sun}. For this purpose, we first conclude from  Lemma
\ref{lem:propqh}(iv) the following: if $p\in A_m^{(l)}$, then
$q_h(i^{(l)}_p)-1\le q_h(i^{(l)}_{\ell(m)})$; and if $s\in
B_m^{(l)}$, then $q_h(i^{(l)}_s)-1\le q_h(i^{(l)}_m)$. It follows
that
\begin{align*}
2(q_h(i^{(l)}_m)+q_h(i^{(l)}_{\ell(m)}))&\ge (2q_h(i^{(l)}_m)-1)^+
+(2q_h(i^{(l)}_p)-1)^+,
\:\:\: m>c_{l-1},\:p\in A_m^{(l)},\\[1mm]
2q_h(i^{(l)}_m)&\ge (2q_h(i^{(l)}_s)-1)^+
+(2q_h(i^{(l)}_m)-1)^+,\:\:\: m\le c_l,\:s\in B_m^{(l)},
\end{align*}
or, equivalently,
\begin{align*}\refstepcounter{equation}
t^{q_h(i^{(l)}_m)+q_h(i^{(l)}_{\ell(m)})}&\le\:
t^{\frac{(2q_h(i^{(l)}_m)-1)^+}{2} +
\frac{(2q_h(i^{(l)}_p)-1)^+}{2}},&&\,m>c_{l-1},\;\,p\in A_m^{(l)},\tag{\theequation a}\label{forc-0}
\\[1mm]
t^{q_h(i^{(l)}_m)}&\le\:t^{\frac{(2q_h(i^{(l)}_s)-1)^+}{2}
+\frac{(2q_h(i^{(l)}_m)-1)^+}{2}},&&\,m\le c_l,\;\,s\in
B_m^{(l)}, \tag{\theequation b}\label{forc-1}
\end{align*}
for any $t\in ]0,1]$. Inequalities \eqref{forc-0} and
\eqref{forc-1} will allow us to split the powers of $t$ by using
Young's inequality in the estimate of the second and third terms in
\eqref{dom-sun}. Since we are looking for a right-neighborhood of
$t=0$ where $g_{\ve}$ is non-positive, without loss of generality we
can assume that $t\in ]0,1]$. We now consider several cases
according to the values of $p$ and $s$. We handle the different
cases for $p$, respectively for $s$, parallely. First, suppose that
$p,s\le c_{l-1}$. Using conditions \eqref{fletcher-1}(c) and
\eqref{fletcher-1}(d), we obtain
$2\eta^{(l)}_{m,\ell(m)}<\eta^{(l)}+\eta^{(l)}_{m,\ell(m)}$  and
$2\eta^{(l)}_{m,m}<\eta^{(l)}+\beta^{(l)}_{m}$,  where
$\beta^{(l)}_m=\eta^{(l)}$ if $m\le c_{l-1}$ and
$\beta^{(l)}_m=\eta^{(l)}_{m,\ell(m)}$ otherwise. From this,
\eqref{forc-0}, \eqref{forc-1} and Young's inequality we can
conclude
\begin{align*}\refstepcounter{equation}
&\left
|\konst^{\eta^{(l)}_{m,\ell(m)}}t^{q_h(i_m^{(l)})+q_h(i_{\ell(m)}^{(l)})}
\langle {\mathscr M}_{m,p}^{(l)}D^l_pu_{\ve}(t),D^l_mu_{\ve}(t)\rangle\right |\\
&\quad\le
o(\konst^{\eta^{(l)}})t^{(2q_h(i_p^{(l)})-1)^+}|D^l_pu_{\ve}(t)|^2
+o(\konst^{\eta^{(l)}_{m,\ell(m)}})t^{(2q_h(i_m^{(l)})-1)^+}|D^l_m u_{\ve}(t)|^2,\label{883-3}\tag{\theequation a}\\[2.5mm]
& \left |\konst^{\eta^{(l)}_{m,m}}t^{2q_h(i_m^{(l)})}
\langle {\mathscr N}^{(l)}_{m,s}D^l_su_{\ve}(t),D^l_mu_{\ve}(t)\rangle\right |\\
&\quad\le
o(\konst^{\eta^{(l)}})t^{(2q_h(i_s^{(l)})-1)^+}|D^l_su_{\ve}(t)|^2
+o(\konst^{\beta^{(l)}_m})t^{(2q_h(i_m^{(l)})-1)^+}|D^l_m
u_{\ve}(t)|^2,\tag{\theequation b}
\end{align*}
for any $t\in ]0,1]$. The case when $m\le c_{l-1}$ and $s> c_{l-1}$
can be addressed similarly, taking now \eqref{fletcher-1}(c) into
account. We thus obtain
\begin{align}
&\left |\konst^{\eta^{(l)}_{m,m}}t^{2q_h(i_m^{(l)})}
\langle {\mathscr N}_{m,s}^{(l)}D^l_su_{\ve}(t),D^l_mu_{\ve}(t)\rangle\right |\no\\
&\quad\le
o(\konst^{\eta^{(l)}_{s,\ell(s)}})t^{(2q_h(i_s^{(l)})-1)^+}|D^l_s
u_{\ve}(t)|^2 +o(\konst^{\eta^{(l)}})t^{(2q_h(i_m^{(l)})-1)^+}|D^l_m
u_{\ve}(t)|^2, \label{fletcher-11}
\end{align}
for any $t\in ]0,1]$. Finally, we consider the case when
$m,p,s>c_{l-1}$. Observe that $p< m$ for $p\in A^{(l)}_m$, $p\neq
m$, and hence from condition \eqref{fletcher-1}(b) we obtain
$2\eta^{(l)}_{m,\ell(m)}<\eta^{(l)}_{m,\ell(m)}+\eta^{(l)}_{p,\ell(p)}$,
whereas we also have
$2\eta^{(l)}_{m,m}<\eta^{(l)}_{m,\ell(m)}+\eta^{(l)}_{s,\ell(s)}$ by
\eqref{fletcher-1}(a). These yield
\begin{align*}\refstepcounter{equation}
\label{883-4} &\left
|\konst^{\eta^{(l)}_{m,\ell(m)}}t^{q_h(i_{\ell(m)}^{(l)})+q_h(i_m^{(l)})}
\langle {\mathscr M}_{m,p}^{(l)}D^l_pu_{\ve}(t),D^l_mu_{\ve}(t)\rangle\right |\no\\
&\quad\le
o(\konst^{\eta^{(l)}_{p,\ell(p)}})t^{(2q_h(i_p^{(l)})-1)^+}|D^l_p
u_{\ve}(t)|^2
+o(\konst^{\eta^{(l)}_{m,\ell(m)}})t^{(2q_h(i_m^{(l)})-1)^+}|D^l_m
u_{\ve}(t)|^2,\tag{\theequation a}
\\[2.5mm]
&\left |\konst^{\eta^{(l)}_{m,m}}t^{2q_h(i_m^{(l)})}
\langle {\mathscr N}_{m,s}^{(l)}D^l_su_{\ve}(t),D^l_mu_{\ve}(t)\rangle\right |\no\\
&\quad\le
o(\konst^{\eta^{(l)}_{s,\ell(s)}})t^{(2q_h(i_s^{(l)})-1)^+}|D^l_su_{\ve}(t)|^2
+o(\konst^{\eta^{(l)}_{m,\ell(m)}})t^{(2q_h(i_m^{(l)})-1)^+}|D^l_m
u_{\ve}(t)|^2, \tag{\theequation b}\label{fletcher-9}
\end{align*}
for any $t\in ]0,1]$.  We now estimate the fourth term in
\eqref{dom-sun}. First, notice that
\begin{equation*}
2(q_h(i_m^{(l)})+q_h(i_{\ell(m)}^{(l)}))\geq (2q_h(i_s^{(l)})-1)^+
+(2q_h(i_{\ell(m)}^{(l)})-1)^+,
\end{equation*}
by Lemma \ref{lem:propqh}(iv), or, equivalently
\begin{equation*}
t^{q_h(i^{(l)}_m)+q_h(i^{(l)}_{\ell(m)})}\le
t^{\frac{(2q_h(i^{(l)}_s)-1)^+}{2} +
\frac{(2q_h(i^{(l)}_{\ell(m)})-1)^+}{2}}, \qq\;\,t\in
]0,1],\;\,m>c_{l-1},\;\,s\in B_m^{(l)}.
\end{equation*}
If $\ell(m)>c_{l-1}$, then because of $2\eta_{m,\ell(m)}^{(l)}<
\eta^{(l)}_{\ell(m),\ell(\ell(m))}$ (see \eqref{fletcher-1}(e)) we
can write
\begin{align}
&\left
|\konst^{\eta^{(l)}_{m,\ell(m)}}t^{q_h(i_m^{(l)})+q_h(i_{\ell(m)}^{(l)})}
\langle H_{\ell(m),m}^{(l)}{\mathscr N}_{m,s}^{(l)}D^l_su_{\ve}(t),D^l_{\ell(m)}u_{\ve}(t)\rangle\right |\no\\
&\quad\le
o(\konst^{\beta^{(l)}_{s}})t^{(2q_h(i_s^{(l)})-1)^+}|D^l_su_{\ve}(t)|^2
+o(\konst^{\eta^{(l)}_{\ell(m),\ell(\ell(m))}})t^{(2q_h(i_{\ell(m)}^{(l)})-1)^+}|D^l_{\ell(m)}
u_{\ve}(t)|^2,\label{eq:sellm}
\end{align}
for all $t\in]0,1]$, where again $\beta^{(l)}_{s}=\eta^{(l)}$ if
$s\le c_{l-1}$ and $\beta^{(l)}_{s}=\eta^{(l)}_{s,\ell(s)}$
otherwise. On the other hand, if $\ell(m)\leq c_{l-1}$, then  we
have $2\eta_{m,\ell(m)}^{(l)}< \eta^{(l)}$ (see
\eqref{fletcher-1}(c) and \eqref{fletcher-1}(d)). Thus, we may
conclude that
\begin{align}
&\left
|\konst^{\eta^{(l)}_{m,\ell(m)}}t^{q_h(i_m^{(l)})+q_h(i_{\ell(m)}^{(l)})}
\langle H_{\ell(m),m}^{(l)}{\mathscr N}_{m,s}^{(l)}D^l_su_{\ve}(t),D^l_{\ell(m)}u_{\ve}(t)\rangle\right |\no\\
&\quad\le
o(\konst^{\beta^{(l)}_{s}})t^{(2q_h(i_s^{(l)})-1)^+}|D^l_su_{\ve}(t)|^2
+o(\konst^{\eta^{(l)}})t^{(2q_h(i_{\ell(m)}^{(l)})-1)^+}|D^l_{\ell(m)}
u_{\ve}(t)|^2,\label{eq:sellm2}
\end{align}
holds for $t\in ]0,1]$. Therefore, by summing up
\eqref{dom-sun},\eqref{fletcher-7}, (3.20)--\eqref{eq:sellm2}, we can deduce that
\begin{align}
g_{2,\ve}(t)\le&\,
-\iota^{(k)}\sum_{l=1}^k\sum_{m=c_{l-1}+1}^{c_l}\Bigl(\konst^{\eta^{(l)}_{m,\ell(m)}}+o(\konst^{\eta^{(l)}_{m,\ell(m)}})\Bigr)
t^{(2q_h(i_m^{(l)})-1)^+}|D^l_m u_{\ve}(t)|^2\no\\
&\,+\sum_{l=1}^k\sum_{m=1}^{c_{l-1}}
o(\konst^{\eta^{(l)}})t^{(2q_h(i_m^{(l)})-1)^+}|D^l_mu_{\ve}(t)|^2
\label{estim-good-term}
\end{align}
holds for any $t\in ]0,1]$.

\subsubsection*{Estimating the term $g_{3,\ve}$}
As it has been already remarked, this term occurs only if $h<k$. We
begin by estimating the term
$t^{q_h(i_m^{(l)})+q_h(i_{\ell(m)}^{(l)})-1}\konst^{\eta^{(l)}_{m,\ell(m)}}|D^l_{\ell(m)}u_{\ve}(t)|\cdot|D^l_mu_{\ve}(t)|$,
when $l>h$ and $m>c_{l-1}$; note that, by Lemma \ref{lem:propqh}(v),
$q_h(i_m^{(l)})+q_h(i_{\ell(m)}^{(l)})-1>0$. By
\eqref{fletcher-1}(b), \eqref{fletcher-1}(c) and
\eqref{fletcher-1}(d), we have that
$2\eta^{(l)}_{m,\ell(m)}<\eta^{(l)}+\eta^{(l)}_{m,\ell(m)}$, if
$\ell(m)\le c_{l-1}$, and
$2\eta^{(l)}_{m,\ell(m)}<\eta^{(l)}_{m,\ell(m)}+\eta_{\ell(m),\ell(\ell(m))}^{(l)}$
if $\ell(m)>c_{l-1}$. Hence, we can estimate
\begin{align}
&t^{q_h(i_m^{(l)})+q_h(i_{\ell(m)}^{(l)})-1}\konst^{\eta^{(l)}_{m,\ell(m)}}|D^l_{\ell(m)}u_{\ve}(t)|\cdot|D^l_mu_{\ve}(t)|\no\\
&\quad\le
o(\konst^{\beta^{(l)}_{\ell(m)}})t^{2q_h(i^{(l)}_{\ell(m)})-1}|D^l_{\ell(m)}u_{\ve}(t)|^2
+o(\konst^{\eta^{(l)}_{m,\ell(m)}})t^{2q_h(i^{(l)}_{m})-1}|D^l_mu_{\ve}(t)|^2,
\label{spino-1}
\end{align}
\pn for any $t>0$, where $\beta^{(l)}_{\ell(m)}$ is as above. From
\eqref{spino-1} and taking condition \eqref{fletcher-1}(a) into
account, we now get
\begin{align}
\no g_{3,\ve}(t) =\;&2\sum_{l=h+1}^k
\sum_{m=1}^{c_l}q_h(i_m^{(l)})\konst^{\eta_{m,m}^{(l)}}t^{2q_h(i_m^{(l)})-1}|D^l_mu_{\ve}(t)|^2\\
&\:+2\sum_{l=h+1}^k
\sum_{m=c_{l-1}+1}^{c_l}\bigl(q_h(i_m^{(l)})+q_h(i_{\ell(m)}^{(l)})\bigr)\konst^{\eta_{m,\ell(m)}^{(l)}}t^{q_h(i_m^{(l)})+q_h(i_{\ell(m)}^{(l)})-1}\no\\
&\quad\quad\quad\quad\quad\quad\quad\quad\quad\times
\langle H_{m,\ell(m)}D^l_{\ell(m)}u_{\ve}(t),D^l_{m}u_{\ve}(t)\rangle\no\\
\leq&
\sum_{l=h+1}^k\sum_{m=1}^{c_{l-1}}o(\konst^{\eta^{(l)}})t^{2q_h(i_m^{(l)})-1}|D^l_mu_{\ve}(t)|^2
\label{Hl'}\\
&+\sum_{l=h+1}^k\sum_{m=c_{l-1}+1}^{c_l}o(\konst^{\eta^{(l)}_{m,\ell(m)}})t^{2q_h(i_m^{(l)})-1}|D^l_mu_{\ve}(t)|^2,\no
\end{align}
for any $t>0$.
\subsubsection*{Estimating the terms $g_{4,\ve}$ and $g_{5,\ve}$} We begin with $g_{4,\ve}$, the case of $g_{5,\ve}$ being completely analogous.
Let us observe that we have
\begin{equation}\label{eq:commTrQ}
[D^l_m,\Tr(Q_{\ve}D^2)]u_{\ve}=\sum_{i,j=1}^N[D^l_m,
q_{ij}^{\ve}D_{ij}]u_{\ve} =\sum_{i,j=1}^{p_0}[D^l_m,
q_{ij}D_{ij}]u_{\ve} =\sum_{z=2}^{l+1}{\mathscr
P}^{(l,z)}_mD^zu_{\ve},
\end{equation}
for any $m=1,\ldots,c_l$ and some matrices  ${\mathscr P}^{(l,z)}_m$
whose entries linearly depend only on the derivatives (of order at
least $1$ and at most $l$) of the functions $q_{ij}$
$(i,j=1,\ldots,p_0)$. In particular, these matrices are independent
of $\ve$. Moreover, if we split the matrices ${\mathscr
P}^{(l,z)}_m$ into sub-blocks $P^{(l,z)}_{m,s}$ ($m=1,\dots,c_l$,
$s=1,\dots,c_z$) according to the rule in Subsection
\ref{sub-notation}, it follows that $P^{(l,z)}_{m,s}=0$ if
$s>c_{z-1}$. To see the above let $\alpha\in \N_0^{N}$ be a
multi-index with $\|\alpha\|=l$,
$|\alpha|=(\alpha_0,\alpha_1,\dots,\alpha_r)\in \N_0^{r+1}$. From
the second equality in \eqref{eq:commTrQ} (which is immediate if we
recall that $q_{ij}^{\ve}$ is constant if at least one of $i$
and $j$ are greater than $p_0$) it follows that the terms appearing
in \eqref{eq:commTrQ} and obtained from
$[D^\alpha,\Tr(Q_{\ve}D^2)]u_{\ve}$ are $D^\beta u_\ve$ with
coefficients in front depending on the derivatives of $q_{ij}$,
$i,j\leq p_0$, and with some $\beta\in \N_0^{N}$ such that
$|\beta|=(\beta_0,\ldots,\beta_r)$ with $1\le\beta_0\le\alpha_0+2$,
$\beta_j\le\alpha_j$ for any $j=1,\ldots,r$ and
$\|\beta\|\le\|\alpha\|+1$. In particular, since $\beta_0>1$, then
$|\beta|=i_s^{(z)}$ with $s\leq c_{z-1}$, where $z=\|\beta\|$.
Denote by $C^{(l)}_m$ the indices $i^{(z)}_s$ obtained in this way
from multi-indices $\alpha$ with $|\alpha|=i^{(l)}_m$. By Lemma
\ref{lem:propqh}(vii) we have $q_h(i^{(l)}_m)\geq q_h(i_s^{(z)})-1$
for $i_s^{(z)}\in C_m^{(l)}$, and $q_h(i^{(l)}_{\ell(m)})\geq
q_h(i_s^{(z)})-1$ for $i_s^{(z)}\in C_{\ell(m)}^{(l)}$. These
inequalities will allow us to split the powers of $t$ by using
Young's inequality. Hence we can write
\begin{align}
&g_{4,\ve}(t)=\sum_{l=1}^k\sum_{m=1}^{c_l}\sum_{z=2}^{l+1}\konst^{\eta^{(l)}_{m,m}}t^{2q_h(i_m^{(l)})}
\langle {\mathscr P}^{(l,z)}_mD^zu_{\ve}(t),D^l_mu_{\ve}(t)\rangle\no\\
&\,
+\sum_{l=1}^k\sum_{m=c_{l-1}+1}^{c_l}\sum_{z=2}^{l+1}\konst^{\eta^{(l)}_{m,\ell(m)}}t^{q_h(i_m^{(l)})+q_h(i_{\ell(m)}^{(l)})}
\langle H_{m,\ell(m)}^{(l)}{\mathscr P}^{(l,z)}_{\ell(m)}
D^zu_{\ve}(t),D^l_mu_{\ve}(t)\rangle\no\\
&\,
+\sum_{l=1}^k\sum_{m=c_{l-1}+1}^{c_l}\sum_{z=2}^{l+1}\konst^{\eta^{(l)}_{m,\ell(m)}}t^{q_h(i_m^{(l)})+q_h(i_{\ell(m)}^{(l)})}
\langle H_{\ell(m),m}^{(l)}{\mathscr P}^{(l,z)}_{m}
D^zu_{\ve}(t),D^l_{\ell(m)}u_{\ve}(t)\rangle,\label{eq:g4e}
\end{align}
for any $t>0$. In the following, we assume again $t\in ]0,1]$ and
denote by $C$ positive constants, independent of $\nu$, $t$ and
$\konst$, which may vary from line to line. We can estimate the
summands in the first term in the right-hand side of \eqref{eq:g4e}
as follows:
\begin{align}
&\sum_{z=2}^{l+1}\konst^{\eta^{(l)}_{m,m}}t^{2q_h(i_m^{(l)})}\big
|\langle {\mathscr
P}^{(l,z)}_{m}D^zu_{\ve}(t),D^l_mu_{\ve}(t)\rangle\big |\no\\
&\quad\le\, C\sqrt{\nu} \konst^{\eta^{(l)}_{m,m}}t^{2q_h(i_m^{(l)})}\sum_{z=2}^l\sum_{s: i^{(z)}_s\in C^{(l)}_m}|D^z_su_{\ve}(t)|\cdot |D^l_mu_{\ve}(t)|\no\\
&\quad\quad\,+C\sqrt{\nu}
\konst^{\eta^{(l)}_{m,m}}t^{2q_h(i_m^{(l)})}\sum_{s: i^{(l+1)}_s\in
C^{(l)}_m}|D^{l+1}_su_{\ve}(t)|\cdot |D^l_mu_{\ve}(t)|,
\label{darm-2}
\end{align}
for any $m\le c_l$. Let us consider the first term in the right hand side of
\eqref{darm-2}. Since we have $z\le l$ here, we can use
\eqref{fletcher-1}(c) and  estimate
\begin{align*}
&\sqrt{\nu} \konst^{\eta^{(l)}_{m,m}}t^{2q_h(i_m^{(l)})}\sum_{z=2}^l
\sum_{s: i^{(z)}_s\in C^{(l)}_m}|D^z_su_{\ve}(t)|\cdot |D^l_mu_{\ve}(t)|\no\\
&\le C\bigg (o(\konst^{\beta^{(l)}_m})t^{(2q_h(i_m^{(l)})-1)^+}
|D^l_mu_{\ve}(t)|^2\hskip-0.2em+\hskip-0.2em\sum_{z=2}^l\sum_{s=1}^{c_{z-1}}o(\konst^{\eta^{(z)}})t^{(2q_h(i_s^{(z)})-1)^+}\nu
|D^z_su_{\ve}(t)|^2\bigg ),
\end{align*}
where $\beta^{(l)}_m=\eta^{(l)}$ if $m\le c_{l-1}$ and
$\beta^{(l)}_m=\eta^{(l)}_{m,\ell(m)}$ otherwise
On the other hand, in the case when $z=l+1$ we use conditions
\eqref{fletcher-1}(a), \eqref{fletcher-1}(c)  to obtain
$2\eta^{(l)}_{m,m}<\beta^{(l)}_{m}+\eta^{(l+1)}$. Thus, we can
estimate by using Young's inequality
\begin{align*}
&C\sqrt{\nu} \konst^{\eta^{(l)}_{m,m}}t^{2q_h(i_m^{(l)})}
\sum_{s: i^{(l+1)}_s\in C^{(l)}_m}|D^{l+1}_su_{\ve}(t)|\cdot |D^l_mu_{\ve}(t)|\no\\
&\le C\bigg (o(\konst^{\beta^{(l)}_m})t^{(2q_h(i_m^{(l)})-1)^+}
|D^l_mu_{\ve}(t)|^2\hskip-0.4em+\hskip-0.2em\sum_{s=1}^{c_{l}}o(\konst^{\eta^{(l+1)}})t^{(2q_h(i_s^{(l+1)})-1)^+}
\nu|D^{l+1}_su_{\ve}(t)|^2\bigg ).
\end{align*}
The summands in the second and the third terms in \eqref{eq:g4e} can
be estimated likewise. By \eqref{fletcher-1}(c),
\eqref{fletcher-1}(d) and \eqref{fletcher-1}(e) we have
$\eta^{(l)}_{m,\ell(m)}<\eta^{(l+1)}$,
$2\eta^{(l)}_{m,\ell(m)}<\eta^{(l)}$ and
$2\eta^{(l)}_{m,\ell(m)}<2\eta^{(l)}_{\ell(m),\ell(\ell(m))}$ (this
latter for $\ell(m)>c_{l-1}$), so by Young's inequality we can
deduce
\begin{align*}
&\sum_{z=2}^{l+1}\konst^{\eta^{(l)}_{m,\ell(m)}}t^{q_h(i_m^{(l)})+q_h(i_{\ell(m)}^{(l)})}
\langle H_{m,\ell(m)}^{(l)}{\mathscr P}^{(l,z)}_{\ell(m)}
D^zu_{\ve}(t),D^l_mu_{\ve}(t)\rangle\\
&\:\leq
C\sqrt{\nu}\,\konst^{\eta^{(l)}_{m,\ell(m)}}t^{q_h(i_m^{(l)})+q_h(i_{\ell(m)}^{(l)})}
\sum_{z=2}^{l+1}\sum_{s: i^{(z)}_s\in C^{(l)}_{\ell(m)}}|D^z_su_{\ve}(t)|\cdot |D^l_mu_{\ve}(t)|\\
&\:\leq o(\konst^{\eta^{(l)}_{m,\ell(m)}})t^{(2q_h(i_m^{(l)})-1)^+}
|D^l_mu_{\ve}(t)|^2+\sum_{z=2}^{l+1}\sum_{s=1}^{c_{z-1}}o(\konst^{\eta^{(z)}})t^{(2q_h(i_s^{(z)})-1)^+}\nu
|D^z_su_{\ve}(t)|^2 \intertext{and}
&\sum_{z=2}^{l+1}\konst^{\eta^{(l)}_{m,\ell(m)}}t^{q_h(i_m^{(l)})+q_h(i_{\ell(m)}^{(l)})}
\langle H_{\ell(m),m}^{(l)}{\mathscr P}^{(l,z)}_{m}
D^zu_{\ve}(t),D^l_{\ell(m)}u_{\ve}(t)\rangle\\
&\:\leq
C\sqrt{\nu}\,\konst^{\eta^{(l)}_{m,\ell(m)}}t^{q_h(i_m^{(l)})+q_h(i_{\ell(m)}^{(l)})}\sum_{z=2}^{l+1}
\sum_{s: i^{(z)}_s\in C^{(l)}_{m}}|D^z_su_{\ve}(t)|\cdot |D^l_{\ell(m)}u_{\ve}(t)|\\
&\:\leq
o(\konst^{\beta^{(l)}_{\ell(m)}})t^{(2q_h(i_{\ell(m)}^{(l)})-1)^+}
|D^l_{\ell(m)}u_{\ve}(t)|^2\hskip-0.2em+\hskip-0.2em\sum_{z=2}^{l+1}\sum_{s=1}^{c_{z-1}}o(\konst^{\eta^{(z)}})t^{(2q_h(i_s^{(z)})-1)^+}\nu
|D^z_su_{\ve}(t)|^2,
\end{align*}
for $t\in ]0,1]$, $m>c_{l-1}$, where, again,
$\beta^{(l)}_{\ell(m)}=\eta^{(l)}$, if $\ell(m)\le c_{l-1}$, and
$\beta^{(l)}_{\ell(m)}=\eta^{(l)}_{\ell(m),\ell(\ell(m))}$,
otherwise. \pn By putting everything together, we obtain
\begin{align}
g_{4,\ve}(t)&\leq\nu\sum_{l=1}^{k+1}\sum_{m=1}^{c_{l-1}}o(\konst^{\eta^{(l)}})
t^{(2q(i^{(l)}_m)-1)^+}|D^l_mu_{\ve}(t)|^2\no\\
&\quad+
\sum_{l=1}^{k}\sum_{m=c_{l-1}+1}^{c_l}o(\konst^{\eta^{(l)}_{m,\ell(m)}})
t^{(2q(i^{(l)}_m)-1)^+}|D^l_mu_{\ve}(t)|^2,\qq\;\,t\in ]0,1].
\label{fletcher}
\end{align}

\med\pn Just in the same way, we can estimate the function
$g_{5,\ve}$ and get
\begin{align}
g_{5,\ve}(t)\le\,\,&
\nu\sum_{l=1}^{k}\sum_{m=1}^{c_{l-1}}o(\konst^{\eta^{(l)}})t^{(2q_h(i^{(l)}_m)-1)^+}|D^l_mu_{\ve}(t)|^2\no\\
&
+\sum_{l=1}^{k}\sum_{m=c_{l-1}+1}^{c_l}o(\konst^{\eta^{(l)}_{m,\ell(m)}})
t^{(2q_h(i^{(l)}_m)-1)^+}|D^l_mu_{\ve}(t)|^2, \qq\;\,t\in ]0,1].
\label{last-estim}
\end{align}

\subsubsection*{Final estimate of the function $g_{\ve}$}
Now, collecting \eqref{g-1-0}, \eqref{estim-good-term}, \eqref{Hl'},
\eqref{fletcher} and \eqref{last-estim} together, we get
\begin{align}
g_{\ve}(t)\le\;& -2\nu\sum_{l=1}^k\sum_{m=1}^{c_{l-1}}
\{\konst^{\eta^{(l)}}+o(\konst^{\eta^{(l)}})\}
t^{(2q_h(i^{(l)}_m)-1)^+}|D^l_mu_{\ve}(t)|^2\no\\
&-\iota^{(k)}
\sum_{l=1}^k\sum_{m=c_{l-1}+1}^{c_l}\{\konst^{\eta^{(l)}_{m,\ell(m)}}+o(\konst^{\eta^{(l)}_{m,\ell(m)}})\}
t^{(2q_h(i^{(l)}_m)-1)^+}
|D^l_mu_{\ve}(t)|^2\no\\
&
-2\nu\sum_{m=1}^{c_k}\{\konst^{\eta^{(k+1)}}+o(\konst^{\eta^{(k+1)}})\}t^{(2q_h(i^{(k+1)}_m)-1)^+}|D^{k+1}_mu_{\ve}(t)|^2,
\label{estim-gve}
\end{align}
\pn for any $t\in ]0,1]$. \med\pn

If we now fix the parameter $\konst$ sufficiently large, condition
\eqref{cond-parameter}(b) is satisfied and, for an even larger
constant $\konst$, the terms in right-hand side of \eqref{estim-gve}
will be negative, provided that one can choose the parameters
$\eta^{(l)}_{m,m}$ and $\eta^{(l)}_{p,\ell(p)}$ such that conditions
\eqref{cond-parameter}(a) and \eqref{fletcher-1}, which we have used
in the previous estimates, are satisfied. Hence, the last part of
the proof is devoted to address this point, and leads us to the
conclusion of the proof.

\subsubsection*{Choice of the parameters $\eta^{(l)}_{m,m}$ and $\eta^{(l)}_{p,\ell(p)}$} We now
show that we can fix all the constants $\eta^{(l)}_{m,m}$ and
$\eta^{(l)}_{p,\ell(p)}$ such that the conditions
\eqref{cond-parameter}(a) and \eqref{fletcher-1} are satisfied. For
each $l=1,\dots,k+1$ we will take a positive, strictly decreasing
sequence $\{a^{(l)}_n\}$ with $a^{(l)}_n<1$, and set
$\eta^{(l)}_{n,m}:=a^{(l)}_{n}a^{(l)}_{m}$. By this restriction,
condition \eqref{cond-parameter}(a) will be satisfied. So from now
 we concentrate only on \eqref{fletcher-1}. Notice that for
$c_{l-1}<m<p$ we have $\ell(m)< \ell(p)$, so \eqref{fletcher-1}(b)
is automatically satisfied by monotonicity. Note also that for such
$m$ we have $\ell(m)<m$. Hence, if we choose
$a^{(l)}_{n+1}<\frac{1}{2}a^{(l)}_{n}$ for all $n$, also
\eqref{fletcher-1}(a) and \eqref{fletcher-1}(e) will be satisfied.
Now we turn to the actual construction keeping all the above
requirements on $a_n^{(l)}$.
 First we choose $a^{(1)}_n$ for $n=1,\dots,c_1=r+1$ according to the
 above. Then, for $l\geq 2$ we proceed inductively, first taking
 $a^{(l)}_1<\frac{1}{2}\sqrt{\eta^{(l)}}$ and, then, choosing $a_n^{(l)}$ ($n=2,\ldots,c_l$) satisfying
$a_n^{(l)}<\frac{1}{2}a_{n-1}^{(l)}$ for any $n$ and with
$a_{c_{l-1}+1}^{(l)}<a_{c_{l-1}}^{(l)}a_{c_{l-1}}^{(l)}$. The first
condition implies that \eqref{fletcher-1}(c) is satisfied. On the
other hand, since $a_n^{(l)}<1$, the latter condition implies
\eqref{fletcher-1}(d). Indeed, for $m>c_{l-1}$ and $p\leq c_{l-1}$
we have, $a_ma_{\ell(m)}\le
a_{c_{l-1}+1}<a_{c_{l-1}}a_{c_{l-1}}<a_p^2$.
\end{proof}

\begin{rem}Notice that the above proof works also for other
functions $q:\N_0^{r+1}\to\R_+$ replacing $q_h$, as long as this
function $q$ has the properties similar to that of $q_h$ as listed
in Lemma \ref{lem:propqh}.
\end{rem}

%%%%%%%%%%%%%%%%%%%%%%%%%%%%%%%%%%%%%%%%%%%%%%%%%%%%%%%%%%%%%%%%%%%%%%%
%%                                                                   %%
%%             CONSTRUCTION                                          %%
%%                                                                   %%
%%%%%%%%%%%%%%%%%%%%%%%%%%%%%%%%%%%%%%%%%%%%%%%%%%%%%%%%%%%%%%%%%%%%%%%

\section{Construction of the semigroup}
\label{sec-sem}

 In this section we prove
that, for any $f\in C_b(\R^N)$, the Cauchy problem
\begin{equation}
\left\{
\begin{array}{lll}
D_tu(t,x)={\mathscr A}u(t,x), \q &t>0, &x\in\R^N,\\[2mm]
u(0,x)=f(x), &&x\in\R^N,
\end{array}
\right. \label{HCP}
\end{equation}
\pn admits a unique classical solution $u$, and consequently we can
associate a semigroup of bounded operators in $C_b(\R^N)$ with the
operator ${\mathscr A}$.

\begin{theorem}\label{thm:constr}
Suppose that Hypotheses \ref{ipos-1} are satisfied. Then the
following assertions hold:
\begin{enumerate}[\rm(i)]
\item
For any $f\in C_b(\R^N)$ there exists a unique classical solution
$u$ to problem \eqref{HCP}.
\item
The family $\{T(t)\}$, defined by $T(t)f:=u(t,\cdot)$ for any $t>0$,
where $u$ is the classical solution to problem \eqref{HCP}
corresponding to the initial value $f$, is a positivity preserving
 semigroup of linear contractions in $C_b(\R^N)$.
\item
If $f\in C_c(\R^N)$, then $T(t)f$ converges to $f$, as $t\to 0^+$,
uniformly in $\R^N$.
\item
For any $f\in C_b(\R^N)$ and any multi-index $\alpha\in\N_0^N$, with
$\|\alpha\|\le\kappa-1$, the derivative $D^{\alpha}T(\cdot)f$ exists
in the classical sense in $]0,+\infty[\times \R^N$ and it is a
continuous function. Moreover, there exists a positive constant $C$,
depending only on $\omega$, $h$ and $\|\alpha\|$ such that, for any
$f\in C^h_b(\R^N)$ and any $\alpha$ as above, we have
\begin{equation}
\|D^{\alpha}T(t)f\|_{C_b(\R^N)}\le Ct^{-q_h(|\alpha|)}e^{\omega
t}\|f\|_{C^h_b(\R^N)},\qq\;\,t\in ]0,+\infty[.
\label{stimaapriori:a0}
\end{equation}
\end{enumerate}
\end{theorem}

\begin{proof} (i) First of all, notice that uniqueness follows immediately form the maximum principle,
Proposition \ref{prop:maxprinc:0}. Throughout the proof, we denote by $C$ positive constants,
independent of $\ve\in ]0,1[$, which may vary from line to line. As
a first step, we show that, for any $\omega>0$, there exists a
positive constant $\hat C=\hat C(\omega)$, independent of $\ve$,
such that
\begin{equation}
\|D^{\alpha}T_{\ve}(t)f\|_{\infty}\le \hat
Ct^{-q_h(|\alpha|)}e^{\omega t} \|f\|_{C^h_b(\R^N)},\qq\;\,t\in
]0,+\infty[,\;\,\ve\in ]0,1], \label{stimaapriori:ave-bis}
\end{equation}
for any $f\in C^h_b(\R^N)$ and any $\|\alpha\|\le\kappa$. Estimate
\eqref{stimaapriori:ave-bis} follows from the semigroup law and from
\eqref{stimaapriori:ave}. Indeed, fix $\omega>0$ and let
$C_0=\min\{1,\inf_{t\in [1,+\infty[}t^{-q_h(|\alpha|)}e^{\omega
t}\}$. Splitting $T_{\ve}(t)=T_{\ve}(1)T_{\ve}(t-1)$, for any $t>1$,
and taking \eqref{stimaapriori:ave} in Theorem \ref{thm:3.2} into
account, we get
\begin{equation*}
\|D^{\alpha}T_{\ve}(t)f\|_{\infty}\le \tilde
C\|T_{\ve}(t-1)f\|_{\infty}\le \tilde CC_0^{-1}C_0 \|f\|_{\infty}\le
\tilde CC_0^{-1}t^{-q_h(|\alpha|)}e^{\omega t}\|f\|_{\infty}.
\end{equation*}
Hence, \eqref{stimaapriori:ave-bis} follows with $\hat C=\tilde
CC_0^{-1}$.

\par
We can now prove that problem \eqref{HCP} admits a unique classical
solution for any $f\in C_b(\R^N)$, For this purpose, as in the proof
of Theorem \ref{thm:3.1}, we set $u_{\ve}=T_{\ve}(\cdot)f$. Then,
from \eqref{stimaapriori:ave-bis}, we easily deduce that for any
$0<T_1<T_2$
\begin{equation*}
\sup_{\ve\in ]0,1[}\sup_{t\in
[T_1,T_2]}\|u_{\ve}(t,\cdot)\|_{C^{\kappa}_b(\R^N)}<+\infty\quad\mbox{holds}.
\end{equation*}
Since the function $u_{\ve}$ solves the Cauchy problem
\eqref{pb-approx} and the coefficients of the operator ${\mathscr
A}_{\ve}$ are locally bounded, uniformly with respect to $\ve\in
]0,1[$, the function $D_tu_{\ve}$ is bounded in $]T_1,T_2[\times
B(R)$, for any $R>0$, by a constant, independent of $\ve$. Therefore
we have $u_{\ve}\in {\rm Lip}([T_1,T_2];C(B(R)))$ with norm
independent of $\ve\in ]0,1[$. By applying \cite[Propositions
1.1.2(iii) and 1.1.4(i)]{Lu1}, we now
deduce that $u_{\ve}\in C^{\theta/2,\kappa-1+\theta}(]T_1,T_2[\times
B(R))$ for any $\ve$ as above and some $\theta\in ]0,1[$, and with
$C^{\theta/2,\kappa-1+\theta}$-norm being bounded by a constant
independent of $\ve$. As a byproduct, using that $u_\ve$ solves
\eqref{pb-approx}, we deduce that $D_tu_{\ve}\in
C^{\theta/2,\kappa-3+\theta}(]T_1,T_2[ \times B(R))$ and, again, its
$C^{\theta/2,\kappa-3+\theta}$-norm is bounded by a constant
independent of $\ve$. Since $T_1$, $T_2$, $R$ are arbitrarily fixed,
using both a compactness and a diagonal argument, we can determine
an infinitesimal sequence $\{\ve_n\}$ such that $\{u_{\ve_n}\}$
converges in $C^{1,\kappa-1}(K)$, for any compact set $K\subset
]0,+\infty[\times\R^N$, to a function $u_f\in
C^{1+\theta/2,\kappa-1+\theta}_{\rm loc}(]0,+\infty[\times\R^N)$. Of
course, the function $u_f$ solves the differential equation in
\eqref{HCP} for $t>0$. The continuity of $u_f$ up to  $t=0$ and the
condition $u_f(0,\cdot)=f$, are obtained in three steps.
\par
{\em Step 1.} Suppose that $f\in C_c^2(\R^N)$. Then, by the proof of
\cite[Proposition 4.3]{MPW}, we know
that
\begin{equation*}
\|u_{\ve_n}(t)-f\|_{\infty}\le t\sup_{s\in
[0,+\infty[}\|T_{\ve_n}(s){\mathscr A}_{\ve_n}f\|_{\infty}\le
t\|{\mathscr A}_{\ve_n}f\|_{\infty}\le Ct,\qq\;\,t\in [0,+\infty[.
\end{equation*}
Hence, taking the limit, first as $n\to +\infty$ and then as $t\to
0^+$, we obtain that $u_f$ is continuous at $t=0$ where it equals
$f$. So we have shown that $u_f$ is the unique classical solution to
problem \eqref{HCP}. Moreover, we infer that $u_{\ve}$ converges to
$u_f$, as $\ve\to 0^+$, in $C^{1,\kappa-1}([T_1,T_2]\times B(R))$
for any $T_1,T_2,R$ as above. Indeed, by uniqueness (or by the
maximum principle in Proposition \ref{prop:maxprinc:0}), any
sequence $u_{\ve_n}$ (with $\ve_n$ being positive and infinitesimal)
which converges in $C^{1,{\kappa}}_{\rm loc}([0,+\infty[\times
\R^)$, must converge to $u_f$. Again the maximum principle implies
that for a non-negative $f\in C_c(\R^N)$ the solution $u_f$ is also
non-negative.
\par
{\em Step 2.} Suppose now that $f$ vanishes at $\infty$. Then, we can
approximate $f$ by a sequence of smooth and compactly supported
functions $f_n$. By estimate \eqref{stimaapriori:1} we know that
\begin{equation*}
\|T_{\ve_m}(t)f_n-T_{\ve_m}(t)f\|_{\infty}\le
\|f_n-f\|_{\infty},\qq\;\,t\in [0,+\infty[,\;\,n,m\in\N.
\end{equation*}
Letting $m\to +\infty$ yields
\begin{equation*}
\sup_{t\ge 0}\|u_{f_n}(t,\cdot)-u_f(t,\cdot)\|_{\infty}\le
\|f_n-f\|_{\infty},\qq\;\,n\in\N.
\end{equation*}
Hence $u_{f_n}$ converges to $u_f$ uniformly in
$[0,+\infty[\times\R^N$. Since $u_{f_n}$ is continuous in
$[0,+\infty[\times\R^N$ and $u_{f_n}(0,\cdot)=f_n$, it follows that
$u_f$ is continuous in $[0,+\infty[\times\R^N$ as well, and that
$u(0,\cdot)=f$ holds.

The same argument in the last part of Step 1, shows that, also in
this situation, the function $T_{\ve}(\cdot)f$ converges to $u_f$ in
$C^{1,\kappa-1}_{\rm loc}(]0,+\infty[\times\R^N)$ as $\ve\to 0^+$.
Moreover for non-negative $f$ we see the solution $u_f$ to be
non-negative as well.
\par
{\em Step 3.} We now consider the general case when $f\in C_b(\R)$.
We fix $R>0$ and a function $\psi\in C^{\infty}_c(\R^N)$ satisfying
$\chi_{B(R)}\le\psi\le\chi_{B(R+1)}$. Further, we split first
$f=\psi f+ (1-\psi)f$, and then we can write
$T_{\ve_n}(t)f=T_{\ve_n}(t)(\psi f)+ T_{\ve_n}(t)((1-\psi)f)$. We
remark that the semigroups $\{T_{\ve}(t)\}$ preserve positivity,
which is well-known but also follows immediately from the maximum
principle, Proposition \ref{prop:maxprinc:0}.  This implies
\begin{equation}
|\{T_{\ve_n}(t)((1-\psi)f)\}(x)|\le
\|f\|_{\infty}(T_{\ve_n}(t)(1-\psi))(x),\qq\;\,t>0\;\,x\in\R^N.
\label{tag}
\end{equation}
Recall that $T_{\ve_n}(\cdot)(1-\psi)=1-T_{\ve_n}(\cdot)\psi$ and
let $n\to +\infty$ in \eqref{tag} to conclude
\begin{equation*}
|u_f(t,x)-u_{\psi f}(t,x)|\le
\|f\|_{\infty}(1-u_{\psi}(t,x)),\qq\;\,t\in
]0,+\infty[,\;\,x\in\R^N.
\end{equation*}
Since, for any $x\in B(R)$, $(\psi f)(x)=f(x)$ and $u_{\psi}(t,x)$
tends to $1$ as $t\to 0^+$, we obtain that $u_f$ is continuous in
$[0,1]\times B(R)$ and $u_f(0,\cdot)=f$ in
$B(R)$. The arbitrariness of $R>0$ allows us to complete the proof.
Moreover, as in the previous cases, $T_{\ve}(\cdot)f$ converges to
$u_f$ in $C^{1,\kappa-1}_{\rm loc}(]0,+\infty[\times\R^N)$, as
$\ve\to 0^+$.

\med \pn (ii) and (iii). They follow from the maximum principle in
Proposition \ref{prop:maxprinc:0} and Steps 1 and 2 in the proof of
 (i).

\med\pn (iv). By (i), we know that the function $T(t)f$ belongs to
$C^{\kappa-1}(\R^N)$ for any $t>0$ and any $f\in C_b(\R^N)$, and $T_{\ve}(t)f$ converges
to $T(t)f$ in $C^{\kappa-1}_{\rm loc}(\R^N)$ as $\ve\to 0^+$. Since
the constant in \eqref{stimaapriori:ave-bis} is independent of
$\ve\in ]0,1]$, it is immediate to conclude that
\eqref{stimaapriori:a0} holds for any $\|\alpha\|\le\kappa-1$.
\end{proof}
With respect to derivatives in the first and second block of
variables we can prove more regularity.
\begin{theorem}\label{thm:reg}
Suppose that Hypotheses \ref{ipos-1} are satisfied and let
$\{T(t)\}$ be  the semigroup constructed in Theorem
\ref{thm:constr}. Then for any $f\in C_b(\R^N)$ and any multi-index
$\alpha\in\N_0^N$, with $\|\alpha\|=\kappa$ and  $\alpha_j\neq 0$
for some $j\le p_0+p_1$, the derivative $D^{\alpha}T(\cdot)f$ exists
in the classical sense in $]0,+\infty[\times \R^N$ and it is a
continuous function. Moreover, there exists a positive constant $C$,
depending only on  $\omega$, $h$ such that, for any $f\in
C^h_b(\R^N)$ and any $\alpha$ as above, we have
\begin{equation}
\|D^{\alpha}T(t)f\|_{C_b(\R^N)}\le Ct^{-q_h(|\alpha|)}e^{\omega
t}\|f\|_{C^h_b(\R^N)},\qq\;\,t\in ]0,+\infty[.
\label{stimaapriori:a1}
\end{equation}
\end{theorem}
\begin{proof}
We set $u=T(\cdot)f$ and split the proof into several steps. In the first one we show a formula that will be used in
Steps 2 to 5, in the actual proof of \eqref{stimaapriori:a1}. Until Step 5, we will assume at least  $f\in C_b^{\kappa-1}(\R^N)$, and then in Step 5 we proceed with an approximation argument.
Finally, in Step 6, we show that the function $D^{\alpha}u$ is
continuous in $]0,+\infty[\times\R^N$.
\par
{\em Step 1.} We fix $R>0$, $j\in\{1,\ldots,N\}$ and $f\in
C_b^{\kappa-1}(\R^N)$, and prove that, for any $\eta=\eta_R\in
C^{\infty}_c(\R^N)$ such that $\chi_{B(R)}\le\eta_R\le\chi_{B(2R)}$,
and any $\vartheta=\vartheta_R\in C^{\infty}_c(]0,+\infty[)$  such
that $\chi_{]R^{-1},R[}\le \vartheta_R \le\chi_{](2R)^{-1},2R[}$, it
holds that
\begin{equation}
D_ju(t,x)=
%T(t)(D_j(\eta f))(x)+
\int_0^t(T(t-s)g_j(s,\cdot))(x)ds,\qq\;\,t\in ]R^{-1},R[,~x\in B(R),
\label{variat-of-constant:2}
\end{equation}
\pn where
\begin{align}\label{eq:gj}
g_j&=\vartheta{\Tr}((D_jQ_0)D^2_{\star}(\eta
u))+\vartheta(B^*{D}_{\star}(\eta u))_j+\vartheta\langle
D_jF,D_\star(\eta u)\rangle\no\\
&\quad -\vartheta D_j(u{\mathscr A}\eta)
-2\vartheta \langle Q_0{D}_{\star} u,{D}_{\star} D_j\eta\rangle\no\\
&\quad -2\vartheta \langle Q_0{D}_{\star}
D_ju,{D}_{\star}\eta\rangle -2\vartheta \langle (D_jQ_0){D}_{\star}
u,{D}_{\star}\eta\rangle +\vartheta' D_j(\eta u),
\end{align}
${D}_{\star}u$ and $D^2_{\star}u$ denoting, respectively, the vector
of first-order derivatives of $u$ with respect to indices not
greater than $p_0$ and the quadratic submatrix obtained erasing the
last $N-p_0$ rows and columns from $D^2u$.

To prove \eqref{variat-of-constant:2}, for any $\delta\in\, ]-1,1[$, we
introduce the operator $\tau^j_\delta$ defined on $C_b(\R^N)$ by
\begin{equation*}
(\tau_\delta ^j\psi)(x)=\frac{\psi(x+\delta e_j)-\psi(x)}{\delta },\qq\;\,x\in\R^N,\;\,
\psi\in C_b(\R^N).
\end{equation*}
Moreover, we set $w_{\ve,\delta }^j=\vartheta\tau_\delta ^jv_{\ve}$ where
$v_{\ve}=\eta u_{\ve}$. In the sequel, in order to shorten the
notation, if there is no danger of confusion we only stress
explicitly  the dependence on $\ve$ of the functions considered. As
it is easily seen,
\begin{equation*}
\left\{
\begin{array}{lll}
\ds D_tw_{\ve}(t,x)={\mathscr A}_{\ve}w_{\ve}(t,x)+g_{\ve}(t,x),\q & t\in\, ]0,+\infty[, &x\in\R^N,\\[3mm]
w_{\ve}(0,x)=0, &&x\in\R^N,
\end{array}
\right.
\end{equation*}
\pn holds with
\begin{align}
g_{j,\delta,\ve}&=g_{\ve}=
\vartheta{\Tr}(\tau(Q_0)D^2_{\star}v_{\ve})+\vartheta(B^*{D}
v_{\ve})_j +\vartheta\langle \tau(F),D_\star
v_\ve\rangle\no\\
&\quad-\vartheta\tau(u_{\ve}{\mathscr A}_{\ve}\eta)
-2\vartheta\langle Q_0{D}_{\star}
u_{\ve}(\cdot,\cdot+\delta e_j),{D}_{\star} \tau(\eta)\rangle
\no\\[1mm]
&\quad -2\ve\vartheta\langle
{D}_{\star\star}\tau(u_{\ve}),{D}_{\star\star}\eta\rangle
-2\vartheta\langle\tau(Q_0){D}_{\star} u_{\ve}(t,\cdot+\delta e_j),{D}_{\star}\eta(\cdot+\delta e_j)\rangle\no\\[1mm]
&\quad-2\ve\vartheta\langle
{D}_{\star\star}u_{\ve}(\cdot,\cdot+\delta e_j),{D}_{\star\star}\tau(\eta)\rangle
-2\vartheta\langle Q_0{D}_{\star} \tau(u_{\ve}),{D}_{\star}
\eta\rangle+\vartheta'\tau(v_{\ve}), \label{gvehkr}
\end{align}
and ${D}_{\star\star}\psi$ denotes the vector of the first-order
derivatives of the function $\psi:\R^N\to\R$ with respect to the
last $N-p_0$ variables. In view of the variation of constants formula (see \cite[Theorem
3.5]{P1}), we obtain that $w_{\ve}$ satisfies
\begin{equation}
w_{\ve}(t,x)=
\int_0^t(T_{\ve}(t-s)g_{\ve}(s,\cdot))(x)ds,\qq\;\,t\in
[0,+\infty[,~x\in\R^N. \label{variat-of-constant}
\end{equation}
\pn We are going to show that we can take the limit as $\ve\to 0^+$
in \eqref{variat-of-constant} and write
\begin{equation}
v_\delta (t,x):=\eta(x)(\tau_\delta u(t,\cdot))(x) =
\int_0^t(T(t-s)g_{j,\delta}(s,\cdot))(x)ds,\quad\;\,t\in
[0,+\infty[,~x\in\R^N, \label{variat-of-constant:1}
\end{equation}
\pn where $g_{j,\delta}$ is obtained from $g_{j,\delta,\ve}$ by replacing
$u_{\ve}$ by $u$ and letting $\ve=0$ in \eqref{gvehkr}. By the
results in the proof of Theorem \ref{thm:constr}(i), it follows
immediately that the continuous function $g_{j,\delta,\ve}$ converges
 to the function $g_{j,\delta}$ uniformly in $[0,+\infty[\times\R^N$, as
$\ve\to 0^+$. This implies that, for any $r,s>0$,
$T_{\ve}(r)g_{j,\delta,\ve}(s,\cdot)$ converges to $T(r)g_{j,\delta}(s,\cdot)$
locally uniformly in $\R^N$, as $\ve\to 0^+$. Indeed, for any
compact set $K\subset\R^N$, we have
\begin{align*}
&\|T_{\ve}(r)g_{j,\delta,\ve}(s,\cdot)-T(r)g_{j,\delta}(s,\cdot)\|_{C(K)}\\[1mm]
&\quad\le
\|T_{\ve}(r)(g_{j,\delta,\ve}(s,\cdot)-g_{j,\delta}(s,\cdot))\|_{C(K)}
+\|T_{\ve}(r)g_{j,\delta}(s,\cdot)-T(r)g_{j,\delta}(s,\cdot)\|_{C(K)}\\[1mm]
&\quad\le\|g_{j,\delta,\ve}(s,\cdot)-g_{j,\delta}(s,\cdot)\|_{\infty}
+\|T_{\ve}(r)g_{j,\delta}(s,\cdot)-T(r)g_{j,\delta}(s,\cdot)\|_{C(K)}.
\end{align*}
From the proof of Theorem \ref{thm:constr}(i) we see that the last
term in  the previous chain of inequalities vanishes as $\ve\to
0^+$. Moreover, since the semigroups $\{T_{\ve}(t)\}$ are
contractive, the function $(r,s)\mapsto
T_{\ve}(r)g_{j,\delta,\ve}(s,\cdot)$ is bounded in $[0,+\infty[\times
[0,+\infty[\times\R^N$, uniformly with respect to $\ve\in ]0,1[$.
Therefore, the dominated convergence theorem yields
\eqref{variat-of-constant:1}.

\par We can now prove formula \eqref{variat-of-constant:2}. Since, by
Theorem \ref{thm:constr}(iv), the function $u$ is bounded in
$[0,\timeT]$ with values in $C^{\kappa-1}_b(\R^N)$ for any
$\timeT>0$ (use $f\in C_b^{\kappa-1}(\R^N)$), it is immediate to see that $g_\delta (s,\cdot)$ converges to
$g(s,\cdot)$ uniformly in $\R^N$ for any $s>0$. Formula
\eqref{variat-of-constant:2} now follows from
\eqref{variat-of-constant:1} via the dominated convergence theorem.
\par
{\em Step 2.} Here, and in the forthcoming Steps 3 and 4, we assume
that $f\in C^{\kappa}_b(\R^N)$. Let us fix a multi-index
$\alpha=(\alpha_1,\ldots,\alpha_N)\in\N_0^N$ with
$\|\alpha\|=\kappa$ and $\|(\alpha_1,\ldots,\alpha_{p_0})\|\ge 1$.
We denote by $j$ the largest integer such that $\alpha_j\neq 0$ and
set $\beta:=\alpha-e_j^{(N)}$. Set
$\beta':=|\beta|-e_1^{(r+1)}=(\beta_0-1,\beta_1,\dots,\beta_r)$, and
denote by $\iota$  the smallest integer with $\beta'_\iota>0$.
%Moreover, we denote by $l\le p_0$ the smallest integer such that
%$\beta_l> 0$, and by $\iota\in\{0,\ldots,r\}$ the smallest integer
%such that for $\gamma=|\beta-e_l^{(N)}|$, we have $\gamma_\iota>0$.
To prove that the derivative $D^{\alpha}u$ exists in the classical
sense it suffices to show that we can differentiate, with respect to
the multi-index $\beta$, the function in
\eqref{variat-of-constant:2}. For this purpose, we observe that,
from \eqref{stimaapriori:a0}  with $h=\kappa-3$ and $h=\kappa-2$, we
deduce that
\begin{equation*}
\|D^{\beta}T(t)\psi\|_{\infty}\le
Ct^{-\frac{1}{2}-(1-\theta)\frac{2\iota+1}{2}}e^{\omega t}
\|\psi\|_{C^{\kappa-3+\theta}_b(\R^N)}
\end{equation*}
holds for any $t\in ]0,+\infty[$, $\psi\in C^{\kappa-3+\theta}_b(\R^N)$
and  $\theta=0,1$. By interpolation, we can extend the previous
estimate to any $\theta\in [0,1]$. Estimate
\eqref{stimaapriori:ave-bis} implies that, for any multi-index
$\gamma$ with length $\kappa-1$ and any $t>0$, the function
$D^{\gamma}u_{\ve}(t,\cdot)$ is Lipschitz continuous in $\R^N$ with
Lipschitz semi-norm that can be bounded  by $Ce^{\omega t}$ for any
$\omega>0$ and some $C=C(\omega)$, where the constants are uniform
in $\ve$. Since $u_{\ve}$ converges to $u$ in $C^{1,\kappa-1}_{\rm
loc}(]0,+\infty[\times\R^N)$, the function
 $D^{\gamma}u(t,\cdot)$ is Lipschitz continuous in $\R^N$ and its norm can be bounded by
$Ce^{\omega t}$. As a byproduct, we infer that, for any $\theta\in
]0,1[$ and any $\timeT>0$, the function
$\|g_j(s,\cdot)\|_{C^{\kappa-3+\theta}_b(\R^N)}$ is bounded in
$]0,\timeT[$. Moreover,  for $0<s<t\le \timeT$ we have
\begin{equation}\label{eq:becs}
\|D^{\beta}T(t-s)g_j(s,\cdot)\|_{\infty} \le
C(t-s)^{-\frac{1}{2}-(1-\theta)\frac{2\iota+1}{2}}e^{\omega
\timeT}\sup_{s\in
]0,\timeT[}\|g_j(s,\cdot)\|_{C^{\kappa-3+\theta}_b(\R^N)}.
\end{equation}
Consequently, if we take $\theta>2\iota/(2\iota+1)$,  we get an
integrable function on the right hand side of \eqref{eq:becs}, and
hence we can differentiate under the integral sign in
\eqref{variat-of-constant:2}. This proves that the derivative
$D^{\alpha}u$ exists in the classical sense. Moreover, it satisfies
\eqref{stimaapriori:a0}. Indeed, the sup-norm of $D^{\alpha}u$ can
be controlled from above by the Lipschitz seminorm of $D^{\beta}u$
which, as we have shown, can be estimated from above by $Ce^{\omega
t}$ for any $t>0$, any $\omega>0$ and some $C=C(\omega)$.
\par
{\em Step 3.} We now assume that $\|\alpha\|=\kappa$ and
$\alpha_i=0$ for all $i=1,\ldots,p_0$, whereas $\alpha_j\neq 0$ for
some $j\in\{p_0+1,\ldots,p_0+p_1\}$; we again set
$\beta:=\alpha-e^{(N)}_j$. We are going to show that we can
differentiate the formula \eqref{variat-of-constant:2} with respect
to the multi-index $\beta$. For this purpose, let again $\iota$ be
the largest integer with $\beta_\iota\neq 0$, and note that it
suffices to prove that, for any $\timeT>0$, the function $g_j$ is
bounded in $]0,\timeT[$ with values in
$C^{\kappa-2+\theta}_b(\R^N)$, for some $\theta\in
]\frac{2\iota-1}{2\iota+1},1[$. Indeed, once this property is
proved, estimate \eqref{stimaapriori:a0} gives
\begin{align*}
\|D^{\beta}T(t-s)g_j(s,\cdot)\|_{\infty} \le
C(t-s)^{-(1-\theta)\frac{2\iota+1}{2}}e^{\omega \timeT}\sup_{s\in
]0,\timeT[}\|g_j(s,\cdot)\|_{C^{\kappa-3+\theta}_b(\R^N)},
\end{align*}
for all $0<s<t\le \timeT$, for arbitrary $\timeT>0$ and some
$C=C(\omega)$, and we can complete the proof applying the same
arguments as in the previous step.

Due to the structure of $g_j$, in order to prove that $g_j$ is
bounded in $]0,\timeT[$ with values in
$C^{\kappa-2+\theta}_b(\R^N)$, it suffices to show that for any pair
of indexes $l\le p_0$ and $l'\le p_0+p_1$, with $l\le l'$, the
function $D_{ll'}u(t,\cdot)$ belongs to $C^{\kappa-2+\theta}(\R^N)$
and $\sup_{t\in
]1/M,M[}\|D_{hl}u(t,\cdot)\|_{C^{\kappa-2+\theta}(B(M))}<+\infty$
for any $M>0$. Actually, only the  first and the fifth, second-order
terms have to be taken care of in \eqref{eq:gj}. Indeed, applying
$D^\gamma$ with $\|\gamma\|=\kappa-2$ to any of the other terms
$\widetilde g$ from \eqref{eq:gj} (in which there are only first
order derivatives of $u$), we get that $D^\gamma \widetilde
g(t,\cdot)$ is Lipschitz continuous, uniformly in $]0,t_0[$ and $\sup_{t\in
]0,t_0[}\|\widetilde g
(t,\cdot)\|_{C^{\kappa-2+\theta}(\R^N)}<+\infty$ (these follow by
approximating $u$ by $u_\ve$ as  we have done several times above).

We now prove the assertion about $D_{ll'}u$. So let now $\beta'\in \N_0^N$ with $\|\beta'\|=\kappa-2$. Denote
furthermore by $i$ the largest integer such that $\beta'_i>0$, and
define $\beta=\beta'+e_l^{(N)}+e_{l'}^{(N)}-e_{i}^{(N)}$. From
\eqref{stimaapriori:a0} and from the already proved assertion in
Step 2 we obtain,  by  using interpolation as well, that
\begin{equation}
\|D^{\beta}T(t)\psi\|_{C^{\rho}_b(\R^N)}\le
Ct^{-\frac{\rho\theta}{2}- (1-\theta)\frac{3\rho+1}{2}}e^{\omega t}
\|\psi\|_{C^{\kappa-2+\theta}_b(\R^N)},\qq\;\,t\in ]0,+\infty[,
\label{1-b}
\end{equation}
for any $\theta,\rho\in [0,1]$. From \eqref{1-b} it now follows that
\begin{equation}
\|D^{\beta}T(t-s)g_i(s,\cdot)\|_{C^{\rho}_b(\R^N)} \le
C(t-s)^{-\frac{3\rho+1}{2}}\|g_i(s,\cdot)\|_{C^{\kappa-2}_b(\R^N)}
\le C(t-s)^{-\frac{3\rho+1}{2}} \label{1-c}
\end{equation}
holds for any $0<s<t\le \timeT$. Hence, if we fix $\gamma\in ]0,1[$
and take $\rho=\theta_1=\frac{\gamma}{3}$, we see that the function
in the right-hand side of \eqref{1-c} is integrable in $]0,\timeT[$
for any $\timeT>0$. Thus we can differentiate under the integral in
\eqref{variat-of-constant:2}, and conclude  that the function
$D_{ll'}u$ is bounded in $]R^{-1},R[$ with values in
$C^{\kappa-2+\theta_1}(B(R))$. Due to the arbitrariness of $R$, it
follows that $D_{ll'}u$ is bounded in $H$ with values in
$C^{\kappa-2+\theta_1}(K)$ for any compact set $H\times K\subset
]0,+\infty[\times \R^N$.

As a second step, using \eqref{1-b}, we deduce that $D_{ll'}u$ is
bounded in $H$ with values in $C^{\kappa_2+\theta_2}(K)$ for any $H$
and $K$ as above, where
$\theta_2=\gamma\frac{1+\theta_1}{3-2\theta_1}$. Iterating, this
argument, we see that $D_{ll'}u$ is bounded in $H$ with values in
$C^{\kappa-2+\theta_k}(K)$, where the sequence $\{\theta_k\}$ is
defined by recurrence as follows:
\begin{equation*}
\left\{
\begin{array}{ll}
\ds\theta_{k+1}=\gamma\frac{1+\theta_k}{3-2\theta_k}, & k\le k_0,\\[3mm]
\theta_0=0,
\end{array}
\right.
\end{equation*}
where either $k_0=+\infty$ or $k_0$ is the largest integer such that
$\theta_k<3/2$.

It easy to see that  $\theta_k<\theta_{k+1}$ holds for any $k\le
k_0$ and choice $\gamma\in ]0,1[$. For the choice
$\gamma=\frac{3}{4}$, the equation
$\ell=\gamma\frac{1+\ell}{3-2\ell}$ has no real solutions. This fact combined with the monotonicity property implies that there exists $k_1$ such that
$\theta_{k_1}>1$. It follows that $D_{ll'}u$ is bounded in $H$ with
values in $C^{\kappa-2+\theta_k}(K)$ for any $\theta\in ]0,1[$ and,
consequently, $g_i$ is locally bounded in $]0,+\infty[$ with values
in $C^{\kappa-2+\theta}_b(\R^N)$ for any $\theta\in ]0,1[$. The
proof of Step 3 is complete.
\par
{\em Step 4.} We now show that $\|D^{\alpha}T(t)f\|_{C_b(\R^N)}\le
Ct^{-q_h(|\alpha|)}e^{\omega t}\|f\|_{C^h_b(\R^N)}$ for any $t\in
]0,+\infty[$, any $f\in C^{\kappa}_b(\R^N)$, any $h\in\N$ with
$h<\kappa$, and any $\alpha\in\N_0^N$ with length $\kappa$ and such
that $\alpha_j\neq 0$ for some $j\le p_0+p_1$. For this purpose, fix $j\le p_0+p_1$
such that $\alpha_j\neq 0$. Further, we let
$\beta=\alpha-e_j^{(N)}$. From \eqref{stimaapriori:a0} it follows
that for any $x\in\R^{N-1}$ the Lipschitz seminorm of the function
$\psi:=(D^{\beta}T(t)f)(x_1,\ldots,x_{j-1},\cdot,x_{j+1},\ldots
x_n)$ does not exceed $Ce^{\omega
t}t^{-q_h(\alpha)/2}\|f\|_{C_b^h(\R^N)}$, with $C$ depending only on
$\omega$. Since the Lipschitz seminorm of the function $\psi$ equals
the sup-norm of the function
$(D^{\alpha}T(t)f)(x_1,\ldots,x_{j-1},\cdot,x_{j+1},\ldots x_n)$
(which is already known to be existing by Steps 2 and 3), the
desired estimate follows.
\par
{\em Step 5.} We now prove \eqref{stimaapriori:a1} for a general
$f\in C_b^h(\R^N)$ ($h<\kappa$) and any multi-index
$\alpha\in\N_0^N$ such that $\|\alpha\|=\kappa$ and $\alpha_j\neq 0$
for some $j\le p_0+p_1$. Let us notice that we can limit ourselves
to proving that the derivative $D^{\alpha}T(t)f$ exists in the
classical sense for any $t>0$ and any $f\in C_b(\R^N)$. Indeed, once
this property is checked, estimate \eqref{stimaapriori:a1} can be
proved arguing as in Step 4.

We begin by considering the case when $f\in BUC(\R^N)$, and we fix a
sequence $\{f_n\}\in C^{\kappa}_b(\R^N)$ converging to $f$ uniformly
in $\R^N$.  We can write
\begin{align*}
\|D^{\alpha}T(t)f_n-D^{\alpha}T(t)f_m\|_\infty&\le Ce^{\omega
t}t^{-q_0(|\alpha|)} \le Ce^{\omega t}t^{-q_0(|\alpha|)} \|f_n-f_m\|_{\infty},
\end{align*}
for any $t>0$ and any $n,m\in\N$. If follows that
$\{D^{\alpha}T(t)f_n\}$ is a Cauchy sequence in $C_b(\R^N)$ and,
consequently, $D^{\alpha}T(t)f\in C_b(\R^N)$.

\noindent We now assume that $f\in C_b(\R^N)$. By splitting
$T(t)f=T(t/2)T(t/2)f$ and noting that $T(t/2)f\in C^1_b(\R^N)\subset
BUC(\R^N)$, from the preceding  we deduce that $D^{\alpha}T(t)f=D^{\alpha}T(t/2)(T(t/2)f)$
exists in the classical sense.

{\em Step 6.} To complete the proof, we have to show that for any
multi-index $\alpha$ with length $\kappa$ such that $\alpha_j\neq 0$
for some $j\le p_0+p_1$, the function $D^{\alpha}T(t)f$ is
continuous in $]0,+\infty[\times\R^N$. For this purpose, let $i$ be
the largest integer such that $\alpha_i>0$. Let us fix
$y\in\R^{N-1}$, and introduce the function
$\psi=D^{\beta}u(\cdot ,y_1,\ldots,y_{i-1},\cdot,y_{i+1},\ldots,y_N)$
where, again, $\beta=\alpha-e_i^{(N)}$, and still $\beta_{j'}>0$ for some $j'\le p_0+p_1$. From the results in Steps 2
to 5 we know that $\psi$ is bounded in $]a,b[$ with values in
$C^{1+\theta}(B(R))$  for
some $\theta\in ]0,1[$ and any $a,b,R>0$, with $a<b$. Applying
\cite[Propositions 1.1.2(iii) and 1.1.4(iii)]{Lu1} to the function
$\psi(t,\cdot)-\psi(s,\cdot)$ ($s,t,\in [a,b]$), we immediately see
that
\begin{align*}
\|\psi(t,\cdot)-\psi(s,\cdot)\|_{C^1(B(R))} \le
&C\|\psi(t,\cdot)-\psi(s,\cdot)\|_{C(B(R))}^{\frac{\theta}{1+\theta}}
\|\psi(t,\cdot)-\psi(s,\cdot)\|_{C^{1+\theta}(B(R))}^{\frac{1}{1+\theta}}\\
\le &
C\|\psi(t,\cdot)-\psi(s,\cdot)\|_{C(B(R))}^{\frac{\theta}{1+\theta}},
\end{align*}
for some constant $C$, independent of $y$. Since $u\in
C^{1,\kappa-1}(]0,+\infty[\times\R^N)$, we immediately deduce that
the right-hand side of the previous chain of inequalities vanishes
as $|t-s|\to 0^+$, implying that the function $D^{\alpha}u(\cdot,x)$
is continuous in $[a,b]$ uniformly with respect to $x\in\R^N$. This
is enough to conclude that $D^{\alpha}u$ is continuous in
$]0,+\infty[\times\R^N$.
\end{proof}

\begin{rem}
{\rm(i)} We remark that the results proved in Theorem
\ref{thm:constr} are stronger than those in
\cite{rot}.

\smallskip\pn{\rm(ii)}
Some calculation yields that the bootstrap argument used in Step 3
of the proof of Theorem \ref{thm:reg} cannot be applied to prove the
existence of the derivative $D^{\alpha}T(t)f$ in the classical sense
when $\|\alpha\|=\kappa$ and $\alpha_j=0$ for all
$j=1,\ldots,p_0+p_1$.
\end{rem}

%%%%%%%%%%%%%%%%%%%%%%%%%%%%%%%%%%%%%%%%%%%%%%%%%%%%%%%%%%%%%%%%%%%%%%%
%%                                                                   %%
%%             PROPERTIES                                            %%
%%                                                                   %%
%%%%%%%%%%%%%%%%%%%%%%%%%%%%%%%%%%%%%%%%%%%%%%%%%%%%%%%%%%%%%%%%%%%%%%%

\subsection{Properties of the semigroup}
In this section we first state some continuity property of the
semigroup $\{T(t)\}$ that will play a fundamental role in order to
prove the Schauder estimates of Section \ref{sec-optimal}. Then, we
characterize the domain of the weak generator of the semigroup.
Since the proofs of the following proposition can be obtained
arguing as in \cite{Lo}, we omit it.

\begin{prop}\label{prop-conv-compatti}
The following assertions hold.
\begin{enumerate}[\rm (i)]
\item
Let $\{f_n\}\subset C_b(\R^N)$ be a bounded sequence of continuous
functions converging to $f\in C_b(\R^N)$, pointwise in $\R^N$. Then,
$T(\cdot)f_n$ converges to $T(\cdot)f$ pointwise in
$[0,+\infty[\times\R^N$.
\item
If  $f_n$ converges to $f$, locally uniformly in $\R^N$ and
$\|f_n\|_\infty$ is bounded, then $T(\cdot)f_n$ converges to
$T(\cdot)f$ locally uniformly in $[0,+\infty[\times\R^N$ and in
$C^{1,2}(F)$ for any compact set $F\subset ]0,+\infty[\times\R^N$.
\item
There exists a family of probability Borel measures $\{p(t,x,dy):
t>0, x\in\R^N\}$ such that, for any $f\in C_b(\R^N)$,
\begin{equation*}
(T(t)f)(x)=\int_{\R^N}f(y)p(t,x,dy),\qq\;\,t>0,\;\,x\in\R^N.
\end{equation*}
Consequently, $\{T(t)\}$ can be extended to the space $B_b(\R^N)$ of
all bounded and Borel measurable functions $f:\R^N\to\R$ with a
semigroup of positive contractions.
\item
$\{T(t)\}$ is strong Feller, i.e., $T(t)f\in C_b(\R^N)$ $($actually
$T(t)\in C^{\kappa-1}_b(\R^N))$ for any $f\in B_b(\R^N)$.
\end{enumerate}\end{prop}

Differently from what happens in the classical case when the
coefficients are bounded, in general the semigroup associated with
elliptic operators with unbounded coefficients is neither analytic
in $C_b(\R^N)$, nor strongly continuous in $BUC(\R^N)$.
 Assertion (ii) above in Proposition
\ref{prop-conv-compatti}, however, expresses the fact that the
semigroup $\{T(t)\}$ is bi-continuous for the topology of
locally uniform convergence $\tau_c$ (see \cite{Ku01, Ku03} or
\cite{Fa03}), or which is essentially the same it is a
locally-equicontinuous semigroup with respect to the mixed topology.
The mixed topology is finest locally convex topology agreeing with
$\tau_c$ on $\|\cdot\|_\infty$-bounded sets. (See {\cite{wiweger} or
\cite{sentilles:1972} for the definition of the mixed topology;
\cite{Fa03} for the equivalence of these two families of semigroups;
and \cite[Section IX.2.]{yosida} for locally-equicontinuous
semigroups). This allows us to associate an infinitesimal generator
$(A,D(A))$ to the semigroup (see \cite{Ku01,Ku03}):
\begin{align*}
D(A)&:=\Bigl\{f\in C_b(\R^N):\:\exists
\tau_c-\lim_{t\to0^+}\tfrac{T(t)f-f}{t}\mbox{ and } \sup_{t\in
]0,1]}\tfrac{\|T(t)f-f\|_\infty}{t}\Bigr\}\\
Af&:=\tau_c-\lim_{t\to0^+}\tfrac{T(t)f-f}{t}.
\end{align*}
With this definition the infinitesimal generator $(A,D(A))$ is a
Hille-Yosida operator, and the resolvent of $A$ can be calculated
\begin{equation*}
R(\lambda,A)f=\int_0^{+\infty}e^{-\lambda t}T(t)fdt%\label{laplace}
\end{equation*}
where the integral exists in the topology $\tau_c$ and for all
positive $\lambda$. In general one could replace here the
$\tau_c$-convergence by pointwise convergence resulting in the
so-called ``weak-generator'', in our case, however, this would not
result in any difference.

\begin{rem}
We note that assertion (iii) in Proposition \ref{prop-conv-compatti}
follows also directly from the first part of (ii). Actually, we even have the
equivalence of these two statements, for details see, e.g.,
\cite{Fa03}.
\end{rem}

The next proposition  characterizes the domain $D(A)$.

\begin{prop}\label{prop:char-dom} The following characterization holds true:
\begin{align}
D(A)=\Big\{f\in C_b(\R^N):&\exists \{f_n\}\subset C^2_b(\R^N),
\exists g\in C_b(\R^N): \no\\&f_n\to f,~{\mathscr A}f_n\to g
~\mbox{loc. uniformly in}~\R^N\no\\
&\mbox{ and}~~\sup_{n\in\N}\,(\|f_n\|_{\infty}+\|{\mathscr
A}f_n\|_{\infty})<+\infty\Big\}. \label{char:dom}
\end{align}
\pn Moreover, $Af={\mathscr A}f$ for any $f\in D(A)$. Here and
above, ${\mathscr A}f$ is meant in the sense of distributions.
\end{prop}
This tells us essentially that $C_b^2(\R^N)$ is a core for the
generator $A$ with respect to the mixed topology, or which is
the same is a bi-core with respect to $\tau_c$ (see \cite{Ku03}).
For the proof we use an invariance argument and need the following
preparatory lemma.
\begin{lemma}\label{lemma:commut}
For the semigroup $\{T(t)\}$ we can state the following.
\begin{enumerate}[\rm(i)]
\item
For any $t>0$, $T(t)$ commutes with ${\mathscr A}$ on $D_0({\mathscr
A}):=\{f\in C^2_b(\R^N):~{\mathscr A}f\in C_b(\R^N)\}$;
\item
if $\{f_n\}\subset C_b(\R^N)$ is a bounded sequence converging
locally uniformly to some function $f\in C_b(\R^N)$, then, for any
$\lambda>0$, $R(\lambda,A)f_n$ converges to $R(\lambda,A)f$, locally
uniformly in $\R^N$;
\item
for any $\lambda>0$, $R(\lambda,A)$ is a bounded operator mapping
$C^{h}_b(\R^N)$ into itself for any $h\in\N$ such that $h<\kappa$.
\end{enumerate}
\end{lemma}

\begin{proof}  (i). We begin the proof by recalling that, for any $t>0$,
$T_{\ve}(t)$ and ${\mathscr A}_{\ve}$ commute on $D(A_0)$ since they
commute on
\begin{equation*}
D_{\max}({\mathscr A}_{\ve}):=\Bigl\{g\in
C_b(\R^N)\cap\bigcap_{1<p<+\infty} W^{2,p}_{\rm loc}(\R^N):{\mathscr
A}_{\ve}g\in C_b(\R^N)\Bigr\}
\end{equation*}
(see e.g., \cite[Propositions 2.3.1, 2.3.6, 4.1.1 and Lemma
2.3.3]{bertoldi-lorenzi}). Hence, we have only to show that, for any
$f\in D(A_0)$, ${\mathscr A}_{\ve}T_{\ve}(t)f$ and
$T_{\ve}(t){\mathscr A}_\ve f$ converge to ${\mathscr A}T(t)f$ and
$T(t){\mathscr A} f$, respectively, as $\ve\to 0^+$. The proof of
Theorem \ref{thm:constr} shows that $T_{\ve}(t)f$ converges to
$T(t)f$ in $C^2(K)$, as $\ve\to 0^+$, for any compact set
$K\subset\R^N$. Therefore, ${\mathscr A}_{\ve}T_{\ve}(t)f$ converges
to ${\mathscr A}T(t)f$ locally uniformly in $\R^N$. On the other
hand, recalling that $\{T_{\ve}(t)\}$ is a contraction semigroup, we can write
\begin{align}
\|T_{\ve}(t){\mathscr A}_{\ve}f-T(t){\mathscr A}f\|_{C(K)}&\le
\|T_{\ve}(t)({\mathscr A}_{\ve}f-{\mathscr A}f)\|_{C(K)}
+\|(T_{\ve}(t)-T(t)){\mathscr A}f\|_{C(K)}\no\\[1mm]
&\le \|{\mathscr A}_{\ve}f-{\mathscr
A}f\|_{\infty}+\|(T_{\ve}(t)-T(t)){\mathscr A}f\|_{C(K)},
\label{commut-prop}
\end{align}
for any $t>0$. Since ${\mathscr A}_{\ve}f$ converges uniformly in
$\R^N$ to ${\mathscr A}f$ as $\ve\to 0^+$, estimate
\eqref{commut-prop} implies that $T_{\ve}(t){\mathscr A}_{\ve}f$
tends to $T(t){\mathscr A}f$, locally uniformly in $\R^N$.

(ii) This is a property shared by resolvents of generators of
bi-continuous semigroups, see \cite{Ku01, Ku03}. For the sake of
completeness we give the straightforward proof. Let $\{f_n\}$ and
$f$ be as in the statement of the lemma. Observe that for any
compact set $K\subset\R^N$,
\begin{equation}
\|R(\lambda,A)(f_n-f)\|_{C(K)}\le\int_0^{+\infty}e^{-\lambda
t}\|T(t)(f_n-f)\|_{C(K)}dt,\qq\;\,\lambda>0. \no
\end{equation}
\pn Theorem \ref{thm:constr}(ii) and Proposition
\ref{prop-conv-compatti}(ii) show that
 $\{\|T(\cdot)(f_n-f)\|_{C(K)}\}$ is a bounded sequence converging pointwise in $[0,+\infty[$
to $0$ as $n\to +\infty$. The assertion now follows from the
dominated convergence theorem.

(iii) It follows immediately from the estimate
\eqref{stimaapriori:a0} with $h=k$.
\end{proof}

\begin{proof}[Proof of Proposition $\ref{prop:char-dom}$] Taking Lemma
\ref{lemma:commut} into account, it is easy to check that, for any
$f\in D_0({\mathscr A})$, it holds that
\begin{align}
(R(1,A){\mathscr A}f)(x)&=
\int_0^{+\infty}e^{-t}({\mathscr A}T(t)f)(x)dt\no\\
&=\int_0^{+\infty}e^{-t}\left (\frac{\partial}{\partial
t}T(t)f\right )(x)dt =-f(x)+(R(1, A)f)(x), \label{int:byparts}
\end{align}
\pn for any $x\in\R^N$. Therefore, $f\in D(A)$ and $Af={\mathscr
A}f$, so that $D_0({\mathscr A})\subset D(A)$ and $A_{|D_0({\mathscr
A})}\equiv {\mathscr A}$. We could now conclude the proof by using
density and the invariance under $\{T(t)\}$ of $D(A_0)$ and by
referring, e.g., to \cite[Proposition 1.21]{Ku03}, or to
\cite[Proposition 2.12]{manca} (the analogous statement for strongly-continuous
semigroups is in \cite[Proposition II.1.7]{engel/nagel:2000}). We
nevertheless give a complete proof.

Let us fix $f\in \hat D$ (the function space defined by the
right-hand side of \eqref{char:dom}) and let $\{f_n\}\subset
C^2_b(\R^N)$ be a bounded sequence with respect to the sup-norm
which converges to $f$ locally uniformly in $\R^N$ and it is such
that the sequence $\{{\mathscr A}f_n\}\subset C_b(\R^N)$ is bounded
and converges locally uniformly in $\R^N$ to some function $g\in
C_b(\R^N)$. By the above results we know that
\begin{equation}
f_n=R(1,A)(f_n-{\mathscr A}f_n),\qq\;\,n\in\N. \label{star:fn}
\end{equation}

Lemma \ref{lemma:commut}(ii) allows us to take the limit as $n\to
+\infty$ in \eqref{star:fn}, getting $f=R(\lambda,A)(\lambda f-g)$,
so that $f\in D(A)$ and $Af=g$. We claim that $Af={\mathscr A}f$
(where ${\mathscr A}f$ is meant in the distributional sense). For
this purpose, it suffices to observe that, for any $\va\in
C^{\infty}_c(\R^N)$, we have
\begin{equation}
\int_{\R^N}\varphi{\mathscr A}f_n\,dx=\int_{\R^N}f_n{\mathscr
A}^*\varphi\, dx,\qq\;\,n\in\N, \label{2star}
\end{equation}
\pn where ${\mathscr A}^*$ is the formal adjoint of the operator
${\mathscr A}$. Letting $n\to +\infty$ in \eqref{2star}, the claim
follows. We have so proved that $\hat D$ is contained in $D(A)$ and
$A={\mathscr A}$ on $\hat D$.
\par
We now prove that $D(A)\subset\hat D$. For this purpose, we fix
$f\in D(A)$, and $h\in C_b(\R^N)$ be such that $f=R(1,A)h$. By
convolution, we can determine a sequence of smooth functions
$\{h_n\}\subset C^2_b(\R^N)$, bounded in $C_b(\R^N)$ and converging
locally uniformly to $h$ as $n\to +\infty$. By Lemma
\ref{lemma:commut}(ii) and (iii), the sequence $\{R(1,A)h_n\}$ is
contained in $C^2_b(\R^N)$ and it converges to $f$ locally uniformly
in $\R^N$. Further, arguing as in the proof of \eqref{int:byparts},
one can easily show that ${\mathscr A}R(1,A)h_n= -h_n+R(1,A)h_n$ for
any $n\in\N$. Hence, the sequence $\{{\mathscr A}R(1,A)h_n\}$ is
bounded in $C_b(\R^N)$ and it converges to $-h+f\in C_b(\R^N)$,
locally uniformly in $\R^N$. It follows that $f\in D(A)$.
\end{proof}

\section{Schauder estimates}
\label{sec-optimal} In this section we prove Schauder estimates for
the (distributional) solutions to the elliptic equation
\begin{equation}
\lambda u-{\mathscr A}u=f,\qq\;\,\lambda>0, \label{nonom:ellipt}
\end{equation}
\pn and to the non-homogeneous Cauchy problem
\begin{equation}
\left\{
\begin{array}{lll}
D_tu(t,x)={\mathscr A}u(t,x)+g(t,x),\q &t\in [0,\timeT], &x\in\R^N,\\[2mm]
u(0,x)=f(x), && x\in\R^N.
\end{array}
\right. \label{nonom:parab}
\end{equation}
Throughout the section, we assume that Hypotheses \ref{ipos-1} are
satisfied with $\kappa$ equal to the least common multiple of the
odd numbers between 1 and $2r+1$.

The main results of this section are collected in the following two
theorems.

\begin{theorem}
\label{main:1} Let $\theta\in ]0,1[$ and $\lambda>0$. Then, for any
$f\in C_b^{\theta}(\R^N)$ there exists a function $u\in {\mathscr
C}^{2+\theta}(\R^N)$ solving equation \eqref{nonom:ellipt} in the
sense of distributions. Moreover, there exists a positive constant
$C$, independent of $u$ and $f$, such that
\begin{equation}
\|u\|_{{\mathscr C}^{2+\theta}(\R^N)}\le
C\|f\|_{C^{\theta}_b(\R^N)}. \label{stima:ellipt:deg}
\end{equation}
\pn Such a function $u$ is the unique distributional solution to the
equation \eqref{nonom:ellipt} which is bounded and continuous in
$\R^N$ and it is twice continuously differentiable in $\R^N$ with
respect to the first $p_0$ variables, with bounded derivatives.
\end{theorem}

\begin{theorem}
\label{main:2} Let $\theta\in ]0,1[$, $\timeT>0$ and $f\in
C^{2+\theta}_b(\R^N)$ and $g\in C_b([0,\timeT]\times\R^N)$ be such
that $g(t,\cdot)\in C_b^{\theta}(\R^N)$ for any $t\in [0,\timeT]$,
and
\begin{equation*}
\sup_{t\in [0,\timeT]}\|g(t,\cdot)\|_{C^{\theta}_b(\R^N)}<+\infty.
\end{equation*}
Then, there exists a function $u\in C_b([0,\timeT]\times\R^N)$,
solution to problem \eqref{nonom:parab} in the sense of
distributions, such that $u(t,\cdot)\in {\mathscr
C}^{2+\theta}(\R^N)$ for any $t\in [0,\timeT]$ and
\begin{equation}
\sup_{t\in [0,\timeT]}\|u(t,\cdot)\|_{{\mathscr
C}^{2+\theta}(\R^N)}\le C\Bigl(\|f\|_{C^{2+\theta}_b(\R^N)}
+\sup_{t\in [0,\timeT]}\|g(t,\cdot)\|_{C^{\theta}_b(\R^N)}\Bigr),
\label{stima:parab:deg}
\end{equation}
\pn for some positive constant $C$, independent of $u,f,g$.
Moreover, $u$ is the unique distributional solution to problem
\eqref{nonom:parab} which is bounded and continuous in
$[0,\timeT]\times\R^N$, and there, it is twice continuously
differentiable with respect to the first $p_0$ spatial variables,
with bounded derivatives.
\end{theorem}

To begin with, we prove an interpolation result. We need to
introduce the auxiliary spaces $\tilde {\mathscr C}^{\theta}(\R^N)$
($\theta\in ]0,+\infty[$) that are defined analogously to
 the spaces ${\mathscr C}^{\theta}(\R^N)$, with the
H\"older spaces $C^{\theta/(2j+1)}(\R^{p_j})$ being replaced by the
Zygmund spaces ${\mathcal C}^{\theta/(2j+1)}(\R^{p_j})$
($j=0,\ldots,r$); see Definition \ref{spazi:holder}. It is clear
that $\tilde {\mathscr C}^{\theta}(\R^N)={\mathscr
C}^{\theta}(\R^N)$ if $\theta/(2j+1)\notin\N$ for any
$j=0,\ldots,r$.

\begin{prop}\label{lem:4.1}
Fix $\theta\in ]0,1[$ and $\beta\in [0,\kappa[$ such that
$\beta/(2j+1)\notin\N$ for any $j=0,\ldots,r$. Then,
\begin{equation}
({\mathscr C}^{\beta}(\R^N),{\mathscr
C}^{\kappa}(\R^N))_{\theta,\infty} =\tilde {\mathscr
C}^{(1-\theta)\beta+\kappa\theta}(\R^N), \label{interp:1}
\end{equation}
\pn with equivalence of the corresponding norms. Here, ${\mathscr
C}^0(\R^N)=C_b(\R^N)$.
\end{prop}

\begin{proof}
We first prove \eqref{interp:1} in the case when $\beta=0$. For this
purpose, we recall that, in \cite[Theorem 2.2]{L}, the author has
proved that, for any $\gamma>0$ and any $\theta\in ]0,1[$, the
topological equality $(BUC(\R^N),\tilde{\mathscr
C}^{\gamma}(\R^N))_{\theta,\infty} =\tilde {\mathscr
C}^{\gamma\theta}(\R^N)$ holds. Since $BUC(\R^N)$ belongs to both
the classes $J_0(C_b(\R^N),\tilde {\mathscr C}^{\gamma}(\R^N))$ and
$K_0(C_b(\R^N),\tilde {\mathscr C}^{\gamma}(\R^N))$, the Reiteration
Theorem (see, e.g., \cite[Theorem 1.2.15]{Lu1}) implies that
\begin{equation}
(C_b(\R^N),\tilde{\mathscr C}^{\gamma}(\R^N))_{\theta,\infty}
=\tilde {\mathscr C}^{\gamma\theta}(\R^N), \label{interp:2-lun}
\end{equation}
with equivalence of the corresponding norms.

Let us now fix $\gamma\in\R\setminus{\mathbb Q}$ such that
$\gamma>\kappa$. This choice of $\gamma$ implies that ${\mathscr
C}^{\gamma}(\R^N)=\tilde {\mathscr C}^{\gamma}(\R^N)$.
 Therefore, the formula \eqref{interp:2-lun} with $\theta=\kappa/\gamma$
 yields the equality
$(C_b(\R^N),{\mathscr
C}^{\gamma}(\R^N))_{\kappa/\gamma,\infty}=\tilde {\mathscr
C}^{\kappa}(\R^N)$ with equivalence of the corresponding norms.
Since ${\mathscr C}^{\kappa}(\R^N)\subset \tilde {\mathscr
C}^{\kappa}(\R^N)$ with a continuous embedding, we easily see that
${\mathscr C}^{\kappa}(\R^N)$ is continuously embedded in
$(C_b(\R^N),{\mathscr C}^{\gamma}(\R^N))_{\kappa/\gamma,\infty}$,
or, equivalently, ${\mathscr C}^{\kappa}(\R^N)$ belongs to the class
$K_{\kappa/\gamma}(C_b(\R^N),{\mathscr C}^{\gamma}(\R^N))$.

Let us prove that ${\mathscr C}^{\kappa}(\R^N)$ belongs also to the
class $J_{\kappa/\gamma}(C_b(\R^N),{\mathscr C}^{\gamma}(\R^N))$.
For this purpose, we recall that, there exists a positive constant
$C$ such that
\begin{equation}
\|\psi\|_{C^{\kappa/(2j+1)}(\R^{p_j})}\le
C\|\psi\|_{\infty}^{1-\frac{\kappa}{\gamma}}
\|\psi\|_{C^{\gamma/(2j+1)}(\R^{p_j})}^{\frac{\kappa}{\gamma}},
\label{spino}
\end{equation}
\pn for any $\psi\in C^{\gamma/(2j+1)}(\R^{p_j})$ and any
$j=0,\ldots,r$ (see e.g., \cite[Proposition 1.1.3(ii)]{Lu1}).

\noindent Fix $f\in {\mathscr C}^{\gamma}(\R^N)$ and $1\leq j\leq r$. By
applying \eqref{spino} to the function
$$\psi=f(x_0,\ldots,x_{j-1},\cdot,x_{j+1},\ldots,x_r)$$ (where we have
split $x\in\R^N$ as $x=(x_0,\ldots,x_r)$, with $x_i\in\R^{p_i}$
($i=0,\ldots,r$)) and then, by taking the supremum when we let the
variable $(x_0,\ldots,x_{j-1},x_{j+1},\ldots,x_r)$ run over
$\R^{N-p_j}$, we conclude that $\|f\|_{j,\kappa}\le
C\|f\|_{\infty}^{1-\kappa/\gamma} \|f\|_{j,\gamma}^{\kappa/\gamma}$
(see \eqref{fjtheta} for the definition of these seminorms), so
that, summing over $j=0,\ldots,r$, we get
\begin{equation*}
\|f\|_{{\mathscr C}^{\kappa}(\R^N)}\le
C\|f\|_{\infty}^{1-\frac{\kappa}{\gamma}} \|f\|_{{\mathscr
C}^{\gamma}(\R^N)}^{\frac{\kappa}{\gamma}},
\end{equation*}
that is ${\mathscr C}^{\kappa}(\R^N)$ belongs to the class
$J_{\kappa/\gamma}(C_b(\R^N),{\mathscr C}^{\gamma}(\R^N))$. Since
$C_b(\R^N)$ belongs to both  classes $J_0(C_b(\R^N),{\mathscr
C}^{\gamma}(\R^N))$ and $K_0(C_b(\R^N),{\mathscr
C}^{\gamma}(\R^N))$, the Reiteration Theorem yields now the equality
$(C_b(\R^N),{\mathscr C}^{\kappa}(\R^N))_{\theta,\infty}=
(C_b(\R^N),{\mathscr C}^{\gamma}(\R^N))_{\theta\kappa/\gamma,\infty}$
(with equivalence of the corresponding norms) that, combined with
\eqref{interp:2-lun}, yields \eqref{interp:1} with $\beta=0$.

The general case when $\beta\in ]0,\kappa[$ is such that
$\beta/(2j+1)\notin\N$ for any $j=0,\ldots,r$ now follows from the
interpolation theorem. Indeed,
\begin{align*}
({\mathscr C}^{\beta}(\R^N),{\mathscr
C}^{\kappa}(\R^N))_{\theta,\infty}&=
((C_b(\R^N),{\mathscr C}^{\kappa}(\R^N))_{\beta/\kappa,\infty},{\mathscr C}^{\kappa}(\R^N))_{\theta,\infty}\no\\
&=(C_b(\R^N),{\mathscr
C}^{\kappa}(\R^N))_{(1-\theta)\beta/\kappa+\theta,\infty}.
\end{align*}
\end{proof}

%%%%%%%%%%%%%%%%%%%%%%%%%%%%%%%%%%%%%%%%%%%%%%%%%%%%%%%%%%%%%%%%%%%%%%%%%%%%%%%%%%%%
%
%  ESTIMATES IN ANISOTROPIC SPACES
%
%%%%%%%%%%%%%%%%%%%%%%%%%%%%%%%%%%%%%%%%%%%%%%%%%%%%%%%%%%%%%%%%%%%%%%%%%%%%%%%%%%%%

\noindent The following  is a straightforward consequence of the
estimates in Theorem \ref{thm:constr}

\begin{lemma}
\label{lemma-trivial} For any $\omega>0$, there exists a positive
constant $C=C(\omega)$ such that
\begin{equation}
\|T(t)f\|_{{\mathscr C}^m(\R^N)}\le Ct^{-\frac{m}{2}}e^{-\omega
t}\|f\|_{C_b(\R^N)},\qq\;\,\mbox{holds for $t\in ]0,+\infty[$}.
\label{anisotr:2}
\end{equation}
\end{lemma}

Combining Theorem \ref{thm:constr} and Lemma \ref{lemma-trivial}, we
can now prove the following.

\begin{prop}
\label{thm:5.4} For any $\timeT>0$, $0<\beta\leq \theta<3$ with
$\beta,\theta\notin\N$, there exists a positive constant
$C=C(\timeT)$ such that, for any $f\in C^{\beta}_b(\R^N)$ the
following inequality holds:
\begin{equation}
\|T(t)f\|_{{\mathscr C}^{\theta}(\R^N)}\le
Ct^{-\frac{\theta-\beta}{2}}\|f\|_{C^{\beta}_b(\R^N)},\qq\;\,t\in
]0,+\infty[. \label{stima-isotr-anisotr}
\end{equation}
\end{prop}

\begin{proof}
The proof follows from an interpolation argument. To simplify the
notation, in the sequel we denote by $\omega$ any positive number
and by $C$ a positive constant, possibly depending on $\omega$ but
being independent of $t$ and $f$, which may vary from line to line.
By applying \cite[Proposition 1.2.6]{Lu1} with
$X_1=X_2=Y_1=C_b(\R^N)$, $Y_2={\mathscr C}^{\kappa}(\R^N)$, and by
taking estimate \eqref{anisotr:2}, Theorem \ref{thm:constr}(ii)
(which implies that $\{T(t)\}$ is a contractive semigroup) and
Proposition \ref{lem:4.1} into account, we obtain
\begin{equation}
\|T(t)\|_{{\mathscr L}(C_b(\R^N),{\mathscr C}^{\theta_1}(\R^N))}\le
Ct^{-\frac{\theta_1}{2}}e^{\omega t},\qq\;\, t\in ]0,+\infty[,
\label{anisotr:3}
\end{equation}
for any $\theta_1\in ]0,\kappa[$ such that $\theta_1$ is not
rational. Of course, \eqref{anisotr:3} holds also with $\theta_1=0$.
Using again \cite[Proposition 1.1.13]{Lu1} now with $X_1=C_b(\R^N)$,
$X_2=C_b^{\kappa}(\R^N)$, $Y_1=Y_2={\mathscr C}^{\kappa}(\R^N)$, we
get
\begin{equation}
\|T(t)\|_{{\mathscr L}(C_b^{\theta_2}(\R^N),{\mathscr
C}^{\kappa}(\R^N))}\le Ct^{-\frac{\kappa-\theta_2}{2}}e^{\omega
t},\qq\;\, t\in ]0,+\infty[, \label{anisotr:4}
\end{equation}
for any $\theta_2\in ]0,\kappa[$ such that $\theta_2$ is not integer
and even for $\theta_2=0,\kappa$. Finally, interpolating the
estimates \eqref{anisotr:3} and \eqref{anisotr:4}, we get
\begin{equation*}
\|T(t)\|_{{\mathscr L}(C_b^{\theta_3\theta_2}(\R^N),\tilde{\mathscr
C}^{(1-\theta_3)\theta_1+\kappa\theta_3}(\R^N))}\le
Ct^{-\frac{(1-\theta_3)\theta_1+(\kappa-\theta_2)\theta_3}{2}}e^{\omega
t},\qq\;\, t\in ]0,+\infty[,
\end{equation*}
and \eqref{stima-isotr-anisotr} follows by taking
$\theta_1,\theta_2,\theta_3$ such that $\theta_2\theta_3=\beta$ and
$(1-\theta_3)\theta_1+\kappa\theta_3=\theta$.
\end{proof}

The estimate \eqref{stima-isotr-anisotr} is the keystone in the proof of
Theorems \ref{main:1} and \ref{main:2}. The candidate to be the
solutions to the equation \eqref{nonom:ellipt} and the
non-homogeneous Cauchy problem \eqref{nonom:parab} are,
respectively, the functions $R(\lambda,A)f$ and $u$ defined by
\begin{equation}
u(t,x)=(T(t)f)(x)+\int_0^t(T(t-s)g(s,\cdot))(x)ds,\qq\;\,t\in
[0,\timeT],\;\,x\in\R^N. \label{variat-const}
\end{equation}

The results in the following proposition are now a straightforward
consequence of the estimate \eqref{stima-isotr-anisotr} and the interpolation arguments in
\cite[Section 3]{Lu-sem}. For this reason we skip the proof,
referring the reader to the quoted paper.
\begin{prop}
\label{prop:5.5} For fix $\theta\in ]0,1[$ and $\timeT>0$ the
following are true.
\begin{enumerate}[\rm (i)]
\item
For any $f\in C^{\theta}_b(\R^N)$, the function $R(\lambda,A)f$
belongs to ${\mathscr C}^{2+\theta}(\R^N)$ and the estimate
\eqref{stima:ellipt:deg} is satisfied by some positive constant $C$
independent of $f$.
\item
For any $f\in C^{2+\theta}_b(\R^N)$ and any function $g\in
C([0,\timeT]\times\R^N)$ such that $g(t,\cdot)\in
C^{\theta}_b(\R^N)$ for any $t\in [0,\timeT]$, with $\sup_{t\in
[0,\timeT]}\|g(t,\cdot)\|_{C^{\theta}_b(\R^N)}<+\infty$, the
function $u$ in \eqref{variat-const} is bounded and continuous in
$[0,\timeT]\times\R^N$. Moreover, $u(t,\cdot)\in {\mathscr
C}^{2+\theta}(\R^N)$ for any $t\in [0,\timeT]$ and estimate
\eqref{stima:parab:deg} is satisfied by some positive constant $C$
independent of $f$ and $g$.
\end{enumerate}
\end{prop}

We can now complete the proofs of Theorems \ref{main:1} and
\ref{main:2}.

\begin{proof}[Proof of Theorem $\ref{main:1}$] By Proposition
\ref{prop:char-dom}, we know that $A\psi={\mathscr A}\psi$ for any
$\psi$ in $D(A)$, where ${\mathscr A}\psi$ is meant in the sense of
distributions. Hence, the resolvent equality immediately implies
that the function $R(\lambda,A)f$ is a distributional solution of
the equation \eqref{nonom:ellipt}. Moreover, by Proposition
\ref{prop:5.5}(i), $R(\lambda,A)f\in {\mathscr C}^{2+\theta}(\R^N)$
and satisfies estimate \eqref{stima:ellipt:deg}. As a byproduct,
Proposition \ref{prop:maxprinc:0}(i) implies that $R(\lambda,A)f$ is
the unique distributional solution to the equation
\eqref{nonom:ellipt} satisfying the properties of Theorem
\ref{main:1}. The proof is now complete.
\end{proof}

\begin{proof}[Proof of Theorem $\ref{main:2}$] The uniqueness part of
the statement follows immediately from the maximum principle in
Proposition \ref{prop:maxprinc:0}(ii). Moreover, by virtue of
Proposition \ref{prop:5.5}, we can limit
ourselves to proving that the convolution term in
\eqref{variat-const}, that we simply denote by $v$, is a
distributional solution to \eqref{nonom:parab}, with $f\equiv 0$.
Actually for smooth  $g$ with compact support this is an easy and
classical argument using variation of constants.  For the general case we pick a
sequence $\{g_n\}\subset C^{1,2}_b([0,\timeT]\times\R^N)$, bounded
in the sup-norm, and converging locally uniformly in
$[0,\timeT]\times\R^N$ to $g$. Moreover, for any $n\in\N$, we denote
by $v_n$ the convolution function defined as $v$, but with $g$ being
replaced by $g_n$. As already indicated above, a straightforward
computation, based on estimate \eqref{stimaapriori:a0} with
$\|\alpha\|=2$ and $h=2$, shows that
$v_n$ is a classical solution to problem \eqref{nonom:parab} (with
$f\equiv 0$ and $g$ being replaced by $g_n$). Moreover, its sup-norm
may be bounded by a positive constant, independent of $n$ and, by
Proposition \ref{prop-conv-compatti}, $v_n$ converges to $v$
pointwise in $[0,\timeT]\times\R^N$.

Now, we observe that, for any smooth function $\va\in
C^{\infty}_c(]0,\timeT[\times\R^N)$, it holds that
\begin{align*}
\int\limits_{]0,\timeT[\times\R^N}\hskip-1em g_n\va\,
dtdx=\int\limits_{]0,\timeT[\times\R^N}\hskip-1em (D_tv_n-{\mathscr
A} v_n)\va\, dtdx =\int\limits_{]0,\timeT[\times\R^N}\hskip-1em
v_n(-D_t\va-{\mathscr A}^{*}\va)\,dtdx,
\end{align*}
\pn where ${\mathscr A}^{*}$ is the formal adjoint to operator
${\mathscr A}$. Letting $n\to +\infty$, we deduce that $v$ is a
distributional solution of \eqref{nonom:parab} with $f\equiv 0$.
\end{proof}

%%%%%%%%%%%%%%%%%%%%%%%%%%%%%%%%%%%%%%%%%%%%%%%%%%%%%%%%%%%%%%%%%%%%%%%
%%                                                                   %%
%%             APPENDIX                                              %%
%%                                                                   %%
%%%%%%%%%%%%%%%%%%%%%%%%%%%%%%%%%%%%%%%%%%%%%%%%%%%%%%%%%%%%%%%%%%%%%%%
\appendix
\section{Technical results}
\label{sec:tec}
\begin{lemma}
Suppose that ${\rm Ker}(Q(x))$ is independent of $x\in\R^N$. Then,
the following conditions are equivalent:
\begin{enumerate}[\rm(i)]
\label{lemma:linear-algebra}
\item
for any $x\in\R^N$, ${\rm Ker}(Q(x))$ does not contain non-trivial
subspaces which are $B^*$-invariant;
\item
for any $x\in\R^N$, let
$W(x)=\{\xi\in\R^N:Q(x)(B^*)^k\xi=0,\;\,k\in\N_0\}$. Then,
$W(x)=\{0\}$;
\item
for any $x\in\R^N$ and any $r\in\N$, let $W_r(x)=\{\xi\in\R^N:
Q(x)(B^*)^k\xi=0,\;\,k=0,\ldots,r-1\}$. Then, there exists $k_0\le
N$, independent of $x$, such that $W_{k_0}(x)=\{0\}$;
\item
the matrix $Q_t(x)=\int_0^te^{sB}Q(x)e^{sB^*}ds$ is positive
definite for any $t>0$ and any $x\in\R^N$;
\item
the rank of the matrix ${\mathscr
F}^{(r)}(x)=[Q(x),BQ(x),B^2Q(x),\ldots,B^rQ(x)]$ is $N$, for any
$x\in\R^N$ and some $r<N$, independent of $x$.
\end{enumerate}
\end{lemma}

\begin{proof}
We will show that $(i)\Leftrightarrow (ii)$,
$(ii)\Leftrightarrow (iii)$, $(ii)\Leftrightarrow (iv)$,
$(iii)\Leftrightarrow(v)$. We preliminarily note that both $W(x)$
and $W_r(x)$ are independent of $x$, so that, in the rest of the
proof, we simply write $W$ and $W_r$ instead of $W(x)$ and $W_r(x)$.

\noindent {\em $(i)\Leftrightarrow (ii)$:} To prove this equivalence,
it suffices to observe that, for any $x\in\R^N$, the set $W(x)$ is
contained in ${\rm Ker}(Q(x))$ and is its largest subspace, which is
invariant for $B^*$.

\noindent{\em  $(ii)\Leftrightarrow (iii)$:} Of course, we have only
to prove that $(ii)\Rightarrow (iii)$. So, let us suppose that
$W=\{0\}$. Since, $W_r\supset W_{r+1}$, then ${\rm dim}(W_r)\ge{\rm
dim}(W_{r+1})$ for any $r\in\N$. Further, ${\rm dim}(W_1)={\rm
dim}({\rm Ker}(Q(0)))$ is positive and strictly less then $N$, since
$Q(0)$ is a singular and not trivial matrix. It follows easily that
there exists $k_0\le N$ such that $W_{k_0}=W_{k_0+1}$. We claim that
$W_{k_0}=\{0\}$. Let $\xi\in W_{k_0}$. Then, $Q(0)(B^*)^{j}\xi=0$
for any $j=0,\ldots,k_0+1$. It follows that $B^*\xi\in W_{k_0}$ and,
consequently, $W_{k_0}$ is a $B^*$-invariant subspace of ${\rm
Ker}(Q(0))$. Therefore, $W_{k_0}\subset W=\{0\}$ and we are done.

\noindent{\em $(ii)\Leftrightarrow (iv)$:} Let us fix $t>0$,
$x\in\R^N$ and let $\xi\in\R^N$ be such that $\langle
Q_t(x)\xi,\xi\rangle=0$. This implies that $\langle
e^{sB}Q(x)e^{sB^*}\xi,\xi\rangle=0$ for any $s\in [0,t]$. Hence,
$Q(x)e^{sB^*}\xi=0$ for any $s$ as above. Since
\begin{equation*}
Q(x)e^{sB^*}\xi=\sum_{k=0}^{+\infty}\frac{s^k}{k!}Q(x)(B^*)^k\xi,\qq\;\,s\in
[0,t],
\end{equation*}
\pn $Q(x)e^{sB^*}\xi=0$ if and only if $Q(x)(B^*)^k\xi=0$ for any
$k\in\N_0$, that is if and only if $\xi\in W$. The equivalence
between (ii) and (iv) follows immediately.

{\noindent $(iii)\Leftrightarrow (v)$:} Let us fix $x\in\R^N$ and
denote by ${\mathscr F}^{(r)}_j(x)$ ($j=1,\ldots,N$) the rows of the
matrix ${\mathscr F}^{(r)}(x)$. Further, fix
$\xi_1,\ldots,\xi_N\in\R$ and set $\xi:=(\xi_1,\ldots,\xi_N)$. As it
is immediately checked, $\sum_{j=1}^N\xi_j{\mathscr F}_j^{(r)}(x)=0$
if and only if $\xi\in W_{r-1}$. Hence, the rows of the matrix
${\mathscr F}^{(r)}$ are linearly independent if and only if
$W_{r-1}=\{0\}$. From this, the equivalence between (iii) and (v)
clearly follows.
\end{proof}

The following lemma plays a crucial role in the proofs of Theorems
\ref{thm:3.2}.

\begin{lemma}
\label{lemma-crucial} Fix $l\ge 1$ and $m> c_{l-1,r}$. Then, for any
function $w\in C^{l+1}_b(\R^N)$ it holds that
\begin{equation}
[D^l_{\ell(m)},\langle B\cdot,D\rangle ]w=\sum_{k\in
A_m^{(l)}}{\mathscr J}_{k}^{(l)} D^l_kw, \label{series-B}
\end{equation}
\pn where the set $A_m^{(l)}$ is defined as follows: if
$d_{j_1},d_{j_2},\ldots,d_{j_k}$ $(1\le j_1<\ldots< j_k\le r)$
 are all the non-zero entries of
the vector $i_m^{(l)}=(0,d_1,\ldots,d_r)$, then
\begin{align}
A_m^{(l)}=\,&\Bigl\{s:
i_s^{(l)}=i_m^{(l)}-e_{j_1}^{(r+1)}+e_{j_1-1}^{(r+1)}-e_{j_i}^{(r+1)}+e_{h}^{(r+1)}\no\\
&\quad\quad\quad\quad\quad\quad\quad\quad\quad\quad\mbox{for some } i=2,\ldots, k \mbox{, and }h\le j_i+1\Bigr\}\no\\
&\cup
\left\{s: i_s^{(l)}=i^{(l)}_m-e_{j_1}^{(r+1)}+e_h^{(r+1)}\;\,\mbox{for some }h\le j_1\right\}\no\\
&\cup \Bigl\{s:
i_s^{(l)}=i_m^{(l)}-2e_{j_1}^{(r+1)}+e_{j_1-1}^{(r+1)}+e_h^{(r+1)}\no\\
&\quad\quad\quad\quad\quad\quad\quad\quad\quad\quad\mbox{for some
}h\le\min\{j_1+1,r\}, \mbox{ if }\alpha_{j_1}>1\Bigr\}, \label{Aml}
\end{align}
\pn where $e_h^{(r+1)}$ denotes the $h^\text{th}$ vector of the Euclidean
basis of $\R^{r+1}$. The entries of the matrices ${\mathscr
J}_k^{(l)}$ $(k\in A_m^{(l)})$ linearly depend only on the entries
of the matrix $B$. In particular, the matrix ${\mathscr
J}_{m}^{(l)}$ has full rank.
\end{lemma}

\begin{proof}
 By using the chain rule and by taking the structure of the matrix $B$ in
\eqref{matrix-B} into account, it is easy to see that for any
multi-index $\alpha\in\N_0^N$ we have
\begin{align}
([D^{\alpha},\langle B\cdot,D\rangle
]w)(x)&=\sum_{i,j=1}^N\sum_{{\beta\le\alpha}\atop{\|\beta\|=1}}
\binom{\alpha}{\beta}
b_{ij}D^{\beta}x_jD^{\alpha-\beta}D_iw(x)\no\\
&= \sum_{i,j=1}^N\sum_{s=0}^r\sum_{\tau\in
\II_s}\langle\alpha,e^{(N)}_{\tau}\rangle
b_{ij}D_{\tau}x_jD^{\alpha-e^{(N)}_{\tau}
+e^{(N)}_i}w(x)\no\\
&= \sum_{s=0}^r\sum_{\tau\in
\II_s}\sum_{h=0}^{\min\{s+1,r\}}\sum_{i\in
\II_h}\langle\alpha,e^{(N)}_{\tau}\rangle
b_{i\tau}D^{\alpha-e^{(N)}_{\tau}+e^{(N)}_i}w(x), \label{lemmaA.8-1}
\end{align}
\pn for any $x\in\R^N$. By definition  we
have
$i_{\ell(m)}=(0,\ldots,0,1,d_{j_1}-1,\ldots,d_{j_2},\ldots,d_{j_k},0,\ldots,0)$.
 In \eqref{lemmaA.8-1} consider all the possible multi-indices
$\alpha\in \N_0^N$ with $|\alpha|=i_{\ell(m)}$. We see immediately
that $[D^l_{\ell(m)},\langle B\cdot,D\rangle ]w$ is given by the
right-hand side of \eqref{series-B} for some matrices ${\mathscr
J}^{(l)}_k$ ($k\in A_m$) and $\mathscr{J}_m^{(l)}$. It remains to
show that the matrix ${\mathscr J}^{(l)}_m$ has full rank which
equals the number of its columns. We split the rest of the proof in
two steps. \med

{\em Step 1.} First, we show that we can make some reduction. More
precisely, we show that, without loss of generality, we can limit
ourselves to prove the assertion for a generic smooth function $w$
when:
\begin{enumerate}[\rm(i)]
\item
the only non-trivial blocks of the matrix $B$ in \eqref{matrix-B}
are $B_1\ldots,B_r$;
\item[(ii)]
$i_{\ell(m)}=(0,\dots,0,1, d_{j_1}-1,0,\dots,0)$.
\end{enumerate}
As a straightforward computation shows, the entries of the matrix
${\mathscr J}^{(l)}_m$ depend only on the matrices $B_1,\ldots,B_r$.
Hence, we can assume (i). This implies that formula
\eqref{lemmaA.8-1} can be rewritten as follows:
\begin{equation}
[D^{\alpha},\langle B\cdot,D\rangle ]w =
\sum_{s=0}^{r-1}\sum_{\tau\in \II_s}\sum_{i\in
\II_{s+1}}\langle\alpha,e^{(N)}_{\tau}\rangle
b_{i\tau}D^{\alpha-e^{(N)}_{\tau}+e^{(N)}_i}w.
\label{lemmaA.8-1-bis}
\end{equation}
But if we have here $|\alpha|=i_{\ell(m)}$, then the only
possibilities to obtain a multi-index
$\alpha-e^{(N)}_{\tau}+e^{(N)}_i$ having the block form $i_m$, are
exactly the choices $\tau\in \II_{j_{1}-1}$ and $i\in \II_{j_1}$. So
if we split
$$
i_{\ell(m)}=(0,\dots,0,1, d_{j_1}-1,0,\dots,0)+(0,\dots,0,0,0,d_{j_1+1},\dots),
$$
and, accordingly $\alpha=\beta+\gamma$ with $|\beta|\hs{1}=\hs{1}(0,\dots,0,1,
d_{j_1}-1,0,\dots,0)$ and
$|\gamma|\hs{1}=\hs{1}(0,\dots,0,0,0,d_{j_1+1},\dots)$, we see
\begin{eqnarray*}
[D^{\alpha},\langle B\cdot,D\rangle ]w =[D^{\beta}D^\gamma,\langle
B\cdot,D\rangle ]w =\sum_{\tau\in \II_{j_1-1}}\sum_{i\in
\II_{j_1}}\langle\alpha,e^{(N)}_{\tau}\rangle
b_{i\tau}D^{\beta-e^{(N)}_{\tau}+e^{(N)}_i}D^\gamma w+\cdots,
\end{eqnarray*}
where we haven't written out the terms, which do not contribute to
${\mathscr J}^{(l)}_m$. This means we that we can argue for the
function $D^\gamma w$ hence assuming (ii), and the general case will
follow, as well.

\med

{\em Step 2.} Let us take a derivative $D^{\alpha}$ in the block
$D^l_{\ell(m)}$. Then, there exists an index $\tau\in \II_{j_1-1}$
such that $\langle\alpha,e^{(N)}_\tau\rangle=1$. Taking formula
\eqref{lemmaA.8-1-bis} into account, it is immediate to see that
\begin{align}
[D^{\alpha},\langle B\cdot,D\rangle ]w&= \sum_{i\in
\II_{j_1}}b_{i\tau}D^{\alpha-e^{(N)}_\tau+e^{(N)}_i}w +\ldots\ldots\no\\
&= [B_{j_1}^*D^1_{j_1}D^{\alpha-e^{(N)}_\tau}w]_\tau+\ldots\ldots =
{\mathscr K}_{\alpha}(w)+\ldots\ldots, \label{formula}
\end{align}
where $[\,\cdot\, ]_\tau$ and  ``\ldots\ldots'' denote,
respectively, the $\tau^{\text{th}}$ component of the vector in
brackets and terms which depend on (some of) the $l^{\text{th}}$
derivatives of $w$ that are in a block different from $D^l_{m}$.
Finally, we recall that $D^1_{j_1}z$ denotes the vector of the first
order derivatives $D_hz$ of the function $z$, with $h\in \II_{j_1}$.
We are interested exclusively in ${\mathscr K}_{\alpha}(w)$, because
only this  term will contribute to $\mathscr{J}^{(l)}_m$. In the
following, we are going to reorder the vectors $D^l_{m}$ and
$D^l_{\ell(m)}$ in such a way that the assertion about the rank of
$\mathscr{J}^{(l)}_m$ will be clear.

\med\pn Order the set $\Gamma\subset\N_0^{p_{j_1}}$ of
multi-indices of length $d_{j_1}-1$ by $\prec$
anti-lexicographically. That is we have
\begin{align*}
(d_{j_1}-1,0,\ldots,0)\prec (d_{j_1}-2,1,\ldots,0)\prec\dots \prec
(0,\ldots,0,d_{j_1}-1).
\end{align*}
Next, we introduce the set
\begin{align*}
\Lambda:=\Bigl\{\lambda_{i,\gamma}:=(0,\ldots,0,e^{(p_{{j_1}-1})}_i,
\underbrace{\,\;\;\;\gamma\,\;\;\;}_{{j_1}^{\rm th}{\rm
~block}},0,\ldots,0):~i=1,\ldots,
p_{{j_1}-1},~\gamma\in\Gamma\Bigr\},
\end{align*}
which we order again anti-lexicographically, still denoting the ordering  by
$\prec$. The set $\Lambda$ describes the possible multi-indices
having block from $i^{(l)}_{\ell(m)}$. Reorder the vector
$D^l_{\ell(m)}w$ according to this ordering. Now pick $\gamma\in \Gamma$. By
considering multi-indices
\begin{equation*}
\gamma^{+i}:=(0,\ldots,0,\underbrace{\gamma+e^{(p_{j_1})}_i}_{{j_1}^{\rm
th}{\rm ~block}},0,\ldots,0),\quad i=1,\ldots, p_{j_1},
\end{equation*}
we recover \emph{all} the multi-indices of block form $i^{(l)}_m$,
but most of them even \emph{several} times. For a $\gamma\in
\N_0^{p_{j_1}}$ let $n(\gamma)$ denote the smallest non-negative integer
$n$ such that for all $n+1<k\leq p_{j_1}$ we have $\gamma_k=0$. We
have, for instance, $n((d_{j_1}-1,0,\ldots,0))=0$ (only in this case
is $n(\gamma)=0$), $n((d_{j_1}-1,1,0,\ldots,0))=1$ and
$n((0,\ldots,0,d_{j_1}-1))=p_{j_1}-1$. Consider now a multi-index
$\gamma\in \Gamma$ and all the multi-indices $\gamma^{+i}$,
$i=1,\dots,p_{j_1}$. Precisely for $i=1,\dots, n(\gamma)$ we obtain
multi-indices $\beta$ which can be written both as $\gamma^{+i}$ and $\gamma'^{+i'}$ for some
$\gamma'\prec \gamma$ and for some
$1\leq i'\leq p_{j_1}$. We set
$D_{m}^{\gamma}w:=(D^{\gamma^{+(n(\gamma)+1)}}w,D^{\gamma^{+(n(\gamma)+2)}}w,\ldots,
D^{\gamma^{+p_{j_1}}}w)^{\top}$. If $\gamma_1\prec\gamma_2 \prec \ldots$ is
an enumeration of $\Gamma$, we have now that
$(D_{m}^{\gamma_1}w,D_m^{\gamma_2}w,\dots)^{\top}$
is a reordering of
$D_m^lw$. Further we set
$\mathscr{H}_{\gamma}(w):=(\mathscr{K}_{\lambda_{1,\gamma}}(w),\mathscr{K}_{\lambda_{2,\gamma}}(w),
\ldots,\mathscr{K}_{\lambda_{p_{j_1}-1},\gamma}(w))^{\top}$.
 From \eqref{formula} we conclude that
\def\matrixstrut{\rule[-1em]{0pt}{2.5em}}
\begin{equation*}
\begin{pmatrix}
\matrixstrut \mathscr{H}_{\gamma_1}(w)\\
\matrixstrut \mathscr{H}_{\gamma_2}(w)\\
\matrixstrut\vdots\\
\matrixstrut \mathscr{H}_{\gamma_k}(w)\\
\matrixstrut\vdots\\
\end{pmatrix}=
\begin{pmatrix}
\matrixstrut[B_{j_1}^*]_{-n(\gamma_1)}&\ldots&\ldots&0&\ldots &0\\
\matrixstrut\star&[B_{j_1}^*]_{-n(\gamma_2)}&0&\ldots &\ldots &0\\
\matrixstrut\vdots&\hskip-2em\ddots&\hskip-2em\ddots&\hskip-2em\ddots&\\
\matrixstrut\star&\ldots&\star&[B_{j_1}^*]_{-n(\gamma_k)}&0\\
\matrixstrut\vdots&\ldots&\ldots&\ddots&\ddots&\ddots\\
\end{pmatrix}
\begin{pmatrix}
\matrixstrut D_{m_{}}^{\gamma_1}w\\
\matrixstrut D_{m_{}}^{\gamma_2}w\\
\matrixstrut\vdots\\
\matrixstrut D_{m_{}}^{\gamma_k}w\\
\matrixstrut\vdots\\
\end{pmatrix},
\end{equation*}
where $[B^*_{j_1}]_{-s}$ denotes the matrix obtained from
$B^*_{j_1}$ by dropping out the first $s$ columns. The block matrix
above is block-lower triangular has full rank, as all its blocks on the diagonal do so, and
its rank is equal to the number of its columns. Thus
$\mathscr{J}^{(l)}_m$ which is similar to the above block matrix,
has the asserted properties.
\end{proof}

The following two lemmas are used in the forthcoming proof of the
maximum principle of Proposition \ref{prop:maxprinc:0}.

\begin{lemma}\label{lemma:approx-B} For the first-order differential operator ${\mathscr B}$, formally defined
by the equality ${\mathscr B}u(x)=\langle Bx,Du(x)\rangle$ for any
$x\in\R^N$ and any $u\in C(\R^N)$, where $Du$ is meant in the sense
of distributions, the following hold:
\begin{enumerate}[\rm(i)]
\item
For any $u\in BUC(\R^N)$ such that ${\mathscr B}u\in C(\R^N)$, there
exists a sequence $\{u_n\}$ of smooth functions, converging to $u$
uniformly in $\R^N$, such that ${\mathscr B}u_n\in C_b(\R^N)$ for
any $n\in\N$ and it converges to ${\mathscr B}u$ locally uniformly
in $\R^N$. In particular, if $u$ is compactly supported in $\R^N$,
then $u_n$ is compactly supported in ${\rm supp}(u)+\ov{B(1)}$, for
any $n\in\N$.
\item
For any $u\in BUC([0,+\infty[\times\R^N)$ such that $D_tu-{\mathscr
B}u\in C(]0,+\infty[\times\R^N)$, where both $D_tu$ and $Du$ are
meant in the sense of distributions, there exists a sequence
$\{u_n\}$ of smooth functions, converging to $u$ uniformly in
$[0,+\infty[\times\R^N$, such that $D_tu_n-{\mathscr B}u_n\in
C_b([0,+\infty[\times\R^N)$ for any $n\in\N$ and it converges to
$D_tu-{\mathscr B}u$ locally uniformly in $]0,+\infty[\times\R^N$.
In particular, if $u$ is compactly supported in
$]0,+\infty[\times\R^N$, then ${\rm supp}(u_n)$ is compact and
contained in a compact set, which is independent of $n$.
\end{enumerate}
\end{lemma}

\begin{proof} (i) For any $n\in\N$, let $u_n=u*\varrho_n$, where $\varrho_n=n^N\varrho(n\cdot)$,
$\varrho\in C^{\infty}_c(B(1))$ being a positive function with
$\|\varrho\|_{L^1(\R^N)}$ and ``$*$'' denotes the convolution
operator. As it is immediately checked, the function $u_n$ is smooth
and converges to $u$ uniformly in $\R^N$. Moreover, if $u$ is
compactly supported in $\R^N$, then each function $u_n$ is compactly
supported in ${\rm supp}(u)+\ov{B(1)}$.

To prove that ${\mathscr B}u_n$ converges to ${\mathscr B}u$ locally
uniformly in $\R^N$, we observe that
\begin{equation}
{\mathscr B}u_n={\mathscr B}u*\varrho_n+{\Tr}(B)u_n+u*{\mathscr
B}\varrho_n. \label{calB-n}
\end{equation}
\pn This is enough for our aims. Indeed, as it is immediately seen,
$u*{\mathscr B}\varrho_n$ converges to ${\Tr}(B)u$ uniformly in
$\R^N$. It follows that the right-hand side of \eqref{calB-n} tends
to ${\mathscr B}u$ as $n\to +\infty$, locally uniformly in $\R^N$.

Formula \eqref{calB-n} is immediately checked in the particular case
when $u\in C^1_b(\R^N)$, by means of a straightforward computation,
based on an integration by parts. To prove it for any $u\in
C_b(\R^N)$, it suffices to write it with $u_n$ and $u$ being
replaced, respectively, by $u_n^m=v_m*\varrho_n$ and $v_m$, where
$\{v_m\}$ is a sequence of smooth functions converging to $u$
uniformly, and then take the pointwise limit as $m\to +\infty$.
Indeed, it is immediate to check that $u_n^m$, ${\mathscr B}u_n^m$
and $v_m*{\mathscr B}\varrho_n$ converge, respectively, to $u_n$,
${\mathscr B}u_n$ and $u*{\mathscr B}\varrho_n$, locally uniformly
in $\R^N$, as $m\to +\infty$. Moreover, since ${\mathscr B}u^m$
converges to ${\mathscr B}u$ in the sense of distributions, then
${\mathscr B}u^m*\varrho_n$ converges to ${\mathscr B}u*\varrho_n$
pointwise in $\R^N$ as $m\to+\infty$.

(ii) The proof is similar to the previous one. We extend $u$ to
$]-\infty,0[\times\R^N$, by setting $\tilde u(t,x)=u(-t,x)$ for such
$(t,x)$'s. Next, we approximate $\tilde u$ by the sequence $\{u_n\}$
defined by taking the convolution of $\tilde u$ with a standard
sequence $\{\varrho_n\}$ of mollifiers in $\R^{N+1}$. Using the same
approximation argument as in the proof of part (i), one can show
that $D_t\tilde u_n-{\mathscr B}\tilde u_n=(D_t u-{\mathscr
B}u)*\varrho_n-{\Tr}(B)u_n -u*{\mathscr B}\varrho_n$, in
$[a,+\infty)\times\R^N$ for any positive number $a$ such that
$na>1$. Letting $n\to +\infty$, it is easy to check that $D_t\tilde
u_n-{\mathscr B}\tilde u_n$ converges to $D_tu-{\mathscr B}u$
locally uniformly in $]0,+\infty[\times\R^N$.
\end{proof}

\begin{lemma}
\label{lemma-utile-uniqueness} The following hold true:
\begin{enumerate}[\rm(i)]
\item
Let $u\in C(\R^N)$ be such that ${\mathscr B}u\in C(\R^N)$, where
${\mathscr B}u$ is meant in the sense of distributions. If
$x_0\in\R^N$ is a maximum $($resp. minimum$)$ point of $u$, then
$({\mathscr B}u)(x_0)=0$.
\item
Let $u\in C(]0,\timeT[\times\R^N)$ be such that $D_tu-{\mathscr
B}u\in C(]0,\timeT[\times\R^N)$, where $D_tu$ and ${\mathscr B}u$
are meant in the sense of distributions. If $(t_0,x_0)\in
]0,\timeT[\times\R^N$ is a maximum $($resp. minimum$)$ point of $u$,
then $(D_tu-{\mathscr B}u)(t_0,x_0)=0$.
\end{enumerate}
\end{lemma}

\begin{proof}
(i) We adapt the proof of \cite[Proposition 3.1.10]{Lu1} to our
situation. Without loosing in generality we can assume that $x_0$ is
a maximum point of $u$ and $u(x_0)>0$. Indeed, if $x_0$ is a minimum
point, it suffices to replace the function $u$ by $-u$. Similarly,
if $x_0$ is a maximum point and $u(x_0)<0$, then the function
$u-2u(x_0)$ has at $x_0$ a positive maximum.

Let $R>0$ be such that $u(x)\le u(x_0)$ for any $x\in x_0+B(R)$.
Further, let $\vartheta\in C^{\infty}_c(x_0+B(R))$ satisfy
$\vartheta(x)<\vartheta(x_0)$ for any $x\in x_0+B(R)$ such that
$x\neq x_0$. As it is immediately seen, the function $v=u\vartheta$
is compactly supported in $x_0+B(R)$ and
 assumes its maximum value only at $x_0$. A straightforward computation shows that ${\mathscr
B}v\in BUC(\R^N)$. Let now $v_n$ be a sequence of smooth functions
compactly supported in $x_0+B(R+1)$, converging to $v$ uniformly in
$\R^N$ and such that ${\mathscr B}v_n$ converges to ${\mathscr B}v$
locally uniformly in $\R^N$, whose existence is guaranteed by Lemma
\ref{lemma:approx-B}(i). Without loss of generality, we can also
assume that $\sup_{\R^N} v_n>0$ for any $n\in\N$. Let
$\{x_n\}\subset x_0+B(R+1)$ be a sequence such that $\sup_{\R^N}
v_n=v_n(x_n)$ for any $n\in\N$. Up to a subsequence, we can assume
that $x_n$ converges in $\R^N$ to a maximum point of $v$. Hence, it
converges to $x_0$. To complete the proof, it suffices to observe
that $({\mathscr B}v_n)(x_n)=0$, for any $n\in\N$, and $({\mathscr
B}v)(x_0)=({\mathscr B}u)(x_0)=0$.

(ii) The proof can be obtained arguing as above, taking Lemma
\ref{lemma:approx-B}(ii) into account, and replacing the function
$\vartheta$, defined in (i), by a cut-off function $\psi\in
C^{\infty}_c(]0,\timeT[\times\R^N)$, compactly supported in
$[t_0-R^{-1},t_0+R^{-1}]\times x_0+B(R)$, for some $R>0$
sufficiently large, and such that $\psi(t,x)<\psi(t_0,x_0)=1$ for
any $(t,x)\in ]0,+\infty[\times\ x_0+B(R)$, with $(t,x)\neq
(t_0,x_0)$.
\end{proof}

We conclude this section with the proof of the maximum principle in
Proposition \ref{prop:maxprinc:0} and with the proof of Lemma
\ref{lem:propqh}.

\begin{proof}[Proof of Proposition $\ref{prop:maxprinc:0}$.] (i)
Let $\va(x)=1+|x|^2$ for any $x\in\R^N$, with $\lambda_0$
sufficiently large such that ${\mathscr A}\va-\lambda_0\va<0$ in
$\R^N$. The existence of such a $\lambda_0$ is guaranteed by our
assumptions on the growth of the coefficients of the operator
${\mathscr A}$ at infinity (see Remark \ref{rem:2.9}(i)).

Set $u_n=u-n^{-1}\va$. Suppose that $\lambda\ge\lambda_0$ and
$\lambda u-{\mathscr A}u=f$ for some $u$ and $f$ as in the statement
of the proposition. Then, $\lambda u_n-{\mathscr A}u_n< f$. Since
$u_n$ tends to $-\infty$ as $|x|\to +\infty$, then, for any
$n\in\N$, there exists $x_n\in\N$ such that $u_n(x_n)=\sup_{\R^N}
u_n$. Taking Lemma \ref{lemma-utile-uniqueness} into account, we can
easily show that ${\mathscr A}u_n(x_n)\le 0$. Hence, $\lambda
u_n(x_n)\le f(x_n)\le \|f\|_{\infty}$. Since
$\sup_{\R^N}u=\lim_{n\to +\infty}\sup_{\R^N}u_n$, it follows that
$\lambda v(x)\le \|f\|_{\infty}$ for any $x\in\R^N$. Applying the
same argument to $-u$ leads us to the assertion in the case when
$\lambda\ge \lambda_0$.

Finally, if $\lambda\in ]0,\lambda_0[$, we can rewrite the equation
$\lambda u-{\mathscr A}u=f$ as $\lambda_0u-{\mathscr A}u=g$, where
$g=f+(\lambda_0-\lambda)u$. Applying the estimate so far obtained
gives $\lambda_0\|u\|_{\infty}\le \|g\|_{\infty}\le
(\lambda_0-\lambda)\|u\|_{\infty}+\|f\|_{\infty}$, which leads us to
the assertion also in this situation. \med\pn (ii) The proof is
similar to the previous one. Suppose that $g\le 0$ in
$]0,\timeT[\times\R^N$ and introduce the function
$u_n:[0,\timeT]\times\R^N\to\R$ defined by
$u_n(t,x)=e^{-\lambda_0t}(u(t,x)-\sup_{\R^N}f)-n^{-1}\va(x)$, for
any $(t,x)\in [0,\timeT]\times\R^N$, where $\lambda_0$ and $\va$ are
as in the proof of (i). The function $u_n$ satisfies the equation
$D_tu_n-({\mathscr A}-\lambda_0)u_n< 0$ in $]0,\timeT[\times\R^N$
and $u_n(0,\cdot)\le 0$. Moreover, it attains its maximum value at
some point $(t_0,x_0)\in [0,\timeT]\times\R^N$. If $t_0=0$, then
$u_n(t,x)\le 0$ for any $(t,x)\in [0,\timeT]\times\R^N$. On the
other hand, if $t_0>0$ by elementary analysis and Lemma
\ref{lemma-utile-uniqueness}(ii), $(D_tu_n-{\mathscr
A}u_n)(t_0,x_0)\ge 0$. Therefore, $u_n\le 0$ in this case, as well.
Taking the limit as $n\to +\infty$ gives $e^{-\lambda_0
t}(u(t,x)-\sup_{\R^N}f)\le 0$ for any $(t,x)\in
[0,\timeT]\times\R^N$, and we are done.

To prove the assertion when $g\ge 0$, it suffices to apply this part
to $-u$. Finally, estimate \eqref{stimaapriori:1} follows
straightforwardly from these results.
\end{proof}

\begin{proof}[Proof of Lemma $\ref{lem:propqh}$]
\par\noindent (i)-(iii) Trivial from the definition.

\par\noindent (iv) Consider all the possible multi-indices $\beta$
which are of the form $\beta=\alpha-e^{(r+1)}_j+e^{(r+1)}_{j'}$ for
some $0\leq j,j'\leq r$ such that $\alpha_j>0$ and $j'\leq j+1$. We
get the largest value of $q_h(\beta)$, if actually $j'=j+1$ holds.
For this choice we have $q_h(\beta)=q_h(\alpha)+1$.

\par\noindent (v) Let $\alpha$ and $\beta=\alpha-e^{(r+1)}_{j_0}+e^{(r+1)}_{j_0-1}$
be as in the assertions. Since $\alpha_0=0$ and $\|\alpha\|>h$,
after dropping out $h$ ``derivatives'' from $\alpha$, starting from
the right, there will remain at least one positive entry which is
not at the $0^{\text{th}}$ position. This gives $q_h(\alpha)>1$.
Now, the equality $q_h(\beta)=q_h(\alpha)-1$ is clear from the
definition.

\par\noindent (vi) Observe that, by (v), $q_h(\hat\alpha)=q_h(\alpha)-1$
if $\|\alpha\|>h$. Now use (iii) to conclude $q_h(\beta)\leq
q_h(\hat\alpha)+1$ and finish the proof.

\par\noindent (vii) Let $\widetilde{\alpha}$ and $\alpha$ be as in
the assertion. By definition we have
$q_h(\beta)=q_h(\widetilde\alpha)+1$ if $\|\widetilde\alpha\|\geq
h$. If $\|\widetilde\alpha\|=h-1$, we have
$q_h(\widetilde\alpha)=0$, $\|\beta\|=h+1$ and $q_h(\beta)=1/2$. For
$\|\widetilde\alpha\|\leq h-2$ we have
$q_h(\widetilde\alpha)=q_h(\beta)=0$ (we have used (i)). So in all
cases we conclude $q_h(\widetilde\alpha)\geq q_h(\beta)-1$. The
inequality $q_h(\widetilde\alpha)\leq q_h(\alpha)$ is trivial, and
hence the proof is complete.
\end{proof}

%%%%%%%%%%%%%%%%%%%%%%%%%%%%%%%%%%%%%%%%%%%%%%%%%%%%%%%%%%%%%%%%%%%%%%%
%%                                                                   %%
%%             REFERENCES                                            %%
%%                                                                   %%
%%%%%%%%%%%%%%%%%%%%%%%%%%%%%%%%%%%%%%%%%%%%%%%%%%%%%%%%%%%%%%%%%%%%%%%

\parindent0pt
\end{document}